\documentclass[11pt]{article}

\usepackage[utf8]{inputenc}
\usepackage{amsmath, amssymb, amsthm, amsfonts}
\usepackage{enumitem}
\usepackage[numbers,sort]{natbib}
\usepackage{hyperref}
\hypersetup{colorlinks=false, pdfborder={0 0 0}}
\usepackage{graphicx}
\usepackage{xcolor}
\usepackage{comment}
\usepackage{dsfont}
\usepackage{mathtools}
\usepackage{bm}
\usepackage{geometry}
\usepackage{booktabs}
\usepackage{multirow}
\usepackage{longtable}
\usepackage{cancel}
\usepackage{subcaption}
\usepackage{IEEEtrantools}
\usepackage{placeins}   
\usepackage{pdflscape}  
\usepackage{float}      
\usepackage{setspace}
\usepackage{tikz}
\usepackage{xltabular}
\usepackage{tabularx,makecell,array,caption}
\usepackage{pdflscape}
\usepackage{adjustbox}
\usetikzlibrary{
  arrows.meta,
  positioning,
  calc,
  shapes.geometric,
  fit
}
\usetikzlibrary{
  automata,
  backgrounds
}
\tikzset{
  box/.style  ={draw,rounded corners,align=center,inner sep=4pt,minimum height=8mm,thick},
  enc/.style  ={box,fill=gray!10},
  pool/.style ={box,fill=green!12},
  head/.style ={box,fill=blue!12},
  crit/.style ={box,fill=orange!15},
  op/.style   ={box,fill=gray!6},
  arr/.style  ={-{Latex[length=3mm]}, ultra thick}
}
\usepackage{mathtools}
\usepackage{algorithm}
\usepackage{algpseudocode}
\DeclarePairedDelimiter{\norm}{\lVert}{\rVert}

\newtheorem{definition}{Definition}
\newtheorem{lemma}{Lemma}
\newtheorem{assumption}{Assumption}
\newtheorem{remark}{Remark}
\newtheorem{proposition}{Proposition}

\newtheorem{theorem}{Theorem}
\newtheorem{example}{Example}

\newcommand{\R}{\mathbb{R}}
\newcommand{\N}{\mathbb{N}}
\newcommand{\E}{\mathbb{E}}
\newcommand{\1}{\mathbf{1}}

\renewcommand{\read}{A_{\text{read}}}
\definecolor{darkgreen}{rgb}{0, .5, 0}

\begin{document}

\title{Reinforcement Learning for Global-Local Control of Dynamic Networks with Applications to Systemic Risk}
\author{Nils Detering\thanks{Heinrich Heine University D\"{u}sseldorf, Mathematics Institute.
  \texttt{nils.detering@hhu.de}}  \and Florian Grell\thanks{Heinrich Heine University D\"{u}sseldorf, Mathematics Institute.
  \texttt{florian.grell@hhu.de}}  \and Thilo Meyer-Brandis\thanks{University of Munich, Mathematics Institute, \texttt{meyerbra@math.lmu.de} }}
\date{\today}

\maketitle

\begin{abstract}
We propose a reinforcement learning framework for solving control problems on dynamic networks, where control decisions consist of a combination of graph-level (global) and node-level (local) actions. To exploit the underlying network structure in a policy-gradient algorithm, we introduce a message-passing neural network architecture with separate global and local policy heads. The architecture ensures that the resulting policies satisfy natural permutation symmetries and transparency constraints.

Our main theoretical contribution is an analysis of the robustness and regularity properties of the proposed architecture. We derive explicit Lipschitz bounds with respect to both parameter and input perturbations and use these estimates to establish convergence guarantees for a projected policy-gradient algorithm. The resulting constants reveal how architectural choices and graph characteristics influence sensitivity to input perturbations and the convergence behavior of gradient-based optimization methods. 

As an application, we study the management of systemic risk in financial networks, where a regulator allocates capital interventions across interconnected institutions. Numerical experiments demonstrate that the proposed reinforcement learning architecture outperforms standard architectures based on feedforward neural networks while the learned policies fulfill desirable symmetry and transparency constraints. The algorithm scales to large networks and effectively combines global budget decisions with local allocation policies by exploiting the underlying graph structure. 
\end{abstract}

\tableofcontents

\section{Introduction}
A wide range of complex systems, including financial systems, communication infrastructures, transportation networks, and social systems, can be modeled as weighted, directed networks whose nodes represent individual entities and whose edges encode their interactions. As these systems often evolve over time according to stochastic dynamics, a central planner may seek to steer their dynamics toward desirable outcomes through targeted interventions. However, controlling dynamic networks is particularly challenging due to their inherently high dimensional nature and because as a result they often admit a very large number of possible states and interventions. It was shown in a deterministic setting \cite{BarabasiNetworkControl} that sparse, heterogeneous networks, which arise in many large complex systems, require a particularly large fraction of driver nodes for structural controllability. 
Indeed, analytical solutions are mainly available for highly structured settings. In the related literature on stochastic games in connected systems, such results have been obtained for chain networks \cite{feng2020directedchain}, (graphon) mean-field interaction structures \cite{fouque2015meanfield,CCGL2021,ACL2022}, and vertex-transitive graphs \cite{hu2025finite1}, typically with linear-quadratic cost functions. For more general network control problems, one usually resorts to numerical methods to compute optimal or near-optimal control strategies \cite{preciado2014optimal,gama2021graph, hu2026finite2}. A recent stream of literature, in particular, focuses on reinforcement learning-based algorithms as a way to derive effective intervention policies or value functions \cite{meirom2021controlling,darvariu2024graph,gammelli2023graph}. 

Interventions in graph control problems may act at different levels. In the literature on Graph Neural Networks (GNNs), one often distinguishes between graph-level, edge-level, and node-level tasks \cite{zhou2020graph}. Following this terminology, we distinguish between graph-level, edge-level, and node-level decisions, using the terms action and decision interchangeably in this context. A graph-level decision means a single decision applies to the entire graph, whereas edge-level and node-level decisions specify decisions that apply to individual edges and nodes, respectively. These categories are not necessarily disjoint, since a control action may contain decisions at several levels. Throughout the paper, we use the terms global and local for graph-level and node-level decisions, respectively. Mixed decisions may be structured hierarchically, with one decision determining or restricting subsequent decisions. For example, Feng et al.~\cite{feng2025sequential} study sequential stochastic optimization problems by first determining the number of nodes to be selected at each time step and subsequently the corresponding nodes.

As graphs lead to large state spaces in the underlying Markov decision processes (MDPs), many applications from different fields focus on policy-based approaches, 
\cite{Lopez2025,deihim2024initial,hart2020graph}, to derive effective control policies. Tabular Q-learning requires action values to be estimated from repeated visits to the relevant state--action pairs and is therefore not suitable for such large state spaces \cite{WatkinsDayan1992,RL-SuttonBarto}. Deep Q-learning addresses this problem by approximating the action-value function with a neural network \cite{Mnih2015DQN}, but still requires the approximation and subsequent maximization of this function over the admissible actions. Policy-based methods instead directly parametrize and optimize the control policy based on trajectories generated by the current policy \cite{policy-gradient-methods}, thereby concentrating the updates on the regions of the state--action space visited under the current policy.

In this work, we propose a policy-based reinforcement learning approach for deriving effective control policies for the general control of dynamic networks when control policies may consists of decisions on both, a global graph-level and a local node-level, with graph-level and node-level only decision as special cases. Our approach ensures that control policies fulfill some natural fairness and transparency constraints. For example, graph-level decisions are always permutation-invariant, which ensures that they do not change under graph isomorphisms. This particularly ensures that a relabeling of nodes does not change the graph-level decision. Node-level decisions, by contrast, are permutation-equivariant, meaning that when nodes are reshuffled, the node-level decisions are reshuffled accordingly. In addition, policies can be enforced to depend only on the {\em k-hop} neighborhood of the nodes. This particularly ensures that two nodes with isomorphic {\em k-hop} neighborhoods are treated equally. It also makes the resulting decisions transparent and, for small $k$, allows them to be explained based only on local information, which is often available to the individual or institution represented by the node. In some instances, such constraints may actually be required by law, in other instances, it may support acceptance for certain control rules.

As graph-based data differs substantially from non-graph-based data, classical feedforward neural networks are often not very effective in processing graph data. GNNs instead offer a class of models that has been successfully used to process such data. See, for example, Kipf and Welling~\cite{kipf2017} for an influential early architecture and \cite{zhou2020graph,gravina2024benchmark} for more recent surveys and empirical comparisons. Accordingly, our policy-based reinforcement learning approach uses Message-Passing Neural Networks (MPNNs) to derive graph embeddings based on which control policies are then derived. Embeddings can either be shared or separate for the node-level and graph-level decisions. This architecture allows for full control of the available information that is used when processing node- and graph-level decisions. We formulate a general reinforcement learning framework for controlling dynamic networks that unifies graph-level, node-level, and joint control decisions under a common policy architecture. The framework allows permutation symmetries and locality constraints to be enforced by construction, yielding policies that satisfy natural fairness and transparency requirements. We analyze a projected policy-gradient algorithm for this setting and establish convergence guarantees for the general framework, including the important special cases of purely graph-level and purely node-level control.

Policy-based approaches are common in the reinforcement learning literature and substantial research has studied their convergence properties. Some results analyze tabular policies or optimization directly in policy space \cite{mei2020global,cen2022fast,xiao2022convergence,bhandari2024global,liu2025convergence}, whereas others consider parameterized policies under regularity assumptions such as boundedness and Lipschitz continuity of the policy score function \cite{Onthetheoryofpolicygradientmethods:optimalityapproximationanddistributionshift,Convergenceofactor-criticmethodswithmulti-layerneuralnetworks,OntheConvergenceofContinuousSingle-timescaleActor-critic}. For policies represented by feedforward neural networks, global Lipschitz continuity or boundedness of the derivatives with respect to the parameters generally holds only when the parameters are restricted to a bounded set \cite{Teichmann}. Similar regularity issues arise for graph-based policies. To overcome this problem and derive convergence results for our algorithm, we restrict the parameter space through projection. This allows us to combine robustness and regularity results that we derive for our architecture with generic results \cite{Borowski, Ghadimi2016} for projected stochastic-gradient methods to derive convergence guarantees for our algorithm. The resulting rates depend on both, graph specific properties such as the minimal and maximal degree as well as architecture specific properties, depth of the message-passing networks, activation function, and aggregation functions. Our results allows one to understand how the graph topology and the policy architecture influence the convergence rates. In addition our numerical experiments investigate the tradeoff between locality-based fairness and transparency on one hand, and control performance on the other. 

Our main contributions are summarized as follows:
\begin{itemize}
\item We formulate a class of infinite-horizon MDPs on dynamic graphs with state-dependent action spaces that combine graph-level and node-level decisions. Purely graph-level and purely node-level control problems are included as special cases.

\item We introduce a message-passing policy architecture consisting of a permutation-invariant global head and a shared permutation-equivariant local head. Adapting the idea of action branching \cite{tavakoli2018action} to policy-based reinforcement learning, the global and local decisions are produced jointly within a single factorized policy.

\item We provide an asymptotic convergence result, as well as a finite-time bound, for a projected version of the REINFORCE algorithm from \cite{REINFORCE} with geometric rollouts. To obtain these results, we first establish the required regularity properties of the proposed graph-based policy, including explicit parameter-side Lipschitz bounds for fixed graph inputs. 
Moreover, we derive input-side Lipschitz bounds with respect to perturbations of the augmented node features and the weighted adjacency matrix for fixed parameters, separating the effects of node-feature and edge-weight perturbations.

\item We apply the proposed framework to the control of systemic risk in financial networks and demonstrate numerically that the global--local policy can effectively combine global budget decisions with local allocation decisions and that the method scales to large graphs.
\end{itemize}

The application considered in this paper concerns the control of systemic risk in financial networks. Financial networks consist of interconnected institutions with bilateral obligations and are therefore commonly modeled as weighted directed graphs \cite{BEST04,UW04,NYYA07,CP14,ACP14}. The interconnectedness of the institutions creates channels through which financial distress can propagate. If one institution becomes unable to meet its liabilities, the resulting losses may spread across the network and potentially trigger further defaults \cite{glasserman2016contagion,acemoglu2015systemic}. In our model, the dynamics of the financial system are driven by exogenous shocks to asset prices and endogenous amplification effects arising from an extended Eisenberg--Noe clearing model \cite{EN01}. The interaction of these mechanisms may lead to complex nonlinear effects when institutions become distressed.

A related line of research considers interventions by a central regulator. These interventions include lender-of-last-resort policies \cite{CAPPONI2015152}, targeted capital injections, bailouts, and equity infusions \cite{minca2014optimal,amini2017optimal,Papachristou23}. 
These approaches seek optimal intervention strategies under different assumptions on the timing, available budget, and admissible allocation of interventions. 
Our approach is complementary to this literature. We construct a scalable graph-based policy that jointly determines the total intervention budget and its allocation across institutions, while enforcing locality and structural consistency through the policy architecture.

In the numerical experiments, we train the proposed policy architecture using a policy-gradient algorithm and evaluate its performance on a collection of small benchmark networks and larger randomly generated core--periphery networks.

The remainder of the paper is organized as follows. Section~\ref{sec:MDP} formulates the intervention problem as an MDP and introduces the global--local control framework. Section~\ref{sec:gnn-architecture} defines the graph-based architecture. Section~\ref{sec:robustness-convergence} contains the Lipschitz estimates and the convergence analysis of projected REINFORCE. Section~\ref{sec:numerical-results} introduces the application to systemic risk and presents the numerical experiments. The technical proofs are collected in the appendix.

\section{Markov Decision Process on Dynamic Graphs}
\label{sec:MDP}

We first introduce a general framework for a discrete-time, infinite-horizon MDP whose states are given by dynamic graphs. A graph is described by a fixed set of $n$ nodes, indexed by $[n]=\{1,\ldots,n\}$, each equipped with a feature vector, and a weighted adjacency matrix. In a dynamic graph, the set of nodes is considered fixed, while the node features and edge weights may vary with time. 
More precisely, let $V\subset\mathbb R^m$ and $W\subset\mathbb R_{\geq 0}$ be Borel subsets, equipped with their relative Borel $\sigma$-algebras, where $V$ is the node-feature space and $W$ is the set of possible scalar edge weights. For simplicity, we restrict attention to non-negative edge weights. The feature space $V$ may be a product of different component spaces, for instance $V=V_1\times\cdots\times V_m$, allowing different node features to have different value ranges. We assume $0\in W$ and exclude self-loops. We define the graph-valued state space, equipped with its relative Borel $\sigma$-algebra, by
\begin{equation}
\label{eq:graph-state-space}
\mathcal S := V^n \times \left\{ E\in W^{n\times n}: E_{ii}=0 \text{ for all } i\in[n] \right\}.
\end{equation}
An element $s=(X,E)\in\mathcal S$ consists of a node-feature matrix $X\in V^n$ and an edge-weight or adjacency matrix $E\in W^{n\times n}$. Global features can be included in this representation by appending them to the feature vector of each node. With an adaptation of the policy architecture introduced in Section~\ref{sec:gnn-architecture}, more general edge features could also be incorporated.

We consider a state-dependent mixed action space consisting of a graph-level decision and node-level decisions, which we refer to as the global and local decisions, respectively. Let \(\mathcal A_{\mathrm{glob}}\) and \(\mathcal A_{\mathrm{node}}\) be finite nonempty sets of admissible global decisions and local decisions, respectively. All finite action spaces are equipped with their discrete $\sigma$-algebras. We assume that there is a distinguished dummy node action $a_{\mathrm{node}}^0\in\mathcal A_{\mathrm{node}}$ and an activation map \(\eta:\mathcal A_{\mathrm{glob}}\to\{0,1\}\).
For each state $s\in\mathcal S$ and each node $i\in\{1,\ldots,n\}$, let \(\mathcal A_{\mathrm{node},i}(s)\subseteq\mathcal A_{\mathrm{node}}\) denote the set of node-level decisions that are admissible for node $i$ in state $s$, with $a_{\mathrm{node}}^0\in\mathcal A_{\mathrm{node},i}(s)$. We assume that $\{s\in\mathcal S:b\in\mathcal A_{\mathrm{node},i}(s)\}$ is a Borel subset of $\mathcal S$ for every node $i$ and every $b\in\mathcal A_{\mathrm{node}}$. 
The admissible action space in state $s$ is given by
\begin{equation}
\label{eq:general-mixed-action-space}
\mathcal A(s)= \Bigl\{(a_{\mathrm{glob}},a_{\mathrm{node}}^0,\ldots,a_{\mathrm{node}}^0):\eta(a_{\mathrm{glob}})=0\Bigr\} \cup \Bigl\{(a_{\mathrm{glob}},a_1,\ldots,a_n):\eta(a_{\mathrm{glob}})=1,\ a_i\in\mathcal A_{\mathrm{node},i}(s)\Bigr\}.
\end{equation}
Thus, an action is written as \(a=(a_{\mathrm{glob}},a_1,\ldots,a_n),\)
where $a_{\mathrm{glob}}$ is a global decision and $a_i$ is the local decision assigned to node $i$. If $\eta(a_{\mathrm{glob}})=0$, the global decision deactivates the local decisions and only the dummy node actions are selected. If $\eta(a_{\mathrm{glob}})=1$, the local decisions are active. The precise way in which the environment interprets such an action is model-specific. Purely node-level decision problems are obtained as a special case by choosing a singleton global action set with $\eta\equiv 1$, whereas purely graph-level decision problems are obtained by choosing $\eta\equiv 0$.

A MDP on \(\mathcal S\) is given by \((\mathcal S,\mathcal A,p,r,\gamma)\), where $\mathcal A=\bigcup_{s\in\mathcal S}\mathcal A(s)$ is the mixed action space with state-dependent admissible action sets $\mathcal A(s)$, $p(\cdot\mid s,a)$ is a Markov transition kernel, $r:\mathcal S\times\mathcal A\times\mathcal S\to\mathbb R$ is a Borel measurable deterministic one-step reward function and $\gamma\in(0,1)$ is the discount factor. We assume that $r$ is bounded on the set of admissible state-action-state triples by some constant $R_{\max}$.

A policy \(\pi=(\pi_t)_{t\in\mathbb N_0}\) is a sequence of decision rules such that, for each \(t\in\mathbb N_0\), \(\pi_t\) assigns to every history \(h_t=(s_0,a_0,\ldots,s_{t-1},a_{t-1},s_t)\) a probability distribution on the admissible actions in the current state, that is,
$$\pi_t(\cdot \mid h_t) \in \mathcal D(\mathcal A(s_t)),$$
where \(\mathcal D(\mathcal A(s_t))\) denotes the set of probability distributions on \(\mathcal A(s_t)\). 
For every $t\in\mathbb N_0$ and every fixed action $a\in\mathcal A$, we require the map $h_t\mapsto\pi_t(a\mid h_t)$ to be Borel measurable on the history space $(\mathcal S\times\mathcal A)^t\times\mathcal S$.
We call such a policy admissible. Thus, a decision rule specifies which action to choose given the observed history of states and actions.

For an initial distribution $\mu$ on $\mathcal S$ and a policy $\pi$, we denote by $\mathbb P^\pi_\mu$ the probability measure on trajectories induced by
$$S_0\sim\mu,\qquad A_t\sim\pi_t(\cdot\mid H_t),\qquad S_{t+1}\sim p(\cdot\mid S_t,A_t),$$
where $H_t=(S_0,A_0,\ldots,S_{t-1},A_{t-1},S_t)$ denotes the random history up to time $t$. The corresponding expectation is denoted by $\mathbb E^\pi_\mu$. If the initial state is deterministic, i.e. $S_0=s_0$, we write $\mathbb P^\pi_{s_0}$ and $\mathbb E^\pi_{s_0}$ instead. Throughout the paper, lowercase letters $s,a,s'$ denote fixed state and action values, whereas $S_t,A_t$ denote the corresponding random variables along a trajectory. The corresponding one-step reward process is defined by \(R_t:=r(S_t,A_t,S_{t+1})\).

For a fixed admissible policy $\pi$, the corresponding state-value function is
\begin{align}
\label{eq:state-value-function}
V^\pi(s):=\mathbb E_s^\pi\left[\sum_{t=0}^\infty \gamma^t R_t\right], \qquad s\in\mathcal S.
\end{align}
The state-action value function is defined by
\begin{align}
\label{eq:state-action-value-function}
Q^\pi(s,a):=\mathbb E_{s,a}^\pi\left[\sum_{t=0}^\infty \gamma^t R_t\right],
\qquad s\in\mathcal S,\ a\in\mathcal A(s),
\end{align}
where $\mathbb E_{s,a}^\pi$ denotes expectation under the dynamics with $S_0=s$, $A_0=a$, and policy $\pi$ used from time $1$ onward. The objective is to maximize the expected discounted return under the initial distribution $\mu$, that is,
\begin{align}
\label{eq:general-objective}
J_\mu(\pi) := \mathbb E_\mu^\pi\left[\sum_{t=0}^\infty \gamma^t R_t\right]
= \mathbb E_\mu\left[V^\pi(S_0)\right].
\end{align}
Under standard regularity assumptions, for instance when $\mathcal S$ is Polish, the admissible action sets are finite, and the one-step rewards are bounded, the discounted unrestricted control problem over all admissible policies admits an optimal stationary deterministic policy, see \cite{MDP-Puterman}. We therefore restrict attention to parametrized, stationary policies, which are introduced in Section~\ref{sec:gnn-architecture}.

For a stationary policy, the action distribution depends only on the current state and we write \(\pi(\cdot\mid s)\in\mathcal D(\mathcal A(s))\).
For a parametric stationary policy $\pi_\theta$, we write
$$V_\theta:=V^{\pi_\theta},\qquad Q_\theta:=Q^{\pi_\theta},\qquad J_\mu(\theta):=J_\mu(\pi_\theta).$$
If the initial distribution $\mu$ is fixed, we suppress it from the notation. If the initial state is deterministic, i.e. $\mu=\delta_{s_0}$, we also write
$J_{s_0}(\theta):=V_\theta(s_0).$

In Subsection~\ref{subsec:financial-network-model}, we apply the above framework to an intervention problem in a financial network.
\section{Global-Local Policy Architecture} \label{sec:gnn-architecture}

In this section, we introduce the parametrized policies used for the graph-valued MDP formulation from Section~\ref{sec:MDP}. As the numbering of the nodes in a graph is arbitrary, renumbering the nodes should not influence the graph-level decision, while the node-level decisions should change accordingly. Hence the policy should respect the symmetries of the underlying graph. More precisely, the graph-level decision should be permutation invariant, whereas the node-level decision rule should be permutation equivariant.

\begin{definition}\label{def:permutation-equivariance-invariance}
Let $m_1,m_2\in\N$. A function $f_{\mathrm{inv}}:\R^{n \times m_1} \times \R^{n \times n} \to \R^{m_2}$
is called permutation invariant if
$$f_{\mathrm{inv}}(PX, PEP^\top) = f_{\mathrm{inv}}(X,E)$$
for all permutation matrices $P \in \{0,1\}^{n \times n}$.

A function $f_{\mathrm{eq}}:\R^{n \times m_1} \times \R^{n \times n} \to \R^{n \times m_2}$
is called permutation equivariant if
$$f_{\mathrm{eq}}(PX, PEP^\top) = P\,f_{\mathrm{eq}}(X,E).$$
\end{definition}

Many GNNs, and in particular the MPNNs introduced below, offer a natural way to parametrize such functions. They share weights across nodes and aggregate information over local neighborhoods, thereby enforcing the desired permutation symmetries by construction. Before defining our MPNN layers, we introduce the required notation.

Given a state $s=(X,E)\in\mathcal S$, the policy network uses the node-feature matrix $X$ and the weighted adjacency matrix $E$. We identify $X\in V^n$ with its matrix representation in $\R^{n\times m}$ and write $x^{(0)}:=X.$
For a node $i$ in the graph associated with $s=(X,E)$, we write $e_{ij}:=E_{ij}.$
The sets of incoming and outgoing neighbors of node $i$ are given by
$$\mathcal{N}^+(i) :=\{j \in [n] \mid e_{ji} > 0 \} \quad \text{and}\quad \mathcal{N}^-(i) :=\{ j \in [n] \mid e_{ij} > 0 \}.$$

In our MPNN layers, we explicitly include and distinguish between incoming and outgoing edges. As shown in \cite{WL-directed} for directed unweighted graphs, considering both directions separately can increase the expressivity of the MPNN. Empirically, this distinction is especially useful in heterogeneous networks, where the incoming and outgoing neighbors of a node may differ substantially in their features.

The following MPNN layer is a specific instance of the general message-passing framework of \cite{Gilmer}. We denote the output dimension of the node features in the $k$-th layer by $d_k$.

\begin{definition}[MPNN layer]\label{def:MPNN-layer}
Let $\bigoplus$ denote a permutation invariant aggregation function, such as $\operatorname{sum}$ or $\operatorname{mean}$. Further let $W^{(k)} \in \R^{d_k \times 3d_{k-1}}$ and $b^{(k)} \in \R^{d_k}$ be learnable parameters given by a feedforward neural network layer and $\sigma$ an activation function. Given edge weights $e_{ij}$ from the underlying graph, a message-passing neural network layer updates the node embeddings $x^{(k-1)}=(x_1^{(k-1)},\ldots,x_n^{(k-1)}) \in \R^{n d_{k-1}}$ for node $i$ through
\begin{align*}
x_i^{(k)} = \sigma \odot \left( W^{(k)} \left[x_i^{(k-1)} \,\,\bigoplus_{j \in \mathcal{N}^+(i)} e_{ji}x_j^{(k-1)} \,\,\bigoplus_{j \in \mathcal{N}^-(i)} e_{ij} x_j^{(k-1)}\right] + b^{(k)} \right) \in \R^{d_k},
\end{align*}
where we concatenate vectors in square brackets and $\odot$ is the component-wise application.
\end{definition}

\begin{remark}\label{rem:MPNN-definition}
For simplicity, we use the same aggregation operator $\bigoplus$ for incoming and outgoing neighbors. Using different aggregators for the two directions would be straightforward. Moreover, we only consider edge-weighted messages of the form $e_{ij}x_j^{(k-1)}$. For sum or mean aggregation, the map
$$x^{(k-1)} \mapsto \left[x_i^{(k-1)} \,\, \bigoplus_{j \in \mathcal{N}^+(i)} e_{ji} x_j^{(k-1)}\,\,  \bigoplus_{j \in \mathcal{N}^-(i)} e_{ij} x_j^{(k-1)}\right]$$
can be represented by a single linear operator, which significantly simplifies the proof of the Lipschitz bounds. Our approach can be adapted to more general messages depending linearly on the edge weights. This is briefly discussed in Appendix~\ref{app:mpnn-matrix}.
\end{remark}

\begin{remark}\label{rem:fairness}
For any given node, a $K$-layer MPNN updates its features using the information contained in its $K$-hop subgraph, which includes all nodes and edges within a distance of $K$. For small values of $K$, 
this locality property brings two major advantages. First, it makes policies based on MPNN architectures transparent, since the information that influences the node-level decision is restricted to a clearly identifiable neighborhood of the node. 
Second, it introduces a natural form of locality-based fairness. Nodes within the same graph with identical $K$-hop subgraphs receive the same node embeddings. This is particularly relevant in regulatory settings, where interventions should be explainable and justifiable to the affected institutions. 
\end{remark}

Our MPNN formulation corresponds to a deliberately simple class of GNNs. While more expressive architectures exist, such as attention-based message-passing \cite{GAT, GATv2}, permutation-equivariant neural networks (PENNs) and their extensions \cite{gonon2026universality}, or graph transformer models \cite{GraphTransformer1, GraphTransformer2}, we focus on this setting because it allows for an explicit and more interpretable Lipschitz analysis. Similar restrictions are common in theoretical work on generalization and stability of graph neural networks \cite{StableFairGraphRepresentationLearningWithLipschitzConstraint, Towardsaunifiedframeworkforfairandstablegraphrepresentationlearning, DBLP:conf/iclr/LiaoUZ21,pmlr-v206-ju23a}.

A $K$-layer MPNN encoder is obtained by successively applying $K$ MPNN layers to the initial node features and edge weights. The resulting node embeddings are then used to parametrize the policy.
Motivated by the mixed structure of the action space~\eqref{eq:general-mixed-action-space}, we use a factorized policy architecture consisting of a graph-level, or global, head and a shared node-level, or local, head on top of the MPNN outputs. The global head produces the distribution of the graph-level decision, while the local head is applied with shared parameters to all node embeddings. This yields a policy class that is permutation invariant on the graph level and permutation equivariant on the node level, while the number of trainable parameters remains independent of $n$.

More precisely, given a state $s=(X,E)\in\mathcal S$, we first compute node embeddings
\begin{align*}
x^{(K)}=f^{(K)}(\theta^K,X,E)
\end{align*}
by a $K$-layer MPNN encoder, where $\theta^K$ denotes the parameters of the encoder. The global head uses an invariant readout to produce a distribution on $\mathcal A_{\mathrm{glob}}$, whereas the shared local head maps each node embedding $x_i^{(K)}$ to a distribution on the corresponding state-dependent admissible node actions.

For readability, the following definition is kept at an informal level and focuses on the main architectural idea. A fully formal definition of the policy is deferred to Appendix~\ref{app:policy-heads}. For now, we denote the set of all trainable parameters by $\theta$.

\begin{definition}[Policy heads]\label{def:heads}
Let $s\in\mathcal S$ be a state and let $x=(x_1,\ldots,x_n)\in \R^{n d_K}$ denote the corresponding node embeddings produced by the MPNN encoder. Let $\operatorname{READ}^{G}: \R^{n d_K} \to \R^{m_1}$ be a permutation invariant readout function. Let $W^G,W^N$ and $b^G,b^N$ be learnable parameters given by feedforward neural network layers of suitable dimensions.

The global head is defined by
$${\pi}_\theta^{\mathrm{glob}}(\cdot\mid s) := \operatorname{softmax}\left(W^G \operatorname{READ}^{G}(x) + b^G\right)\in \mathcal{D}(\mathcal{A}_{\mathrm{glob}}).$$

For each node $i\in[n]$, the local head is defined by
$${\pi}_\theta^i(\cdot\mid s) := \operatorname{softmax} \left(\operatorname{mask}_{s,i}(W^N x_i + b^N)\right) \in \mathcal{D}(\mathcal{A}_{\mathrm{node},i}(s)),$$
where $\operatorname{mask}_{s,i}$ assigns zero probability to node-level actions that are not admissible for node $i$ in state $s$. Hence, the node-level policy satisfies the state-dependent feasibility constraint by construction.

The complete policy $\pi_\theta$, which factorizes into the graph-level decision and the node-level decisions, is given by ${\pi}_\theta(\cdot\mid s)\in\mathcal{D}(\mathcal{A}(s))$ with
\begin{equation}
\pi_\theta(a \mid s) 
= \pi_\theta^{\mathrm{glob}}(a_{\mathrm{glob}} \mid s)
\Biggl(\prod_{i=1}^n \pi_\theta^{i}(a_i \mid s)\Biggr)^{\eta(a_{\mathrm{glob}})},
\qquad
a=(a_{\mathrm{glob}},a_1,\ldots,a_n)\in\mathcal{A}(s).
\label{eq:policy-def}
\end{equation}
\end{definition}

Let
$$\mathcal A_{\mathrm{glob}}^0:=\{a_{\mathrm{glob}}\in\mathcal A_{\mathrm{glob}}:\eta(a_{\mathrm{glob}})=0\},
\qquad
\mathcal A_{\mathrm{glob}}^1:=\{a_{\mathrm{glob}}\in\mathcal A_{\mathrm{glob}}:\eta(a_{\mathrm{glob}})=1\}.$$
The above policy is a probability distribution on the admissible action space, since
\begin{align*}
\sum_{a \in \mathcal{A}(s)} \pi_\theta (a \mid s)
=& \sum_{a_{\mathrm{glob}}\in\mathcal A_{\mathrm{glob}}^0} \pi_\theta^{\mathrm{glob}}(a_{\mathrm{glob}}\mid s)
+ \sum_{a_{\mathrm{glob}}\in\mathcal A_{\mathrm{glob}}^1} \pi_\theta^{\mathrm{glob}}(a_{\mathrm{glob}}\mid s)\prod_{i=1}^n \sum_{a_i\in\mathcal A_{\mathrm{node},i}(s)}\pi_\theta^{i}(a_i \mid s) \\
=& \sum_{a_{\mathrm{glob}}\in\mathcal A_{\mathrm{glob}}} \pi_\theta^{\mathrm{glob}}(a_{\mathrm{glob}}\mid s)=1.
\end{align*}

The logarithmic form of the policy is therefore given by
\begin{align*}
\log \pi_\theta(a \mid s)
= \log \pi_\theta^{\mathrm{glob}}(a_{\mathrm{glob}} \mid s) + \eta(a_{\mathrm{glob}}) \sum_{i=1}^n\log \pi_\theta^{i}(a_i \mid s).
\end{align*}
Thus, the score function is
\begin{align}
g_{(s,a)}(\theta) := \nabla_\theta\log \pi_\theta(a \mid s)
= \nabla_\theta\log \pi_\theta^{\mathrm{glob}}(a_{\mathrm{glob}} \mid s)+\eta(a_{\mathrm{glob}})\sum_{i=1}^n\nabla_\theta \log \pi_\theta^{i}(a_i \mid s).
\label{eq:logpolicy}
\end{align}

\begin{remark}[Gradient flow under inactive node-level actions]
\label{rem:null-budget-gradient}
The policy in \eqref{eq:policy-def} is chosen so that, whenever $\eta(a_{\mathrm{glob}})=0$, the realized action is independent of the node-level outputs. Consequently, the local head receives no gradient signal in this case. We do not impose the converse restriction. In particular, it may happen that the realized effect of an action does not depend on the graph-level decision for certain node-level choices, while the global head may still receive gradient updates. Excluding such cases would require a more involved renormalized policy representation, which would substantially complicate the Lipschitz analysis in Section~\ref{sec:robustness-convergence}. We therefore retain the simpler factorized parametrization.
\end{remark}

The above construction allows the different heads of the policy architecture to either share an encoder or utilize separate ones. Figure~\ref{fig:policy-architecture} visualizes the corresponding architectures, both of which are evaluated in Section~\ref{sec:numerical-results}. Common choices for the readout functions correspond to those used for aggregation, with sum or mean readout being particularly popular.

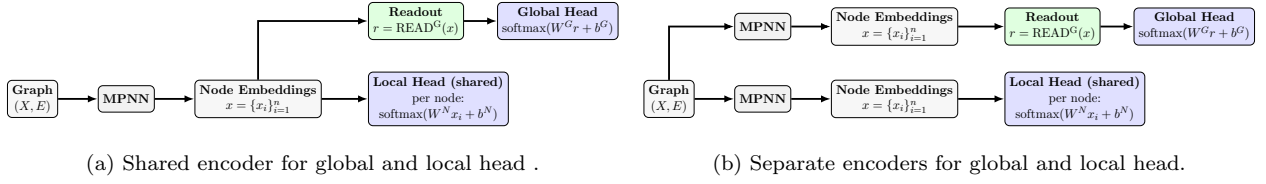
\begin{figure}[t]
  \centering
  \noindent
  \begin{minipage}[t]{0.49\textwidth}
    \centering
    \resizebox{\linewidth}{!}{%
      \begin{tikzpicture}[node distance=7mm and 12mm]
        \node[op]   (in)   {\textbf{Graph} \\ $(X,E)$};
        \node[enc, right=12mm of in] (MPNN)  {\textbf{MPNN}};
        \node[op,  right=12mm of MPNN] (outputMPNN)   {\textbf{Node Embeddings} \\ $x=\{x_i\}_{i=1}^{n}$};

        \draw[arr] (in) -- (MPNN);
        \draw[arr] (MPNN) -- (outputMPNN);

        \node[pool, above right=12mm and 14mm of outputMPNN] (Readout_G) {\textbf{Readout} \\ $r=\operatorname{READ^G}(x)$};
        \node[head, right=10mm of Readout_G] (global) {\textbf{Global Head} \\ $\operatorname{softmax}(W^G r + b^G)$};
        \draw[arr] (outputMPNN) |- (Readout_G);
        \draw[arr] (Readout_G) -- (global);

        \node[head, right=14mm of outputMPNN, minimum width=40mm] (nodehead)
          {\textbf{Local Head (shared)} \\ per node: \\ $\operatorname{softmax}(W^Nx_i+b^N)$};
        \draw[arr] (outputMPNN) -- (nodehead);
      \end{tikzpicture}%
    }

    \vspace{1mm}
    {\scriptsize (a) Shared encoder for global and local head .}
  \end{minipage}\hfill%
  \begin{minipage}[t]{0.49\textwidth}
    \centering
    \resizebox{\linewidth}{!}{%
      \begin{tikzpicture}[node distance=7mm and 12mm]
        \node[op]   (in)   {\textbf{Graph} \\ $(X,E)$};
        \node[enc, above right=12mm and 12mm of in] (MPNN)  {\textbf{MPNN}};
        \node[op,  right=12mm of MPNN] (outputMPNN)   {\textbf{Node Embeddings} \\ $x=\{x_i\}_{i=1}^{n}$};

        \draw[arr] (in) |- (MPNN);
        \draw[arr] (MPNN) -- (outputMPNN);

        \node[enc, right=12mm of in] (MPNN2) {\textbf{MPNN}};
        \node[op,  right=12mm of MPNN2] (outputMPNN2) {\textbf{Node Embeddings} \\ $x=\{x_i\}_{i=1}^{n}$};        

        \node[pool, right=14mm of outputMPNN] (Readout_G) {\textbf{Readout} \\ $r=\operatorname{READ^G}(x)$};
        \node[head, right=10mm of Readout_G] (global) {\textbf{Global Head} \\ $\operatorname{softmax}(W^G r+b^G)$};
        \draw[arr] (outputMPNN) -- (Readout_G);
        \draw[arr] (Readout_G) -- (global);

        \node[head, right=14mm of outputMPNN2, minimum width=40mm] (nodehead)
          {\textbf{Local Head (shared)} \\ per node: \\ $\operatorname{softmax}(W^N x_i+b^N)$};
        \draw[arr] (outputMPNN2) -- (nodehead);

        \draw[arr] (in) -- (MPNN2);
        \draw[arr] (MPNN2) -- (outputMPNN2);

      \end{tikzpicture}%
    }

    \vspace{1mm}
    {\scriptsize (b) Separate encoders for global and local head.}
  \end{minipage}
  \caption{Architectures built from a message-passing encoder and invariant readouts.}
  \label{fig:policy-architecture}
\end{figure}
\section{Robustness and Convergence Analysis}
\label{sec:robustness-convergence}
In this section, we first analyze the sensitivity of the MPNN encoder and the induced policy to parameter and input perturbations. These bounds then provide the regularity estimates used in our convergence analysis below. We state the resulting Lipschitz estimates only at the level of their asymptotic order. Up to architecture-dependent multiplicative constants, the same orders are obtained for a shared encoder architecture and for architectures with separate encoders for the global and local components, provided that the separate encoders have the same architecture. The proofs of all claims, as well as more precise constants, are given in the appendix.

\subsection{Architecture Dependent Lipschitz Bounds}
\label{subsec:Lipschitz}
This subsection establishes two types of Lipschitz estimates for the MPNN-based policy architecture. The first type concerns perturbations of the trainable parameters and is the key regularity input for the convergence analysis in Subsection~\ref{subsec:REINFORCE}. In particular, we derive bounds for the MPNN encoder, the logarithmic policy, and the corresponding score function as functions of the network parameters. The second type concerns perturbations of the graph input. These estimates are not needed for the projected policy-gradient argument, but provide complementary robustness information and make explicit how the architecture reacts to changes in node features and edge weights.

There is already a substantial literature on Lipschitz and stability properties of graph neural networks with respect to perturbations of the graph input, see e.g.~\cite{StableFairGraphRepresentationLearningWithLipschitzConstraint, Towardsaunifiedframeworkforfairandstablegraphrepresentationlearning, AligningRelationalLearningwithLipschitzFairness, Lipschitznormalizationforselfattentionlayerswithapplicationtographneuralnetworks}. Our input-side estimates are related to this line of work, but are tailored to the directed, edge-weighted MPNN architecture considered here. In particular, they distinguish between perturbations of the node features and perturbations of the weighted adjacency matrix, and they allow us to compare the graph-dependent constants arising under sum and mean aggregation.

However, for the projected policy-gradient analysis, we need regularity estimates for the logarithmic policy and the score function as functions of the trainable parameters. Parameter-perturbation analyses for graph neural networks, such as \cite{DBLP:conf/iclr/LiaoUZ21, pmlr-v206-ju23a}, study how changes in the trainable parameters affect the output of a graph neural network. However, these works do not provide explicit Lipschitz bounds for the full MPNN viewed as a function of its parameters in the form required for our analysis, nor do they consider Lipschitz bounds for the derivative of the MPNN with respect to its parameters. To the best of our knowledge, explicit bounds of this kind have not yet been established for MPNNs, while related estimates are available for standard neural networks, for example in \cite{Low-RankPlusSparseDecompositionofCovarianceMatricesUsingNeuralNetworkParametrization, Teichmann}. Our analysis follows a similar strategy, adapted to the message-passing structure and to the actor architecture introduced in Section~\ref{sec:gnn-architecture}.

All spaces considered throughout can be viewed as subsets of Euclidean spaces. Accordingly, $\norm*{\cdot}$ denotes the Euclidean norm on vectors and the induced operator norm on matrices, while $\norm*{\cdot}_\infty$ denotes the uniform norm.

\begin{definition}[Lipschitz continuous]\label{def:Lipschitz}
Let $\mathcal{X}\subset \R^l$ and $\mathcal{Y}\subset \R^m$. A function $f:\mathcal{X}\to \mathcal{Y}$ is called Lipschitz continuous with constant $L$ if, for all $x,y \in \mathcal{X}$,
$\norm*{f(x)-f(y)} \leq L \norm*{x-y}. $
\end{definition} 

With a slight abuse of notation, we denote by $f_{\max}\geq 0$ a finite entrywise bound on the initial node features $x^{(0)}$, i.e., $|X_{i,j}|\leq f_{\max}$ for all $i=1,\ldots,n$ and $j=1,\ldots,d_0.$
For a fixed state, $f_{\max}$ always exists. For statements on a set of states, the bound is, if possible, chosen uniformly over that set. 

The next assumption is required in order to derive global Lipschitz bounds.

\begin{assumption}[Bounded parameters]
\label{assum:bounds}
We assume that there exist constants $\beta_1,\ldots, \beta_K, \beta_G, \beta_N > 0$ such that 
\begin{align*}
\norm*{W^{(i)}} \leq \beta_i \text{ for all } i \in \{1,\ldots,K\}, \quad \norm*{W^G} \leq \beta_G, \quad \norm*{W^N} \leq \beta_N.   
\end{align*} 
Here, $W^{(i)}$ denotes the weight matrix in layer $i$, and $W^G,W^N$ denote the weights of the global and shared local head, respectively. In particular, this means that the parameter space is assumed to be a closed subset of a Euclidean space.

Moreover, we assume that the activation function $\sigma$ is twice differentiable with bounded derivatives. Define 
\begin{align*}
s_0:=\max\left\{\norm*{\sigma}_\infty^2,f_{\max}^2\right\},\qquad s_1:=\norm*{\sigma'}_\infty^2, \qquad s_2:=\norm*{\sigma''}_\infty^2.
\end{align*}
For notational convenience, we further assume that $s_0 d_k\geq 1$ for all $k\in\{0,\ldots,K\}$.
\end{assumption}

Assuming bounded state and parameter spaces, it suffices for the activation function and its derivatives to be bounded over the compact range of pre-activation values generated by the network. Thus, the global boundedness assumption of the activation function serves primarily to avoid carrying additional constants. Further, we introduced $s_0$ such that the node features satisfy $\norm*{x^{(k)}}^2\leq s_0nd_k$ for $k=0,\ldots,K.$

In the special case of a graph consisting of a single node, a feedforward neural network is a special case of an MPNN. Hence, the necessity of parameter restrictions is analogous to the feedforward setting. In particular, \cite{Teichmann} shows that such restrictions are unavoidable for global Lipschitz estimates. This also becomes apparent in Example~\ref{ex:lower-bounds} and further highlights that global Lipschitz bounds generally require a restricted parameter space.

A key driver of all subsequent bounds is the choice of the aggregation function. Both $\operatorname{sum}$ and $\operatorname{mean}$ aggregation are linear in the node features and can therefore be represented as multiplication by a matrix whose operator norm can be bounded.

\begin{proposition}[Operator norm of the aggregation function]\label{prop:operatornorm-aggregate}
Let $A^{(k)}_{\operatorname{mean}}(E)$ and $A^{(k)}_{\sum}(E)$ denote the linear representations of the aggregation functions in layer $k$ when choosing $\operatorname{mean}$ or $\operatorname{sum}$, respectively. We denote the maximum in-degree of the graph by $\Delta^+$ and the minimal positive in-degree by $\delta^+$. Similarly, we define the maximum out-degree $\Delta^-$ and the minimal positive out-degree $\delta^-$. For $E=0$, both aggregation operators have operator norm one. Otherwise, let $e_{\min}$ and $e_{\max}$ denote the minimal and maximal positive edge weights. Then the operator norms are independent of the node-feature dimension and satisfy
\begin{align*}
1+e^2_{\min}\max\left\{\frac{\Delta^-}{{\Delta^+}^2}, \frac{1}{\delta^+}, \frac{\Delta^+}{{\Delta^-}^2}, \frac{1}{\delta^-}\right\}\leq\norm*{A_{\operatorname{mean}}^{(k)}(E)}^2 \leq 1 + e^2_{\max}\left(\frac{\Delta^+}{\delta^-} + \frac{\Delta^-}{\delta^+}\right),   
\end{align*} 
and
\begin{align*}
1+e_{\min}^2 \max\{\Delta^+, \Delta^-\} \leq \norm*{{A}_{\sum}^{(k)}(E)}^2 \leq 1 + 2e_{\max}^2 \Delta^+ \Delta^-.
\end{align*}
\end{proposition}

The aggregation operator's norm depends only indirectly on the number of nodes, through the degree quantities and edge weights of the underlying graph. Sum aggregation leads to a stronger dependence on maximal degrees, whereas mean aggregation additionally depends on minimal positive degrees. Hence, in terms of these norm estimates, mean aggregation is especially advantageous for graphs with homogeneous node degrees. As shown in Proposition~\ref{prop:lipschitz-aggregation-map} in Appendix~\ref{app:aggregation-bounds}, mean aggregation also leads to smaller Lipschitz bounds with respect to perturbations of the weighted adjacency matrix $E$. However, as we observe in Section~\ref{sec:numerical-results}, these theoretical advantages do not directly translate into uniformly better empirical performance. This may be due to the lower expressive power of mean aggregation compared with sum aggregation, as discussed in \cite{WL-undirected2}. 

Other common choices for the aggregation function, such as $\min$ and $\max$, are less straightforward to compare for two reasons. First, they cannot be represented as multiplication by a single fixed matrix depending only on $E$, which leads to longer proofs and upper bounds of a different form. Second, they are only almost everywhere differentiable and thus require additional technical formalism. For these reasons, we omit them and henceforth assume that all layers use the same matrix-representable aggregation function. Although the corresponding matrices $A^{(k)}(E)$ act on feature spaces of different dimensions, their operator norms coincide. We therefore write $\norm*{A(E)}$ for this common graph-dependent aggregation norm.
For the asymptotic estimates in this subsection, a common upper bound on $\norm*{A(E)}$ is held fixed. For sum and mean aggregation, such an aggregation bound follows, for example, from uniformly bounded maximal edge weights and maximal in- and out-degrees by Proposition~\ref{prop:operatornorm-aggregate}.

The next theorem provides Lipschitz bounds for a $K$-layer MPNN encoder and its parameter derivative.

\begin{theorem}[Lipschitz bounds for the MPNN up to layer $K$]\label{theo:explicit-bounds}
Let Assumption~\ref{assum:bounds} hold, and suppose that all layers use the same matrix-representable aggregation function. Fix a state $s=(X,E)\in\mathcal S$. For $K\in\N$, set
\begin{align*}
d := \max\{d_0,\ldots,d_K\}, \qquad \beta_{\leq K}:=\max\{\beta_1,\ldots,\beta_K\}, \qquad q:=s_1\norm*{A(E)}^2\beta_{\leq K}^2,
\end{align*}
and
\begin{align*}
\alpha := s_1 n \left(s_0\norm*{A(E)}^2 d +1\right).
\end{align*}
Then, for all admissible parameter vectors $\theta^K,\bar\theta^K$, there exists a constant
$B_{\leq K}$ such that
\[
\left\|
f^K(\theta^K,X,E)-f^K(\bar\theta^K,X,E)
\right\|
\leq
B_{\leq K}\left\|\theta^K-\bar\theta^K\right\|,
\qquad
B_{\leq K}^2\leq \alpha S_K(q),
\]
where
\[
S_K(q):=
\begin{cases}
\dfrac{q^K-1}{q-1}, & q\neq 1,\\[1ex]
K, & q=1.
\end{cases}
\]

Moreover, the derivative of the $K$-layer MPNN with respect to the parameters is Lipschitz continuous, i.e., there exists a constant $L_{\leq K}$ such that
\[
\left\|
D_{\theta^K}f^K(\theta^K,X,E)
-
D_{\theta^K}f^K(\bar\theta^K,X,E)
\right\|
\leq
L_{\leq K}\left\|\theta^K-\bar\theta^K\right\|,
\qquad
L_{\leq K}^2
=
\mathcal O\!\left(n^2d^2G_K(q)\right),
\]
where
\[
G_K(q):=
\begin{cases}
q^{2K}, & q>3,\\
K3^{2K}, & q=3,\\
3^Kq^K, & \frac13<q<3,\\
K, & q=\frac13,\\
1, & q<\frac13.
\end{cases}
\]
The $\mathcal{O}$-notation is with respect to $n$, $d$ and $K$, with $q$ fixed.
\end{theorem}

Using uniform bounds up to layer $K$, the depth dependence of the squared Lipschitz bound for the MPNN is governed by the geometric sum $S_K(q)$ over $q=s_1\norm*{A(E)}^2\beta_{\leq K}^2.$
For fixed $q>1$, the bound grows geometrically with base $q$. At $q=1$, the sum grows linearly in $K$, whereas for $q<1$ it remains bounded uniformly in $K$. These regimes reflect how the layer-wise estimates control the propagation of perturbations originating in earlier layers.

This further motivates the analysis of the aggregation norm in Proposition~\ref{prop:operatornorm-aggregate}. The norm $\norm*{A(E)}$ is independent of the node-feature dimension, but depends on structural properties of the graph through its degrees and edge weights. It enters both the depth-dependent factor $q$ and the prefactor $\alpha=s_1n\left(s_0\norm*{A(E)}^2d+1\right),$ while further powers of the aggregation norm are absorbed into the constants in the $\mathcal O$-notation. Consequently, the graph structure can affect not only the magnitude of the bound but also its exponential dependence on the depth. 

Similarly, the squared Lipschitz bound for the parameter derivative is governed by $G_K(q)$. For fixed $q>3$, its exponential base is $q^2$. With the activation bounds, weight bounds, and aggregation norm held fixed, the corresponding prefactor is of order $n^2d^2$.
Although the estimates in Theorem~\ref{theo:explicit-bounds} are not sharp, the following simple example shows that both exponential depth rates are attained in the worst case for fixed $q>3$.

\begin{example}[Lower bounds]
\label{ex:lower-bounds}
Consider a graph consisting of a single node without any edges and with one-dimensional node features. Let $K\geq 2$, set $x^{(0)}=0$ and choose $\sigma=\tanh$. Fix $\beta>\sqrt{3}$ and let the weights and biases in every MPNN layer be given by
\begin{align*}
W^{(k)} &:= \begin{pmatrix}\beta & 0 & 0\end{pmatrix}
\in\R^{1\times 3}, \qquad k=1,\ldots,K,\\
b^{(k)} &:=
\begin{cases}
0, & k=1,\ldots,K-1,\\
x_0, & k=K,
\end{cases}
\end{align*}
where $x_0:=\operatorname{arctanh}(1/\sqrt{3})$. Then $q=\beta^2>3$ and the Jacobian of the MPNN satisfies
\begin{align*}
\norm*{D_{\theta^K}f^K(\theta^K,x^{(0)},E)}^2
=\frac{4}{9}S_K(q)
=\frac{4}{9}\frac{q^K-1}{q-1}.
\end{align*}
Moreover its Lipschitz constant L can be bounded by
\begin{align*}
s_2q^{2(K-1)} \leq L^2 = \mathcal{O}(q^{2K}),
\end{align*}
where the upper bound follows from Theorem~\ref{theo:explicit-bounds}.
\end{example}
Letting $\beta$ tend to infinity shows that restrictions on the layer weights are necessary in general to obtain global bounds on the Lipschitz constants. The detailed analysis of this example is provided in Appendix~\ref{app:mpnn-parameter-bounds}.

We next consider Lipschitz continuity with respect to the graph input.

\begin{theorem}[Input Lipschitz bound for the MPNN]\label{theo:input-lipschitz-GNN}
Let $\mathcal D\subset\mathcal S$ be a subset of the state space such that, for every $k\in\{1,\ldots,K\}$, there exist constants $a,c\geq 0$ with
\begin{align*}
\norm*{A^{(k)}(E)} \leq a, \qquad 
\norm*{A^{(k)}(E)-A^{(k)}(\bar E)} \leq c\norm*{E-\bar E}
\end{align*}
for all $s=(X,E),\bar s=(\bar X,\bar E)\in\mathcal D$. Let $f_{\max}$ be a uniform entrywise bound on the initial node features over $\mathcal{D}$ and fix parameters satisfying Assumption~\ref{assum:bounds}. Further, define
\begin{align*}
\left(q_{\leq K}^{\operatorname{in}}\right)^2 := s_1a^2 \beta_{\leq K}^2, 
\quad \left(\alpha_{\leq K}^{\operatorname{in}}\right)^2 &:= s_0s_1 \beta_{\leq K}^2 c^2 nd. 
\end{align*}
Then, for all $s=(X,E),\bar s=(\bar X,\bar E)\in\mathcal D$, the $K$-layer MPNN satisfies
\begin{align*}
\norm*{f^K(\theta^K,X,E)-f^K(\theta^K,\bar X, \bar E)} 
\leq \left(q_{\leq K}^{\operatorname{in}}\right)^K \norm*{X-\bar X}_F 
+ \alpha_{\leq K}^{\operatorname{in}} S_K(q_{\leq K}^{\operatorname{in}}) \norm*{E-\bar E}.
\end{align*}
\end{theorem}

For sum aggregation, the map $E\mapsto A^{(k)}(E)$ is globally Lipschitz continuous with $c=\sqrt{2}$ by Proposition~\ref{prop:lipschitz-aggregation-map} in Appendix~\ref{app:aggregation-bounds}. If the weighted adjacency matrices are uniformly bounded on $\mathcal D$, Proposition~\ref{prop:operatornorm-aggregate} further yields a uniform bound $a$ on the aggregation norm.  For mean aggregation, the corresponding bounds hold on domains with uniformly bounded weighted adjacency matrices and fixed edge support.

The previous theorem controls the sensitivity of the MPNN encoder with respect to perturbations of the graph input. The following theorem transfers this estimate to the policy. Since state-dependent action masks can change discontinuously when the admissible action sets changes, we state the result on subsets of the state space on which these sets are fixed.

\begin{theorem}[Input Lipschitz bound for the policy]\label{theo:input-lipschitz-policy}
Let the assumptions of Theorem~\ref{theo:input-lipschitz-GNN} hold on $\mathcal D\subset\mathcal S$, and consider the policy with summation readout and a $K$-layer MPNN encoder. Suppose further that the admissible node-level action sets are fixed on $\mathcal D$. Then, for every action $a=(a_{\mathrm{glob}},a_1,\ldots,a_n)$ admissible on $\mathcal D$, the policy satisfies
\begin{align*}
\left|\pi_\theta(a\mid s)-\pi_\theta(a\mid \bar s)\right| \leq \frac{\sqrt n}{2} \left( \beta_G+\eta(a_{\mathrm{glob}})\beta_N \right) \left( \left(q_{\leq K}^{\operatorname{in}}\right)^K\norm*{X-\bar X}_F
+ \alpha_{\leq K}^{\operatorname{in}}S_K(q_{\leq K}^{\operatorname{in}})\norm*{E-\bar E} \right)
\end{align*}
for all $s=(X,E),\bar s=(\bar X,\bar E)\in\mathcal D$.
\end{theorem}

We now transfer the parameter-side MPNN bounds to the logarithmic policy and the score function.

\begin{theorem}[Explicit bounds for the score function]\label{theo:score-explicit}
Let the assumptions of Theorem~\ref{theo:explicit-bounds} hold, and consider the logarithmic policy with summation readout and a $K$-layer MPNN encoder. Then there exists a constant $B_g$ such that, for all admissible parameter vectors
$\theta,\bar\theta$,
\[ \left| \log\pi_\theta(a\mid s) - \log\pi_{\bar\theta}(a\mid s) \right|
\leq B_g\|\theta-\bar\theta\|, \]
where
\[ B_g^2= \begin{cases}
\mathcal O\!\left(n^3 d\,S_K(q)\right),
    & \eta(a_{\mathrm{glob}})=1,\\[1mm]
\mathcal O\!\left(n^2 d\,S_K(q)\right),
    & \eta(a_{\mathrm{glob}})=0.
\end{cases}
\]

Moreover, there exists a constant $L_g$ such that the score function is Lipschitz continuous,
\[\left\|\nabla_\theta\log\pi_\theta(a\mid s)-\nabla_\theta\log\pi_{\bar\theta}(a\mid s)\right\|
\leq L_g\|\theta-\bar\theta\|, \]
where
\[ L_g^2=
\begin{cases}
\mathcal O\!\left(n^4 d^2 G_K(q)\right),& \eta(a_{\mathrm{glob}})=1,\\[1mm]
\mathcal O\!\left(n^3 d^2 G_K(q)+n^4 d^2 S_K(q)^2\right), & \eta(a_{\mathrm{glob}})=0.
\end{cases}
\]

The $\mathcal{O}$-notation is with respect to $n$, $d$ and $K$, with $q$ fixed.
\end{theorem}

More precise constants for the global and shared local head are stated in Proposition~\ref{prop:lipschitz-score} in Appendix~\ref{app:score-function-bounds}.

\begin{remark}[Interpretation of the prefactors]\label{rem:theorem-score}
The distinction between the cases $\eta(a_{\mathrm{glob}})=1$ and $\eta(a_{\mathrm{glob}})=0$ indicates the different contributions of the global and local policy heads. When the local decisions are inactive, all local-head contributions vanish, yielding the improved bounds stated in Theorem~\ref{theo:score-explicit}.

The contribution of the global head can be further reduced by replacing the summation readout with a mean readout. In that case, the squared norm of the Jacobian of the readout function is $\tfrac{1}{n}$ instead of $n$, which yields
$$B_g^2 =\mathcal{O}(dS_K(q)) \quad \text{and} \quad L_g^2=\mathcal{O}(nd^2G_K(q))$$
for the global head alone. Alternative readouts such as $\min$ and $\max$ could also be used for the global task, but we do not analyze them here since they lead to the same technical difficulties as the corresponding aggregation functions.
\end{remark}

\subsection{Convergence of Projected REINFORCE}
\label{subsec:REINFORCE}

The parameter-side bounds from Subsection~\ref{subsec:Lipschitz} provide the regularity estimates needed to analyze policy-gradient updates for the MPNN-based policy introduced in Section~\ref{sec:gnn-architecture}. Such regularity assumptions are standard in convergence analyses for policy-gradient and actor-critic methods, building on the seminal work of \cite{Actor-CriticAlgorithms}.
For instance, \cite{Onthetheoryofpolicygradientmethods:optimalityapproximationanddistributionshift} assumes Lipschitz continuity of $\nabla \log \pi$ to establish agnostic convergence for natural policy gradient methods under function approximation. Similarly, \cite{Ageneralsamplecomplexityanalysisofvanillapolicygradient} proves convergence of policy-gradient methods under weaker assumptions than Lipschitz continuity, but points out that for multi-layer neural networks one may need to restrict the parameter space. Related regularity assumptions also appear in more recent actor-critic analyses: \cite{Convergenceofactor-criticmethodswithmulti-layerneuralnetworks} assumes Lipschitz continuity of the policy and the logarithmic policy, \cite{OntheConvergenceofContinuousSingle-timescaleActor-critic} assumes boundedness and Lipschitz continuity of $\nabla\log\pi$, and \cite{closingthegap} proves last-iterate convergence under stronger regularity assumptions.

In our setting, these properties are not imposed abstractly. Instead, once the parameters are restricted to a compact set, they follow from the explicit architecture-dependent estimates derived in Subsection~\ref{subsec:Lipschitz}. Since the corresponding constants enter convergence guarantees and rates, it is useful to understand how they scale with architectural design choices and structural properties of the underlying graph.

We now consider a projected version of the REINFORCE algorithm from \cite{REINFORCE}. To prove convergence, we build on existing results for random-horizon policy-gradient estimators and projected stochastic approximation. Throughout this subsection, we keep the policy parametrization from the previous sections, and fix an initial state $s_0\in\mathcal S$. We further assume that the state space is compact and again that the activation function $\sigma$ is twice differentiable with bounded derivatives. Since the number of nodes is fixed and the action sets $\mathcal A_{\mathrm{glob}}$ and $\mathcal A_{\mathrm{node}}$ are finite, all graph- and input-dependent quantities entering the constants in Subsection~\ref{subsec:Lipschitz} are then uniformly bounded over $\mathcal S$. Hence, the constants in Theorem~\ref{theo:score-explicit} can be chosen uniformly over all admissible state-action pairs. With a slight abuse of notation, we denote the corresponding uniform bounds on the score function and its parameter-side Lipschitz constant by $B_g$ and $L_g$.

Let $\rho_\theta$ denote the normalized discounted state-action occupancy measure induced by $\pi_\theta$ and the initial state $s_0$. Then the policy gradient theorem yields
\begin{align*}
\nabla J(\theta)=\frac{1}{1-\gamma}\E_{\rho_\theta}\left[\nabla_\theta\log\pi_\theta(A\mid S)\,Q_\theta(S,A)\right],
\end{align*}
where $(S,A)$ is distributed according to $\rho_\theta$.
We use this as the starting point of the convergence analysis and refer to \cite{policy-gradient-methods, Zhang} for further details.

The gradient estimator in Algorithm~\ref{alg:projected-reinforce} is based on geometric rollout horizons. Following \cite{Zhang}, a first geometric horizon is used to sample a state-action pair from the discounted occupancy measure, while a second independent geometric rollout estimates the corresponding action-value term. Compared with the algorithm in \cite{Zhang}, we add a projection step in order to keep the iterates in a compact parameter set.

This subsection contains two convergence results. Theorem~\ref{theo:projected-reinforce-convergence} gives almost sure asymptotic convergence to the set of constrained stationary points, while Theorem~\ref{theo:finite-time-constrained-stationarity} provides a finite-time bound for the tangential component of the policy gradient. The proofs are deferred to Appendix~\ref{app:proofs-reinforce}.

To formulate the projection explicitly, we define the compact parameter set onto which the algorithm projects after each update.

\begin{definition}[Box-constrained parameter set]
\label{def:box-parameters}
Fix $W_\Theta>0$. We denote by $\Theta$ the set of all parameter vectors $\theta$ for which every scalar entry of
$$W^{(1)},b^{(1)},\dots,W^{(K)},b^{(K)},W^G,b^G,W^N,b^N$$
lies in the interval $[-W_\Theta,W_\Theta]$. We denote by $\Pi_\Theta$ the Euclidean projection onto $\Theta$.
\end{definition}

After vectorization, $\Theta$ is a compact convex hyperrectangle. In particular, the weights of all admissible parameters $\theta\in\Theta$ satisfy Assumption~\ref{assum:bounds} for suitable constants $\beta_1,\ldots,\beta_K,\beta_G,\beta_N$. Together with the compactness of the state space, this implies that the activation function and its derivatives only need to be bounded on the compact range of pre-activation values generated by the network.

\begin{remark}
The box constraint in Definition~\ref{def:box-parameters} can be replaced by different bounded intervals for different trainable parameters. The resulting parameter set is still a compact hyperrectangle, and the convergence arguments remain unchanged.
\end{remark}

We consider the projected mini-batch update in Algorithm~\ref{alg:projected-reinforce}. We use the convention that $\mathrm{Geom}(p)$ is supported on $\mathbb N_0$ with $\mathbb P(T=t)=p(1-p)^t$.

\begin{algorithm}[!htb]
\caption{Mini-batch Projected REINFORCE with Geometric Rollouts}
\label{alg:projected-reinforce}
\begin{algorithmic}[1]
\Require initial parameter $\theta_0 \in \Theta$, initial state $s_0 \in \mathcal S$, step sizes $(\alpha_k)_{k\in\N_0}$, batch sizes $(m_k)_{k\in\N_0}$
\State $k \gets 0$
\While{not converged}
    \For{$i=1,\dots,m_k$}
        \State Set $S^{(i)}_0 \gets s_0$
        \State Sample $T_{k,i} \sim \mathrm{Geom}(1-\gamma)$
        \For{$t=0,\dots,T_{k,i}-1$}
            \State Sample $A^{(i)}_t \sim \pi_{\theta_k}(\cdot \mid S^{(i)}_t)$ on $\mathcal A(S^{(i)}_t)$ and $S^{(i)}_{t+1} \sim p(\cdot \mid S^{(i)}_t,A^{(i)}_t)$
        \EndFor
        \State Sample $A^{(i)}_{T_{k,i}} \sim \pi_{\theta_k}(\cdot \mid S^{(i)}_{T_{k,i}})$ on $\mathcal A(S^{(i)}_{T_{k,i}})$
        \State Set $\widetilde S^{(i)}_0 \gets S^{(i)}_{T_{k,i}}$ and $\widetilde A^{(i)}_0 \gets A^{(i)}_{T_{k,i}}$
        \State Sample $T'_{k,i} \sim \mathrm{Geom}(1-\sqrt{\gamma})$ independently of $T_{k,i}$ and the subsequent rollout
        \State Sample $\widetilde S^{(i)}_1 \sim p(\cdot \mid \widetilde S^{(i)}_0,\widetilde A^{(i)}_0)$
        \State Observe $\widetilde R^{(i)}_0 = r(\widetilde S^{(i)}_0,\widetilde A^{(i)}_0,\widetilde S^{(i)}_1)$
        \For{$t=1,\dots,T'_{k,i}$}
            \State Sample $\widetilde A^{(i)}_t \sim \pi_{\theta_k}(\cdot \mid \widetilde S^{(i)}_t)$ on $\mathcal A(\widetilde S^{(i)}_t)$ and $\widetilde S^{(i)}_{t+1} \sim p(\cdot \mid \widetilde S^{(i)}_t,\widetilde A^{(i)}_t)$
            \State Observe $\widetilde R^{(i)}_t = r(\widetilde S^{(i)}_t,\widetilde A^{(i)}_t,\widetilde S^{(i)}_{t+1})$
        \EndFor
        \State Compute $$\widehat g_{k,i} \gets \frac{1}{1-\gamma}\,\nabla_\theta \log \pi_{\theta_k}(A^{(i)}_{T_{k,i}} \mid S^{(i)}_{T_{k,i}}) \sum_{t=0}^{T'_{k,i}} \gamma^{t/2}\,\widetilde R^{(i)}_t$$
    \EndFor
    \State Compute the mini-batch estimator $\bar g_k \gets \frac{1}{m_k}\sum_{i=1}^{m_k}\widehat g_{k,i}$
    \State Update $\theta_{k+1} \gets \Pi_\Theta\!\left(\theta_k+\alpha_k\bar g_k\right)$
    \State $k \gets k+1$
\EndWhile
\end{algorithmic}
\end{algorithm}

The projection ensures that the network parameters remain in the compact set $\Theta$, so that the uniform score-function bounds from Subsection~\ref{subsec:Lipschitz} apply throughout the algorithm. Since the parameters are restricted to a hyperrectangle, convergence is stated in terms of constrained first-order stationarity. For a closed convex set $M\subset\R^d$ and $x\in M$, the normal cone at $x$ is defined by
$$N_M(x):=\bigl\{v\in\R^d:\langle v,x-c\rangle\ge 0 \text{ for all } c\in M\bigr\},$$
and the tangent cone at $x$ is given by
$$T_M(x):=\bigl\{w\in\R^d:\langle w,v\rangle\le 0 \text{ for all } v\in N_M(x)\bigr\}.$$

\begin{theorem}[Almost sure convergence to constrained stationary points]
\label{theo:projected-reinforce-convergence}
Consider the parametrized policy $\pi_\theta$ from Section~\ref{sec:gnn-architecture}, and let $(\theta_k)_{k\in\N_0}$ be the iterates generated by Algorithm~\ref{alg:projected-reinforce}. Assume that $m_k\in\N$ for all $k\in\N_0$ and that the step sizes satisfy
$$\alpha_k>0,\qquad \sum_{k=0}^\infty \alpha_k=\infty,\qquad \sum_{k=0}^\infty \alpha_k^2<\infty.$$
Further, let
$$S_\Theta:=\Bigl\{\theta\in\Theta:\Pi_{T_\Theta(\theta)}\bigl(\nabla J(\theta)\bigr)=0\Bigr\}$$
have empty interior. Then
$$\operatorname{dist}(\theta_k,S_\Theta)\xrightarrow[k\to\infty]{}0 \qquad\text{almost surely.}$$
\end{theorem}

\begin{remark}\label{rem:interpretation-convergence}
Although Theorem~\ref{theo:projected-reinforce-convergence} does not imply convergence of the full sequence $(\theta_k)_{k\in\N_0}$, compactness of $\Theta$ guarantees the existence of convergent subsequences. Every accumulation point $\bar\theta$ satisfies $\Pi_{T_\Theta(\bar\theta)}\bigl(\nabla J(\bar\theta)\bigr)=0$ and is thus a constrained first-order stationary point of $\max_{\theta\in\Theta} J(\theta).$ If $\bar\theta$ lies in the interior of $\Theta$, this condition reduces to $\nabla J(\bar\theta)=0$.
\end{remark}

The previous theorem gives an asymptotic almost sure convergence statement, but does not provide a rate, since its proof is based on a compactness argument. The next result gives a finite-time bound for the tangential component of the policy gradient. For the finite-time estimate, set
$$L_J:=\frac{R_{\max}}{(1-\gamma)^{2}}\left( L_g+\frac{1+\gamma}{1-\gamma} B_g^2\right), \qquad \sigma:=\frac{B_gR_{\max}}{(1-\gamma)(1-\sqrt{\gamma})}.$$

\begin{theorem}[Finite-time constrained stationarity bound]
\label{theo:finite-time-constrained-stationarity}
Consider the parametrized policy $\pi_\theta$ from Section~\ref{sec:gnn-architecture}, and let $(\theta_k)_{k=0}^{N}$ be generated by Algorithm~\ref{alg:projected-reinforce}. Assume that $0<\alpha_k<2/L_J$ for $k=0,\ldots,N-1$.

Define
\[w_k:=\alpha_k-\frac{L_J}{2}\alpha_k^2, \qquad W_N:=\sum_{k=0}^{N-1}w_k,\qquad \mathbb P(\tau=k):=\frac{w_k}{W_N},\]
where $\tau$ is independent of the algorithmic randomness. Further, set
$\Delta_0:=\max_{\theta\in\Theta}J(\theta)-J(\theta_0)$ and $\alpha_{\max}:=\max_{0\leq k\leq N-1}\alpha_k$.

Then
\[\mathbb E\!\left[ \left\|\Pi_{T_\Theta(\theta_{\tau+1})}\bigl(\nabla J(\theta_{\tau+1})\bigr)\right\|^2\right]
\leq 2(1+L_J\alpha_{\max})^2 \frac{ \Delta_0+\sigma^2\sum_{k=0}^{N-1}\frac{\alpha_k}{m_k}}{W_N}+2\sigma^2\sum_{k=0}^{N-1}\frac{w_k}{W_Nm_k}.\]

\noindent The expectation is taken with respect to $\tau$ and the rollout randomness used to generate the iterates.
\end{theorem}

\begin{remark}
For constant batch sizes $m_k=m$, the right-hand side in Theorem~\ref{theo:finite-time-constrained-stationarity} does not vanish as $N$ tends to infinity, independently of the choice of the step sizes. As $w_k\leq\alpha_k$, we have
\begin{align*}
&2(1+L_J\alpha_{\max})^2
\frac{\Delta_0+\sigma^2\sum_{k=0}^{N-1}\frac{\alpha_k}{m}}{W_N}
+
2\sigma^2\sum_{k=0}^{N-1}\frac{w_k}{W_Nm}
\geq
\left(2(1+L_J\alpha_{\max})^2+2\right)\frac{\sigma^2}{m}.
\end{align*}
For batch sizes satisfying $m_k\xrightarrow[k\to\infty]{}\infty$ and step sizes such that, for some $\overline\alpha<2/L_J$, $$0<\alpha_k\leq\overline\alpha,\qquad \sum_{k=0}^{\infty}\alpha_k=\infty,$$
the bound in Theorem~\ref{theo:finite-time-constrained-stationarity} converges to zero.
\end{remark}

\newcommand{\numresultcell}[3]{\makecell[c]{\scriptsize $#1$\\[-0.22em]{\tiny Cap./Def.: $#2/#3$}}}
\newcommand{\nummodelhead}[2]{\makecell[c]{\scriptsize #1\\[-0.20em]{\tiny #2 parameters}}}
\newcommand{\numplainhead}[1]{\makecell[c]{\scriptsize #1}}
\newcommand{\networkpanel}[3]{%
    \begin{subfigure}[t]{0.18\textwidth}
        \centering
        \includegraphics[width=\linewidth,height=0.088\textheight,keepaspectratio]{#1}
        \caption{#2}
        \label{#3}
    \end{subfigure}%
}

\section{Applications to Systemic Risk}\label{sec:numerical-results}
This section applies the framework developed above to the control of systemic risk in financial networks. At each time step, a central regulator may inject capital into solvent banks in order to reduce losses caused by future default cascades. The graph-level action determines the total intervention budget, while the node-level actions determine its allocation across the banks. The numerical experiments investigate the performance of the proposed policy architecture across different network structures, with particular emphasis on the role of graph topology, locality-based fairness, and scalability. We first formulate the intervention problem and describe the numerical setup, then evaluate the learned strategies on small benchmark networks and larger randomly generated financial networks.

Since we represent each financial network as a weighted, directed graph, we use the terms network and graph interchangeably throughout this section.

\subsection{Dynamic Financial Network Model}
\label{subsec:financial-network-model}

We consider a financial system consisting of $n$ banks whose bilateral obligations form a directed, weighted network. Each bank is represented as a node of the network, and a directed edge from $i$ to $j$ carries weight $L_{ij}(t)$, representing the nominal interbank liability of bank $i$ towards bank $j$ at time $t$. Liabilities towards entities outside the modeled banking system are included as node features.

In addition to their interbank exposures, institutions hold portfolios of external assets. These consist of two components: a risk-free asset whose value is normalized and does not vary over time, and a risky asset whose market price evolves endogenously. Modeling external assets in this way is common in the literature, see \cite{FS-Cifuentes, FS-Amini, AW_15}.

Let
$$V:=V_{\mathrm{rf}}\times V_{\mathrm{risky}}\times V_{\mathrm{out}}\times V_S\subset\mathbb R^4$$
be the node-feature space, where the components correspond to risk-free asset holdings, risky asset holdings, liabilities towards the outside sector, and the current price of the risky asset. We assume that all component spaces are compact intervals, not necessarily identical. The component spaces of the asset positions may contain negative values, corresponding for instance to short positions. 

We denote by \(X(t)\in V^n\) the node-feature matrix at time \(t\). The market price of the risky asset is denoted by \(S(t)\in V_S\) and is included as a node feature by assigning the same value to every node, that is,
$$
X_{i,S}(t)=S(t),\qquad i=1,\ldots,n.
$$

The interbank liability matrix satisfies \(L(t)\in W^{n\times n}\), where \(W\subset\mathbb R_{\geq 0}\) is a compact interval. Hence, the state of the financial network at time \(t\) is given by \((X(t),L(t))\), and the state space is
\begin{equation}
\label{eq:financial-state-space}
\mathcal S \subseteq V^n \times W^{n\times n}.
\end{equation}

This is a graph-valued state space of the form \eqref{eq:graph-state-space} with node-feature space $V$.
With a slight abuse of notation, for a state $s=(X,L)\in\mathcal S$, we write $S(s)$ for the common asset-price feature. The total interbank liabilities and total interbank claims of bank $i$ ar determined by the liability matrix and are given by
$$\ell_i^{\mathrm{ib}}(L):=\sum_{j=1}^n L_{ij},\qquad \chi_i^{\mathrm{ib}}(L):=\sum_{j=1}^n L_{ji}.$$
Thus, the total asset value of bank \(i\) in state \(s\) is \(\chi_i^{\mathrm{ib}}(L)+X_{i,\mathrm{rf}}+X_{i,\mathrm{risky}}S(s)\), whereas its total liabilities are \(\ell_i^{\mathrm{ib}}(L)+X_{i,\mathrm{out}}\). Its net worth is therefore
$$C_i(s):=\chi_i^{\mathrm{ib}}(L)+X_{i,\mathrm{rf}}+X_{i,\mathrm{risky}}S(s)-\ell_i^{\mathrm{ib}}(L)-X_{i,\mathrm{out}}.$$
A bank is solvent in state $s$ if $C_i(s)\geq 0$ and insolvent otherwise.

At each time step, the regulator may inject capital into solvent banks.\footnote{Our framework can be extended to incorporate support for insolvent banks or the provision of loss-coverage guarantees as an alternative to direct capital injections, both of which have been widely used in financial crisis interventions \cite{Panetta2009BIS}.} The additional capital is added to the risk-free asset holdings of each supported bank. The total amount of capital the regulator can distribute at each decision is limited. Since insolvent banks are excluded from support, interventions are preemptive. This leads to a state-dependent action space.

The financial network action space follows the mixed action structure introduced in \eqref{eq:general-mixed-action-space}. The graph-level decision is the total intervention budget, while the node-level decisions are allocation weights assigned to the individual banks. Let \(\mathcal A_{\mathrm{bud}}\subset\mathbb R_{\geq 0}\) and \(\mathcal A_{\mathrm{node}}\subset\mathbb R_{\geq 0}\) be finite sets of admissible budget levels and admissible allocation weights, respectively, with \(0\in\mathcal A_{\mathrm{bud}}\cap\mathcal A_{\mathrm{node}}\).
In the notation of the general action space, we set \(\mathcal A_{\mathrm{glob}}=\mathcal A_{\mathrm{bud}}\), \(a_{\mathrm{node}}^0=0\), and \(\eta(a_{\mathrm{bud}})=\1_{\{a_{\mathrm{bud}}>0\}}\).
We denote the cardinalities of the action sets by $K_B=|\mathcal A_{\mathrm{bud}}|$ and $K_{\mathrm{node}}=|\mathcal A_{\mathrm{node}}|$.

For a given state $s\in\mathcal S$, the feasible action set for bank $i$ is
\begin{align*}
\mathcal A_{\mathrm{node},i}(s)=\begin{cases}
\mathcal A_{\mathrm{node}}, & \text{if bank $i$ is solvent in state $s$},\\
\{0\}, & \text{if bank $i$ is insolvent in state $s$}.
\end{cases}
\end{align*}
Thus, insolvent banks cannot receive corrective support.

An action $a=(a_{\mathrm{bud}},a_1,\ldots,a_n)$, where $a_{\mathrm{bud}}\in\mathcal A_{\mathrm{bud}}$ and $a_i\in\mathcal A_{\mathrm{node},i}(s)$, induces the actual injection vector
\begin{equation}
\label{eq:injection-vector}
I(a)\in\mathbb R_{\geq 0}^n,\qquad I(a)_i= 
\begin{cases}
a_{\mathrm{bud}}\dfrac{a_i}{\sum_{j=1}^n a_j}, & \text{if }\sum_{j=1}^n a_j>0,\\[1em]
0, & \text{if }\sum_{j=1}^n a_j=0.
\end{cases}
\end{equation}
If the agent chooses action $a$, node $i$ receives $I(a)_i$ additional resources. Hence, the chosen budget is distributed proportionally according to the node-wise weights. This yields the feasible action space
\begin{equation}
\label{eq:action-space}
\mathcal A(s)=\bigl(\{0\}\times\{0\}^n\bigr)\cup\Bigl((\mathcal A_{\mathrm{bud}}\setminus\{0\})\times\prod_{i=1}^n\mathcal A_{\mathrm{node},i}(s)\Bigr).
\end{equation}
We note that different admissible actions may induce the same injection vector and therefore the same immediate effect on the financial system.

The transition mechanism works as follows. After a shock to external assets or otherwise triggered stress, banks may be unable to meet their obligations in full. In such a situation, their creditors incur losses, which may in turn render further institutions insolvent, leading to a cascade of defaults through the network. 

The seminal clearing model of Eisenberg and Noe \cite{EN01} provides a simple and elegant framework for analyzing the propagation of such default cascades in financial networks. Extensions of this framework incorporate bankruptcy costs \cite{RV13}, liquidation and default costs \cite{FS-Amini}, and fire-sale effects \cite{FS-Cifuentes,AW_15}. These models mainly study clearing and contagion in static or single-period settings, and have been the basis for explicit numerical methods to calculate systemic risk measures \cite{doi:10.1287/opre.2021.0040,gonon2026systemic}. We follow the extension of Weber and Weske \cite{AW_15}, which incorporates both bankruptcy costs and fire-sale effects into the clearing mechanism.

With bankruptcy costs, the value of assets in default can only be recovered partly, leading to an actual loss of value in the system. This differs from the ordinary Eisenberg--Noe algorithm, where clearing only redistributes wealth. Fire sales further amplify this effect by a system-wide mechanism. In our setup, only insolvent banks sell their risky assets. As interbank liabilities are settled through the risk-free asset, this may lead to a negative position in the risk-free asset for a solvent bank. Nonetheless, such a bank does not liquidate its risky asset holdings.

At each time step $t$, the evolution of the system proceeds in three stages. Starting from the post-clearing state $(X(t),L(t))$ with common asset price $S(t)$, the regulator first chooses an action $a_t$, and the induced capital injection $I(a_t)$ is added to the risk-free asset holdings of the supported banks. Next, the price of the risky asset evolves according to an exogenous stochastic process. In our implementation, we model this dynamics by an Ornstein--Uhlenbeck process, yielding an updated pre-clearing price $S^-(t)$.\footnote{The Ornstein--Uhlenbeck process is chosen deliberately for numerical convenience. In particular, the mean-reverting dynamics yield a more stationary environment, which allows us to isolate and better understand the effects of network interactions and regulatory interventions. We emphasize that the framework is sufficiently flexible to accommodate more realistic asset price dynamics, and extending the analysis in this direction is conceptually straightforward.} Together, the intervention and the price update define the pre-clearing node features
$$X^-_{i,\mathrm{rf}}(t)=X_{i,\mathrm{rf}}(t)+I(a_t)_i,\qquad X^-_{i,\mathrm{risky}}(t)=X_{i,\mathrm{risky}}(t),\qquad X^-_{i,\mathrm{out}}(t)=X_{i,\mathrm{out}}(t),\qquad X^-_{i,S}(t)=S^-(t).$$
The pre-clearing liability matrix is given by $L^-(t)=L(t)$, so that the pre-clearing state is \newline$(X^-(t),L^-(t))$.

Lastly, in case the pre-clearing system is in stress and banks are unable to meet their obligations in full, the new cleared financial network is determined by
$$(X(t+1),L(t+1))=\Lambda(X^-(t),L^-(t)),$$
where $\Lambda$ denotes the clearing operation. In case of insolvency, the clearing payment vector is first determined for the system. Subsequently, only those interbank liabilities involving at least one insolvent bank are settled according to this vector, meaning that the corresponding clearing payments are settled and the associated edges are removed from the graph. Liabilities between solvent banks remain unchanged. Hence, over time, the system can only lose edges, but cannot gain new ones. If all banks remain solvent, the new state only differs by the price of the risky asset and the injected capital.

Together with the distribution of the exogenous price shock, this defines the transition kernel $p(\cdot\mid s,a)$.
These transition dynamics lead to a particularly challenging learning problem, as the network structure only changes when defaults occur. At the same time, such events can result in substantial changes to the network. For example, in the 200 node networks considered in the numerical experiments below, when no capital is injected, new defaults occur on average once per 100-step episode, with approximately 92\% of the banks ending up insolvent.

For a transition from state $s$ to state $s'$ under action $a$, the regulator receives the reward
\begin{align}
\label{eq:reward-per-time-step}
r(s,a,s') := -\kappa a_{\mathrm{bud}} +\sum_{i=1}^n \left(\widehat C_i(s')- C_i(s)\right) - \ell_{\mathrm{out}}(s,a,s') -\lambda D_{\mathrm{new}}(s,a,s').
\end{align}
Here, $\kappa>1$ is a fixed intervention cost parameter, $D_{\mathrm{new}}(s,a,s')$ counts newly defaulted banks, and $\lambda\geq 0$ is a penalty factor. $\widehat C_i(s')$ denotes the net worth of bank $i$ according to the clearing operation generating state $s'$. Claims are valued according to the clearing payment vector, while own liabilities are considered with their nominal value. The term $\ell_{\mathrm{out}}(s,a,s')$ denotes the losses of external creditors due to default, measured as the difference between the nominal value of their affected claims and the corresponding clearing payments.

The first term captures the direct cost of the intervention, while the second term measures the change in aggregate bank value. The third term accounts for losses of external creditors, and the fourth penalizes newly defaulted banks.
The particular reward specification in \eqref{eq:reward-per-time-step} is only one possible choice. More generally, the algorithm can be applied to arbitrary bounded reward functions.

Since the state space is compact\footnote{The Ornstein-Uhlenbeck is restricted to a compact interval, by projecting values outside the interval onto the boundaries.} and the admissible action sets are finite, the reward is bounded. Hence, the financial network model satisfies the bounded-reward assumption from the graph-valued MDP formulation above.

\subsection{Numerical Setup and Hyperparameters}
\label{subsec:systemic-risk-numerical-setup}
We consider fixed benchmark networks with ten banks, as well as randomly generated networks with ten or 200 banks. The fixed networks make it possible to relate performance differences to structural properties of the graph. The randomly generated networks test whether a policy can learn an allocation rule across different initial financial systems.
In the following, we denote the class of randomly generated financial networks with $n$ banks by Random$n$.

In addition to the features introduced in Subsection~\ref{subsec:financial-network-model}, the numerical observation contains total interbank liabilities, liabilities towards the outside sector, total interbank claims, and a default-status indicator for each bank. To account for bankruptcy costs, we assume that only $80\%$ of the asset value of an insolvent bank is available for redistribution to its creditors. 
 
The complete code for our simulation and training setup, as well as additional configurations, is available on \href{https://github.com/flknm/global-local-rl-dynamic-networks}{GitHub}.

\subsubsection{Network Construction}
\paragraph{Small Benchmark Networks}
Figure~\ref{fig:small-networks} shows the fixed benchmark networks. Green nodes represent banks without exposure to the risky asset, while nodes become increasingly red as their dependence on the risky asset increases. The Core-Periphery, Star, and Inverse Star networks are used for the main ten-bank experiments. The Symmetric Stars and Mirrored Cycles networks are separate deterministic one-decision examples especially used to examine the implications of fairness.

\begin{figure}[t]
    \centering
    \captionsetup[subfigure]{font=scriptsize,justification=centering,skip=1pt}
    \setlength{\abovecaptionskip}{2pt}
    \setlength{\belowcaptionskip}{0pt}
    \networkpanel{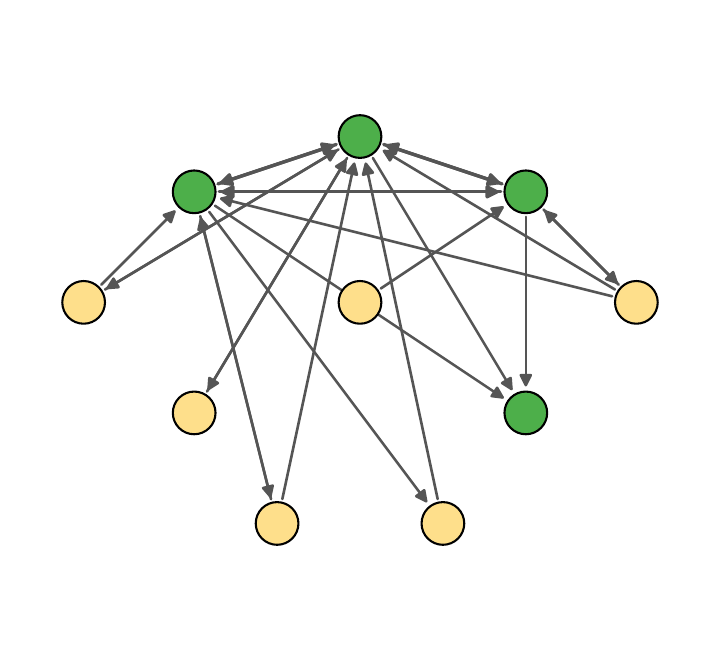}{Core-periphery}{fig:network-standard}\hfill
    \networkpanel{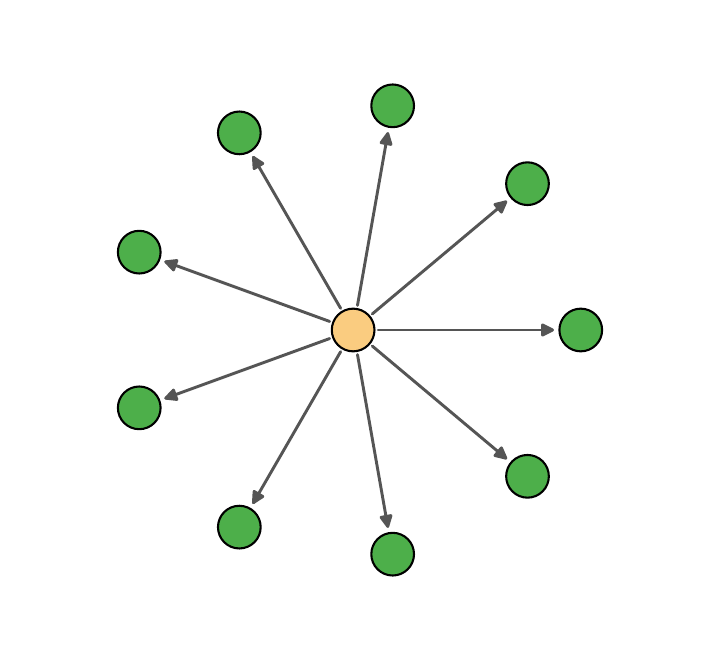}{Star}{fig:network-star}\hfill
    \networkpanel{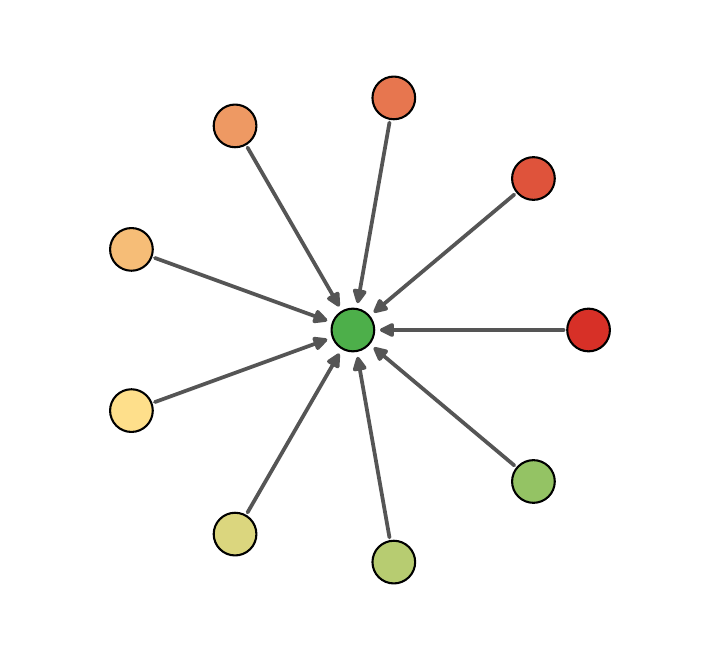}{Inverse Star}{fig:network-no-graph}\hfill
    \networkpanel{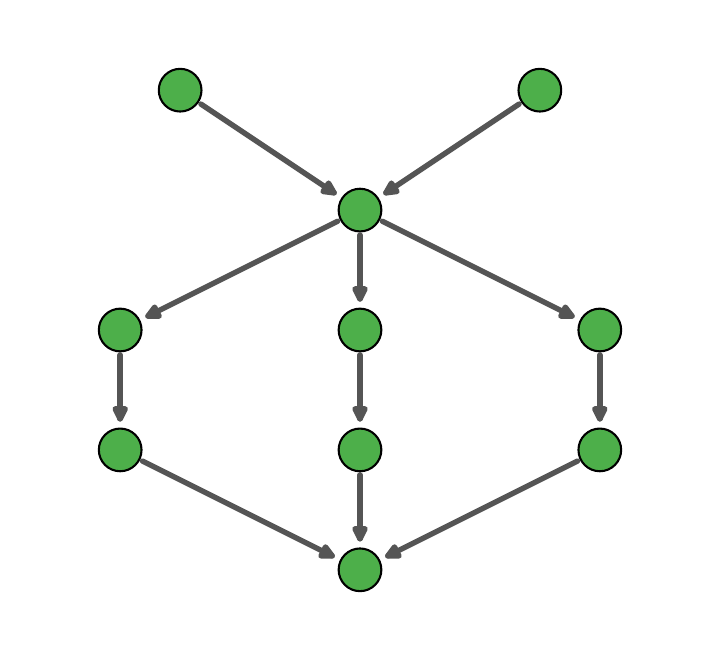}{Sym. Stars}{fig:network-sym-stars}\hfill
    \networkpanel{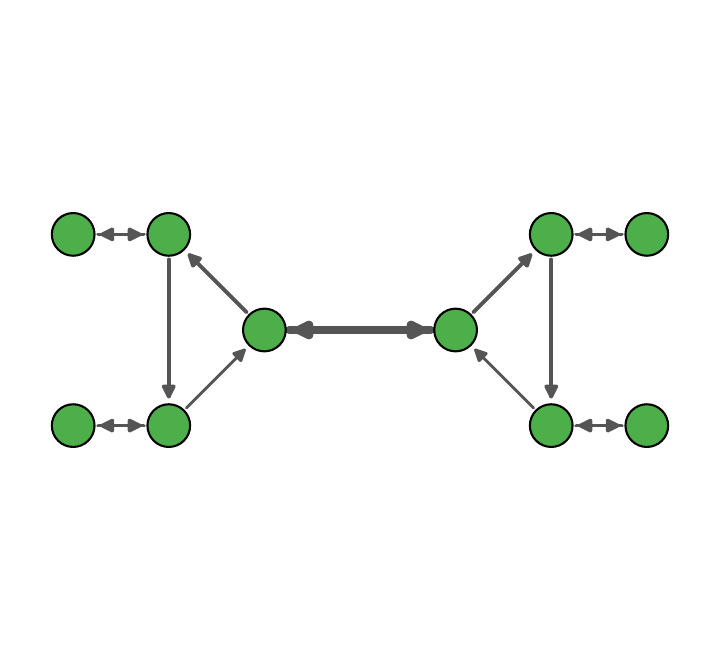}{Mirr. Cycles}{fig:network-sym-circles}
    \caption{Fixed ten-bank benchmark networks. Node colors indicate exposure to the risky asset and edges point from debtors to creditors.}
    \label{fig:small-networks}
\end{figure}

Numbering the banks from 1 to 10, Tables~\ref{tab:benchmark-liabilities}
and~\ref{tab:benchmark-features} specify the initial interbank liabilities and node features of the fixed benchmark networks, respectively.

\begin{table}[t]
\centering
\small
\setlength{\tabcolsep}{5pt}
\renewcommand{\arraystretch}{1.15}
\caption{Initial nonzero interbank liabilities of the fixed benchmark networks. Each ordered pair $(i,j)$ denotes a liability from bank $i$ to bank $j$ with the weight shown in the second column. All unlisted liabilities are zero.}
\label{tab:benchmark-liabilities}
\begin{tabularx}{\textwidth}{@{}l r X@{}}
\toprule
Network & Weight & Ordered pairs $(i,j)$ \\
\midrule

Core-Periphery
& $0.75$
& $(1,2),(1,3),(2,1),(2,3),(3,1),(3,2)$ \\

& $0.32$
& $(4,3),(5,1),(5,2),(6,1),(6,2),(6,3),(7,1),(9,1),(9,2),(10,1)$ \\

& $0.26$
& $(1,5),(1,7),(1,8),(2,8),(2,9),(2,10),(3,6),(3,8)$ \\

\addlinespace

Star
& $0.5$
& $(1,j)$ for $j=2,\ldots,10$ \\

\addlinespace

Inverse Star
& $1$
& $(i,5)$ for $i\in\{1,2,3,4,6,7,8,9,10\}$ \\

\addlinespace

Symmetric Stars
& $2$
& $(1,3),(2,3),(3,4),(3,5),(3,6),(4,7),(5,8),(6,9),
(7,10),(8,10),(9,10)$ \\

\addlinespace

Mirrored Cycles
& $3.9$
& $(1,4),(4,1)$ \\

& $1$
& $(1,2),(2,3),(2,7),(3,8),(4,5),(5,6),(5,9),(6,10),
(7,2),(8,3),(9,5),(10,6)$ \\

& $0.5$
& $(3,1),(6,4)$ \\

\bottomrule
\end{tabularx}
\end{table}
\begin{table}[t]
\centering
\small
\setlength{\tabcolsep}{5pt}
\renewcommand{\arraystretch}{1.15}
\caption{Initial node features of the fixed benchmark networks.}
\label{tab:benchmark-features}
\begin{tabularx}{\textwidth}{@{}l X@{}}
\toprule
Network & Initial node features \\
\midrule

Core-Periphery
&
$(X_{i,\mathrm{rf}})_{i=1}^{10}
=
(X_{i,\mathrm{risky}})_{i=1}^{10}
=
(0,0,0,0.18,0.21,0.38,0.03,0,0.21,0.03)^\top$,

$(X_{i,\mathrm{out}})_{i=1}^{10}=0$.
\\

\addlinespace

Star
&
$(X_{i,\mathrm{rf}})_{i=1}^{10}
=
(2.2,0,\ldots,0)^\top$,

$(X_{i,\mathrm{risky}})_{i=1}^{10}
=
(2.75,0,\ldots,0)^\top$,

$(X_{i,\mathrm{out}})_{i=1}^{10}
=
(0,0.5,\ldots,0.5)^\top$.
\\

\addlinespace

Inverse Star
&
$(X_{i,\mathrm{rf}})_{i=1}^{10}
=
(0,0.11,0.22,0.33,0,0.44,0.55,0.66,0.77,0.88)^\top$,

$(X_{i,\mathrm{risky}})_{i=1}^{10}
=
(1.1,0.99,0.88,0.77,0,0.66,0.55,0.44,0.33,0.22)^\top$,

$(X_{i,\mathrm{out}})_{i=1}^{10}=0$.
\\

\addlinespace

Symmetric Stars
&
$(X_{i,\mathrm{rf}})_{i=1}^{10}
=
(1,1,3.3,0,\ldots,0)^\top$,

$(X_{i,\mathrm{risky}})_{i=1}^{10}=0$,

$(X_{i,\mathrm{out}})_{i=1}^{10}
=
(0,\ldots,0,6)^\top$.
\\

\addlinespace

Mirrored Cycles
&
$(X_{i,\mathrm{rf}})_{i=1}^{10}
=
(X_{i,\mathrm{risky}})_{i=1}^{10}=0$,

$(X_{i,\mathrm{out}})_{i=1}^{10}
=
(0,0,0.5,0,0,0.5,0,0,0,0)^\top$.
\\

\bottomrule
\end{tabularx}
\end{table}

For the Symmetric Stars and Mirrored Cycles networks, we set $\mathcal A_{\mathrm{node},i}(s)=\mathcal A_{\mathrm{node}}$ for all banks, thereby allowing insolvent banks to also receive capital injections.
Then, by construction, both networks have an obvious, optimal solution when considering a maximum budget of one. In the Mirrored Cycles network, the full budget has to be equally divided among banks one and four. In the Symmetric Stars network, a system-wide cascade of the network can only be prevented, by either supporting bank one or bank two with the total available budget. 
As the neighborhoods of the critical nodes in both networks are identical, our fair message-passing policies can theoretically solve the MirroredCcycles network optimally, but not the Symmetric Stars network. Hence, these examples distinguish a situation in which effective intervention requires unequal treatment of indistinguishable banks from one in which equal treatment is sufficient.
As the actions are sampled stochastically during training, unfair capital allocations may occur even when the action distributions for different banks are identical. Thus, the policy may realize and even learn to assign high probability to the optimal solution for the Symmetric Stars network during training, although this intervention is not available under deterministic evaluation. Approaches to reduce this mismatch between stochastic action selection during training and deterministic action selection during evaluation, which may complicate the learning process, are left to future work.

\paragraph{Random Network Generation}
The generation process of our random core-periphery networks is similar to that of \cite{AW_15}. We use three core banks for $n=10$ and 50 core banks for $n=200$. Directed edges between banks are sampled indepedently, with probabilities $0.66$, $0.15$, $0.07$, and $0.001$ for core-to-core, core-to-periphery, periphery-to-core, and periphery-to-periphery edges, respectively.

The total amount of interbank liabilities is set equal to the number of modeled banks and is split among the four edge categories. Of the total liabilities, $0.35$ are assigned to liabilities between core banks, $0.16$ to liabilities from core banks to periphery banks, $0.47$ to liabilities from periphery banks to core banks, and $0.02$ to liabilities between periphery banks. These amounts are then equally distributed over the sampled edges in the corresponding category. If a category contains no edges, its amount is omitted.
External liabilities are uniformly sampled as a fraction between $0$ and $0.2$ of the bank's interbank liabilities.

To ensure that all banks are initially solvent, we choose their risk-free and risky asset holdings as 
\begin{align*}
X_{i,\mathrm{rf}}(0) &= (1-\rho)(1+\delta) \max\bigl\{ \ell_i^{\mathrm{ib}}(L(0)) +X_{i,\mathrm{out}}(0) -\chi_i^{\mathrm{ib}}(L(0)),\,0\bigr\},\\
X_{i,\mathrm{risky}}(0) &= \frac{\rho(1+\delta)}{S(0)} \max\bigl\{ \ell_i^{\mathrm{ib}}(L(0)) +X_{i,\mathrm{out}}(0) -\chi_i^{\mathrm{ib}}(L(0)),\,0\bigr\},
\end{align*}
where we choose a buffer factor of $\delta=0.1$ and a risky asset portion of $\rho=0.5$.

\subsubsection{Policies and Training}
\label{subsubsec:model-configurations-training-evaluation}
The default policy uses a shared MPNN encoder with two message-passing layers as in Definition~\ref{def:MPNN-layer}, node embedding dimension $d=8$, $\tanh$ activations, sum aggregation, and sum readout. We compare this policy with architectural variants that vary one component at a time. We change the depth or width of the MPNN encoder, utilize mean readout or aggregation function as mentioned in Proposition~\ref{prop:operatornorm-aggregate} and Remark~\ref{rem:theorem-score} respectively, and separate budget and allocation encoders as illustrated in Figure~\ref{fig:policy-architecture}.
We consider two configurations with separate encoders. One uses two message-passing layers in each encoder and is labeled with $2/2$, whereas $4/1$ refers to the configuration with four layers in the budget encoder and one layer in the allocation encoder.
Additionally, we compare our MPNN architecture with a bidirectional variant of an attention-based encoder as in \cite{GATv2,GAT}. 
To estimate the influence of Algorithm~\ref{alg:projected-reinforce}, we also include a stabilized variant of a REINFORCE algorithm that combines a leave-one-out baseline, standardization of the centered return coefficients, gradient-norm clipping, and Adam \cite{ADAM}.

The precise model configurations differ between the small and large-network experiments and are introduced together with the corresponding result graphs.

\paragraph{Initial States and Training Methodology}
To test different aspects of the architecture, we consider two training regimes. For topology-specific training, policies are trained on one network class. Every trajectory starts from the selected network class and thus, except for Random10 and Random200, the first observed state is fixed, while the subsequent risky-asset paths are stochastic. In case of the randomly generated networks, every trajectory starts with a newly generated graph and balance sheet.

Multi-topology training independently samples the initial state for each trajectory uniformly from eight different network classes: Core-Periphery, Star, Inverse Star, Random10, as well as a circular graph, a chain graph and two bipartite graphs. Symmetric Stars and Mirrored Cycles are excluded from this training distribution. The circular graph, the chain graph and the two bipartite graphs only occur during training and, for brevity, are not considered further in the analysis.

To test generalization across previously unseen network topologies, we evaluate the models trained solely on Random10 on all benchmark networks without further training.

\paragraph{Rollout Schemes and Baseline Setup}
In the fixed benchmark networks, the reward of the regulator heavily depends on its early decisions. Hence, for the small networks, we deviate from Algorithm~\ref{alg:projected-reinforce} by using a fixed horizon of length 50. 

Following the standard policy-gradient likelihood-ratio argument~\cite{policy-gradient-methods}, the objective for a fixed finite horizon of length $H$
$$J_H(\theta)=\E\left[\sum_{t=0}^{H-1}\gamma^t R_t\right]$$
yields
$$\nabla J_H(\theta)=\E\left[\sum_{t=0}^{H-1}\gamma^t\left(\sum_{j=0}^{H-1-t}\gamma^j R_{t+j}\right)\nabla_\theta \log\pi_\theta(A_t \mid S_t)\right].$$
Hence, we use the trajectory-based policy-gradient estimator
$$\widehat g=\sum_{t=0}^{H-1}\gamma^t\left(\sum_{j=0}^{H-1-t}\gamma^j r_{t+j}\right)\nabla_\theta\log\pi_\theta(a_t\mid s_t).$$

Its average over each batch is used in the projected parameter update. This ensures that the gradient contribution associated with the initial decision is included in every trajectory.
Further, as the Symmetric Stars and Mirrored Cycles networks do not depend on any randomness, they are modeled as one time decisions.

For our main experiments on Random200, we use the geometric-rollout training regime as stated in Algorithm~\ref{alg:projected-reinforce}.

To obtain robust results, each model configuration is trained on ten distinct training seeds, which are shared across all configurations. A given seed is used to initialize the parameters, the environment and the random draws used for stochastic action selection during training.

Table~\ref{tab:baseline-parameters} summarizes the baseline numerical setup, including the rollout schemes, training hyperparameters and the evaluation settings.

\begin{table}[t]
\centering\small
\setlength{\tabcolsep}{5pt}
\renewcommand{\arraystretch}{1.08}
\caption{Summary of the baseline numerical setup.}
\label{tab:baseline-parameters}
\begin{tabularx}{0.96\textwidth}{@{}p{0.25\textwidth}X@{}}
\toprule
Component & Setting \\
\midrule
Risky-asset dynamics & Ornstein--Uhlenbeck process with initial price $S(0)=1.0$, long-run mean $1.0$, volatility $0.15$, and mean-reversion speed $0.2$. \\
Action space & Budgets $\mathcal A_{\mathrm{bud}}=\{0,0.1,\ldots,1\}$ for $n=10$ and $\{0,0.5,\ldots,30\}$ for $n=200$. Node weights $\{0,1,2\}$ and $\{0,1,\ldots,6\}$, respectively. \\
Reward & Intervention cost $\kappa=4/3$, discount factor $\gamma=0.9$ and $\lambda =0$ unless stated otherwise.\\
Training budget & $10{,}000{,}000$ formal transitions, batches of size four, and ten training seeds. The specialized one-decision runs use $2{,}000{,}000$ transitions. \\
Rollouts & $H=50$ for ordinary ten-bank training, $H=1$ for specialized one-decision training, geometric sampling for Random200. \\
Step size and projection & $\alpha_k=\alpha_0(0.2)^{k/\bar K}$, where $\bar K$ is the expected number of updates within the training budget. Initially, $\alpha_0=10^{-4}$ for $n=10$, and $\alpha_0=10^{-3}$ for $n=200$. All trainable parameters are projected elementwise onto $[-5,5]$ after each update.\\
Evaluation & Deterministic action selection on 100 common environment seeds, with horizon 100. The one-decision examples use a horizon of one. \\
Default policy & Shared bidirectional MPNN, two layers, node embedding dimension 8, $\tanh$ activations, sum aggregation, and sum readout. \\
\bottomrule
\end{tabularx}
\end{table}

\paragraph{Heuristics and Feedforward Baselines}
We compare the learned policies with three heuristics. The \emph{uniform} heuristic distributes the available capital uniformly among all banks eligible for support. The \emph{zero} heuristic never injects capital. At each time step, the \emph{greedy} heuristic supports one eligible bank, which has the largest positive difference between liabilities and the sum of risk-free assets and interbank claims, with the total available capital. If no bank has such a deficit, it does nothing.

In addition, we compare the graph-based agents with feedforward neural network (FFNN) policies. These FFNNs use flattened observations and two hidden layers, with width 16 for ten banks and width 50 for 200 banks. The ten-bank FFNN was trained and evaluated using a fixed node ordering, while the 200-bank FFNN setup uses random node permutations upon environment reset.
These comparisons allow us to analyze the benefits of graph neural networks, as well as potentiall performancy loss by the restrictions of fairness and transparency.

\subsection{Results}
\label{subsec:systemic-risk-results}
In the performance figures below, each point represents one learned policy's mean discounted return on the common evaluation episodes, while the larger diamonds indicate the median performance across the ten training-seed means. The distinct training regimes, topology-specific and multi-topology, as well as the zero-shot performance of the policies trained only on Random10 are distinguishable by color. For reference, the performance of the heuristics are displayed as vertical lines, if they are inside the displayed range and otherwise as numerical values in the legend. Further, the best performance of any learned policy is also highlighted as a vertical line. 

The large-network Table~\ref{tab:random-200-combined} additionally reports undiscounted realized capital injections and final defaults. Note that the reward penalizes the selected budget level, whereas the tables report the realized capital injection. These differ, if a positive budget decision is combined with zero allocation weights.

\subsubsection{Controlled Experiments on Small Networks}
\label{subsubsec:small-network-results}
Figures~\ref{fig:performance-core10}, \ref{fig:performance-star10}, \ref{fig:performance-invstar10}, and~\ref{fig:performance-r10} give a compact overview of the experiments on the small networks. 

\begin{figure}[!htb]
 \includegraphics[width=\textwidth,height=0.76\textheight,keepaspectratio]{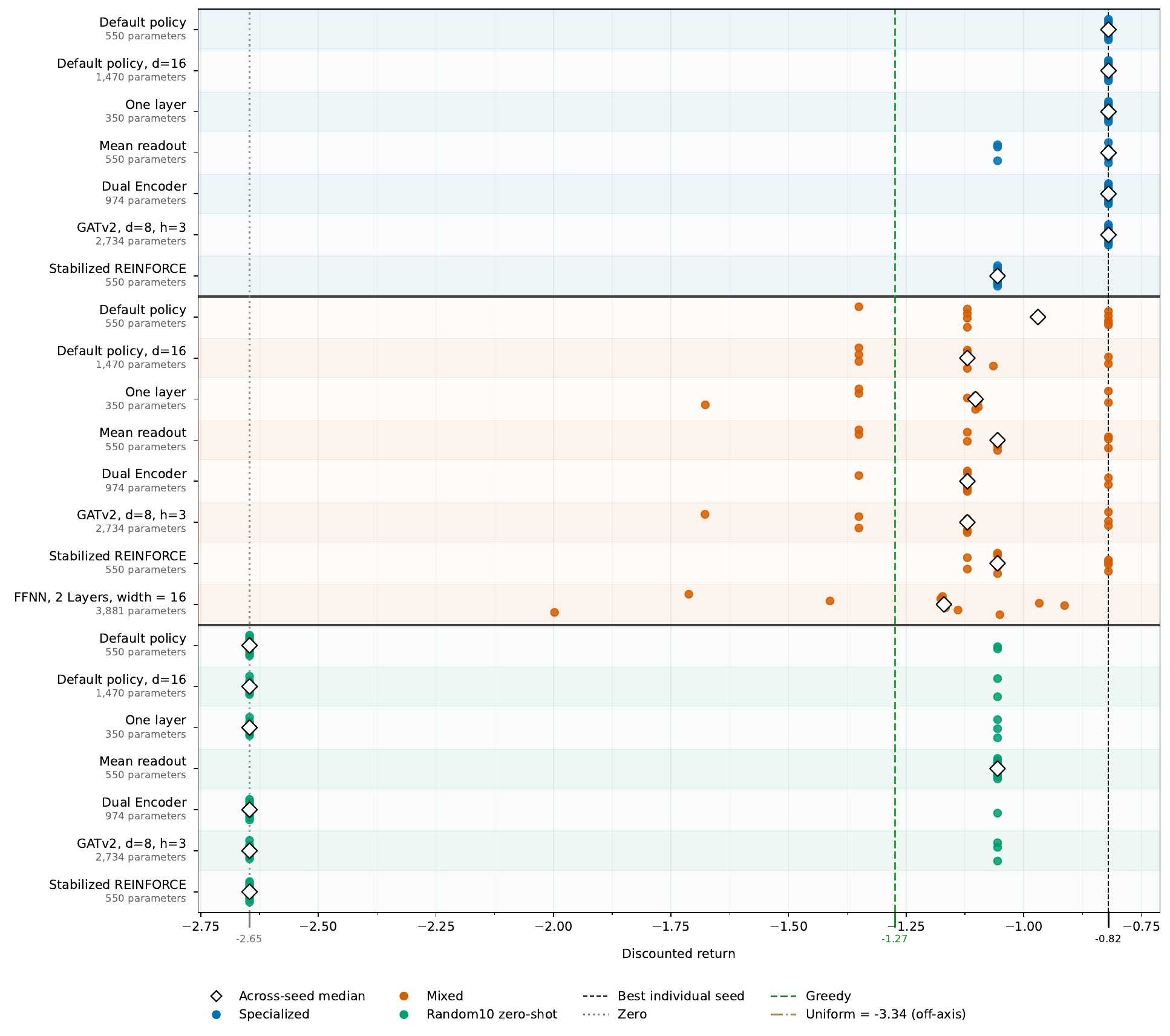}
 \caption{Core-Periphery}
 \label{fig:performance-core10}
\end{figure}

\begin{figure}[!htb]
 \centering
 \includegraphics[width=\textwidth,height=0.76\textheight,keepaspectratio]{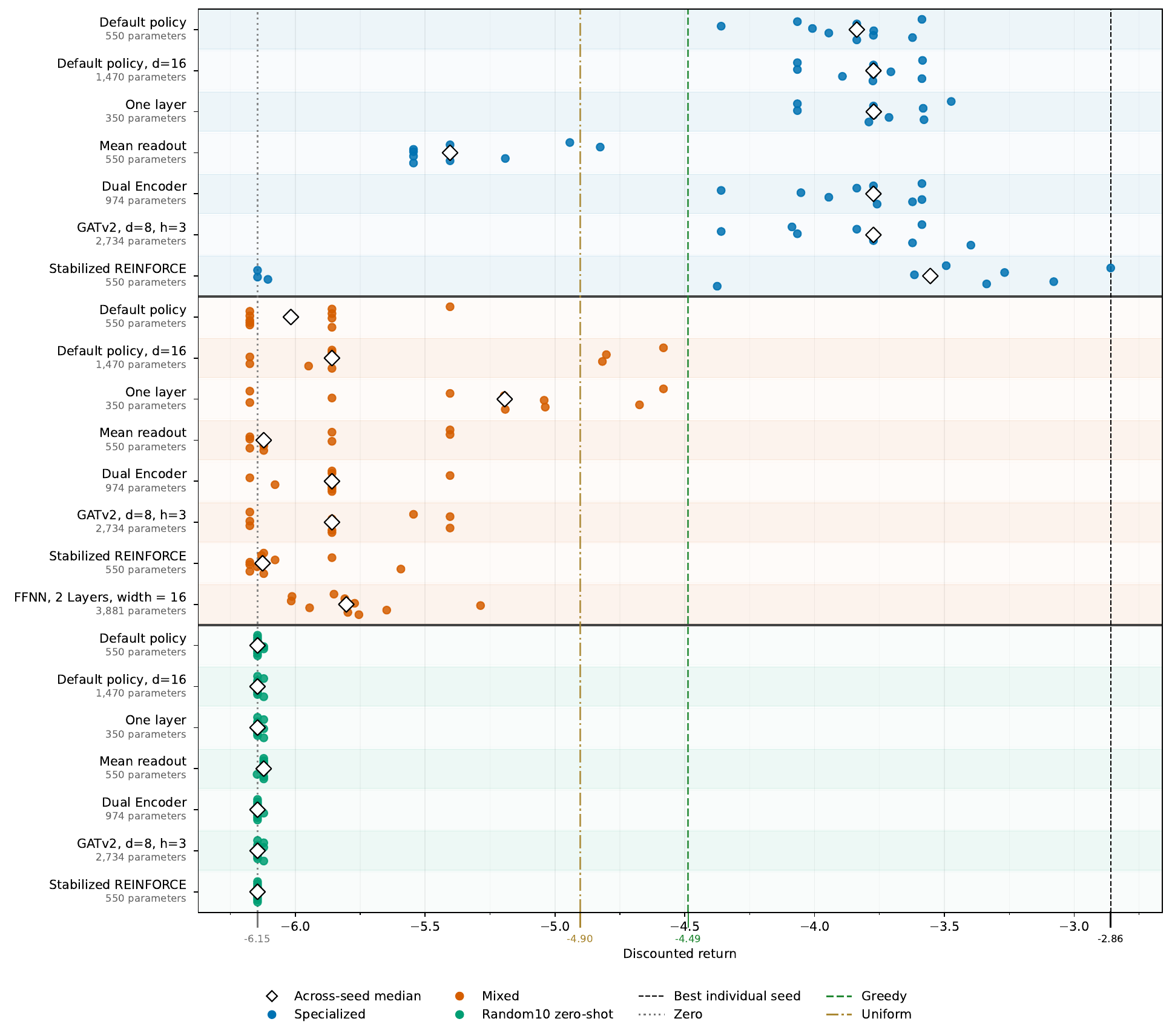}
 \caption{Inverse Star}
 \label{fig:performance-invstar10}
\end{figure}

\begin{figure}[!htb]
 \centering
 \includegraphics[width=\textwidth,height=0.76\textheight,keepaspectratio]{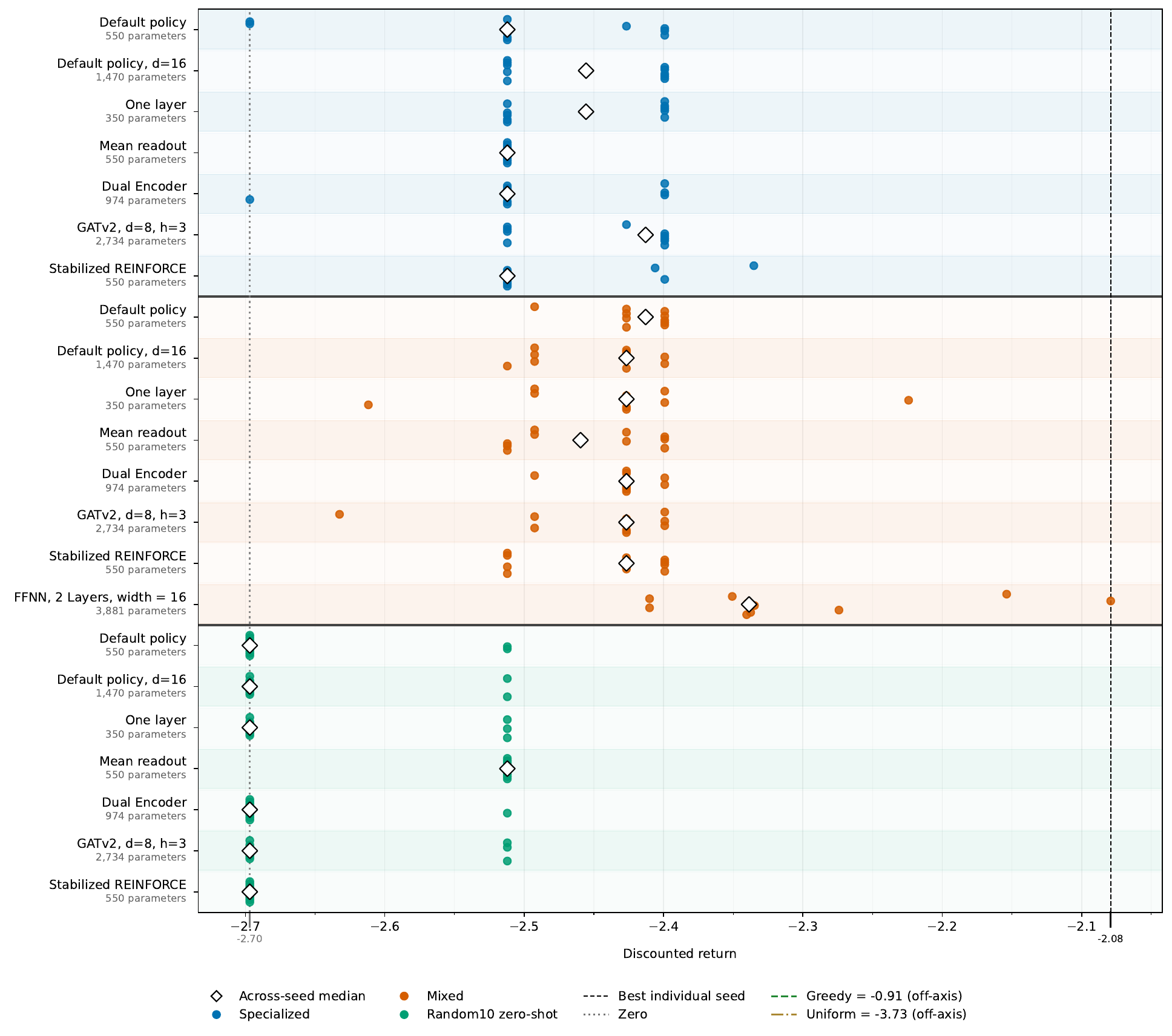}
 \caption{Star}
 \label{fig:performance-star10}
\end{figure}

\begin{figure}[!htb]
 \centering
 \includegraphics[width=\textwidth,height=0.76\textheight,keepaspectratio]{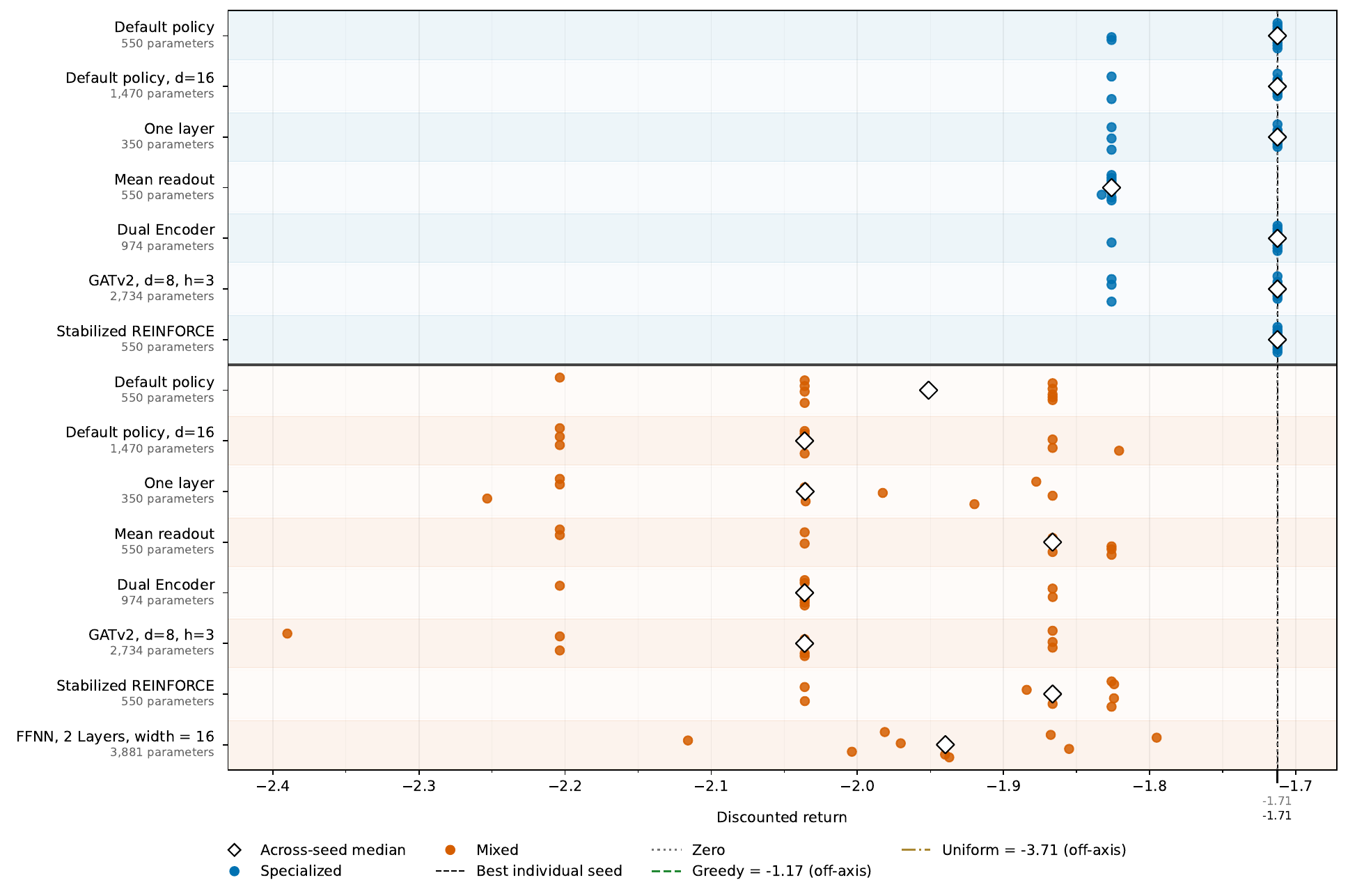}
 \caption{Random10}
 \label{fig:performance-r10}
\end{figure}

\paragraph{Overall Performance}
The heuristic policies provide a first indication of the difficulty of the intervention problems. The zero heuristic performs relatively bad on all small networks, except Random10. This confirms that intervention is generally necessary. 
Further, they highlight where the learned interventions perform better or worse. This is highly dependend on the financial network. 

On networks where multiple banks are essential for stability, the learned policies perform relatively better. On Core-Periphery, every topology-specific configuration exceeds the best heuristic in median return, while on Inverse Star this holds for all topology-specific configurations except mean readout.
The topology-specific default policy achieves median returns of $-0.82$ and $-3.84$ on Core-Periphery and Inverse Star, compared with $-1.27$ and $-4.49$ for greedy, respectively.

In contrast, on the Star network, where there is a clear, single essential bank, even the specialized models are widely outperformed by the greedy heuristic, as the topology-specific default policy achieves a median return of $-2.51$, compared with $-0.91$ for greedy.
The heuristic already identifies the essential bank. Thus, to outperform it, a learned policy must combine the correct targeting with improved choices for the timing and budget of the capital allocations. However, an in-depth look at the allocation factors for each bank reveals that the node head fails to assign a relatively high weight factor to the crucial bank. Under deterministic evaluation, all ten default models assign all banks in the initial state the maximal allocation weight factor. Paired with the small budget decisions, this leads to a poor performance relative to the greedy heuristic. As the learned policies are able to effectively recognize the essential banks in the Symmetric Stars network, this suggests that the difficulty of recognizing the critical bank in the Star network may be due to the well-known poor sample efficiency of the REINFORCE algorithms.

Increasing encoder width or replacing the MPNN by GATv2 does not systematically improve the small-network results. Interestingly, restricting to a single message-passing layer, which leads to more transparent node-wise weight decisions as they only depend on the direct neighborhood of a node, leads to no consistent performance loss. This aspect is discussed in more detail for the larger networks below. 

Although theoretically the mean readout function leads to a smaller upper bound for the global-head contribution to the score regularity, as discussed in Remark~\ref{rem:theorem-score}, this does not seem to consistently improve the small-network performance.

Changing the training regime to include more modern approaches as mentioned above can reduce the variance in performance across different training seeds, although this effect depends on the network. On Random10, only the stabilized REINFORCE version is able to learn to not intervene for all ten training seeds, which is the best learned strategy. However, no learned policy is able to outperform the greedy heuristic. As Random10 instances may vary tremendously, the results suggest that learning useful interventions across varying networks simultaneously is challenging. 
This is supported by the fact that the topology specific models outperform the multi-topology models by a lot. For the default policy, training exclusively on Random10 gives a median return of $-1.71$, compared with $-1.95$ under multi-topology training.
The stabilized REINFORCE variant achieves $-1.71$ on Random10, compared with $-1.17$ for greedy. In contrast, the learned policies perform substantially better relative to the heuristics on the larger Random200 networks, as discussed below. One possible explanation is that the structural properties of the generated networks are more similar across instances.

\paragraph{Training Across Topologies}
Surprisingly, multi-topology training achieves both the highest median return and the best individual performance among all learned policies on the Star network. The fact that topology-specific training does not yield the best performance further indicates the difficulty of learning an effective intervention strategy on this particular network. On the Core-Periphery network, as well as Random10, the performance is worse than for the specialized models, but only by a slight margin. In contrast, the multi topology models perform relatively bad on the Inverse Star network.

Despite its larger parameter count and ability to distinguish banks beyond the fairness restriction, the FFNN does not dominate the graph-based alternatives. This suggests that the structural restrictions of the graph policy are compatible with competitive performance.

In contrast, the transfer results do show that performance on previously unseen small networks is generally poor. For example, the default policy trained on Random10 achieves a median return of $-2.65$ on Core-Periphery, compared with $-0.82$ after topology-specific training. This is especially interesting, as the policies trained on Random10 are used to seeing different network configurations during training and the Core-Periphery network has a similar structure to the randomly generated networks. Mean readout is an exception, achieving a  mean return of $-1.06$ on Core-Periphery for all ten training seeds.

\subsubsection{Fairness in One-Decision Examples}
\label{subsubsec:one-decision-fairness}
\begin{figure}[!htb]
 \centering
 \includegraphics[width=\textwidth,height=0.76\textheight,keepaspectratio]{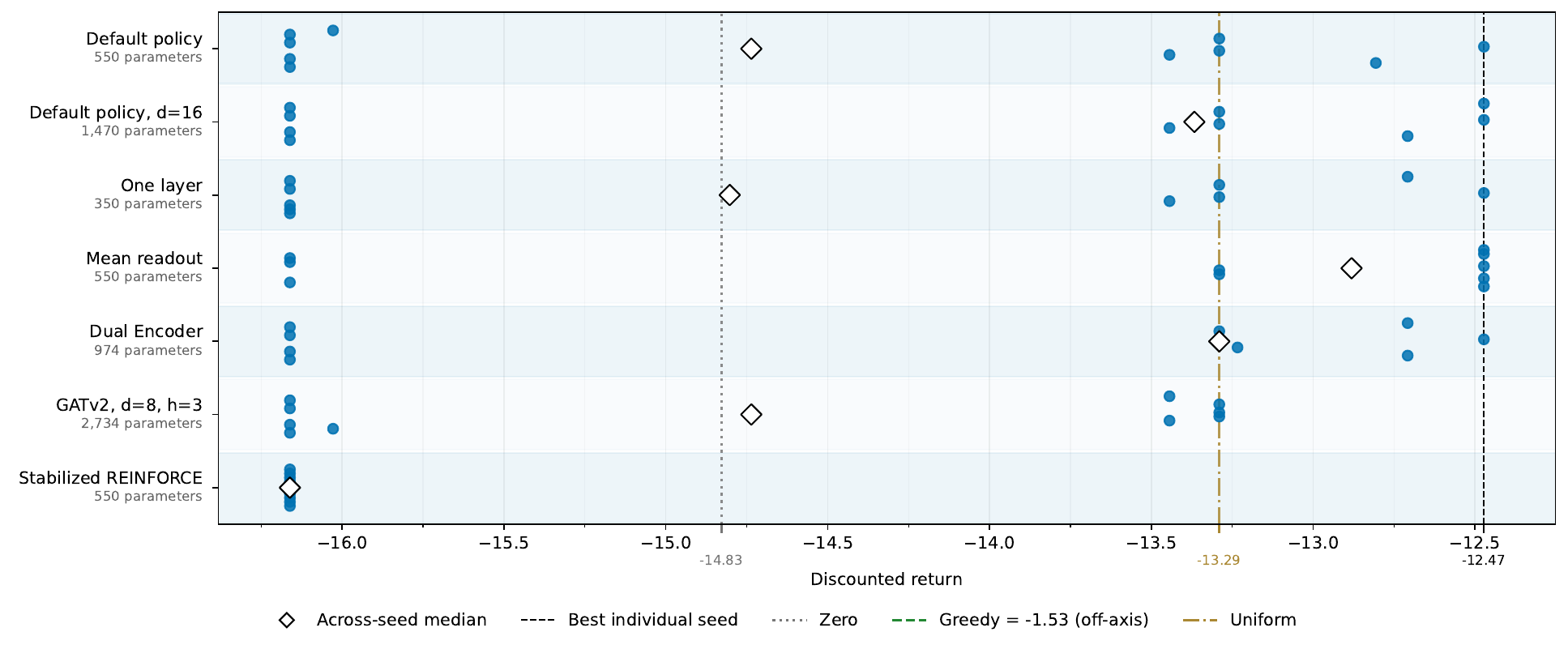}
 \caption{Symmetric Star}
 \label{fig:performance-symstars10}
\end{figure}

\begin{figure}[!htb]
 \centering
 \includegraphics[width=\textwidth,height=0.76\textheight,keepaspectratio]{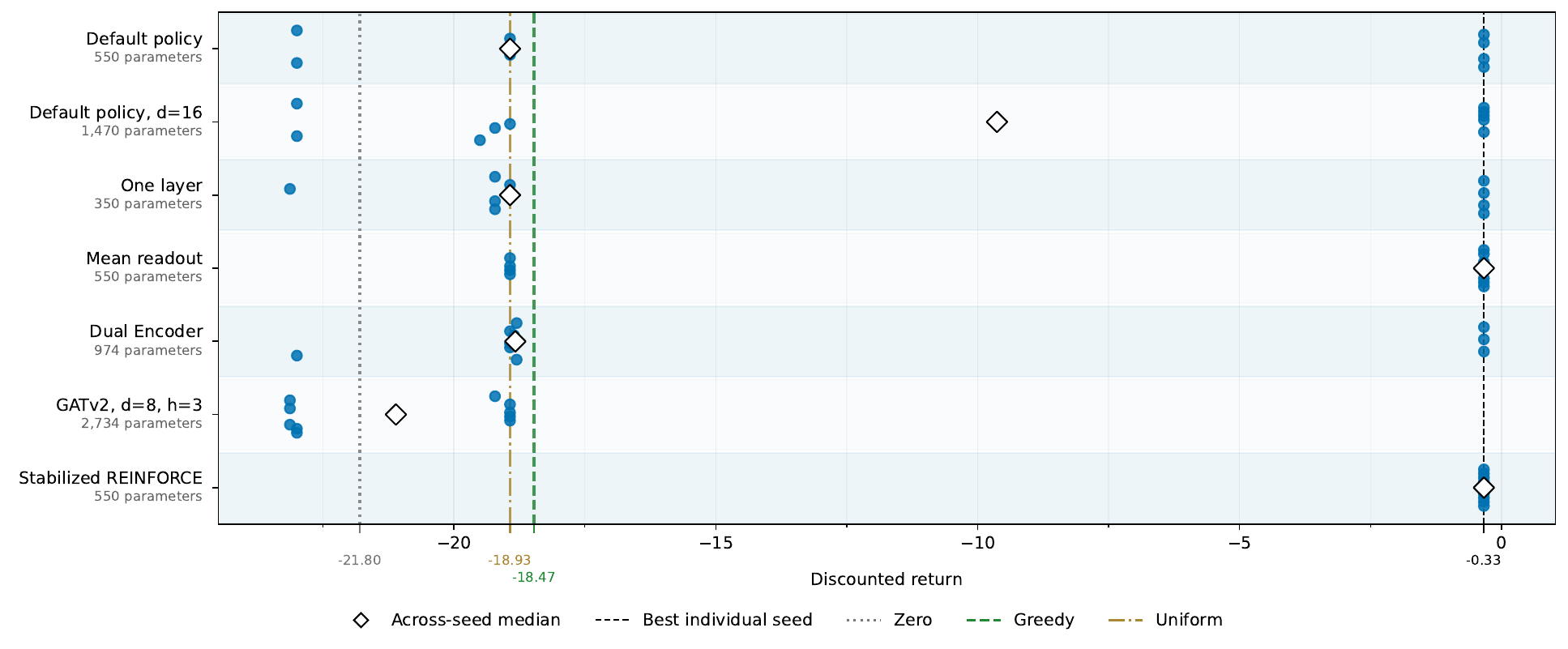}
 \caption{Mirrored Cycles}
 \label{fig:performance-mirrorcycles10}
\end{figure}

Figures~\ref{fig:performance-symstars10} and~\ref{fig:performance-mirrorcycles10} illustrate a precise limitation of fairness. To keep it more compact, the figures only include the topology-specific models. As mentioned above, in Symmetric Stars, the two initial debtors have identical features and symmetric neighborhoods. Hence our MPNN policies, which are restricted by fairness, assigns both of these the same weight  under deterministic action selection, leading to a complete cascade in the network. As the system collapse is inevitable with this limitation, the best performing learned strategies are evenly splitting the maximal available budget between the two banks, leading to a return of $-12.47$. Testing all possible fair allocations in this scenario revealed that the learned strategy is actually the best performing fair action, where fairness refers to equal treatment in the realized node-wise allocation weights, as discussed in Remark~\ref{rem:fairness}.
In contrast, the greedy heuristic finds the optimal solution and fully supports one of the two banks, achieving a return of $-1.53$. Although this strategy performs far better, under our chosen fairness restriction, our regulator shall not be able to perform it.

Stochastic evaluation of the worst performing seeds on the Symmetric Stars network exposes a substantial mismatch between the training policy and its deterministic evaluation. For example, the deterministic return of $-16.16$ of a default policy seed improves to a mean stochastic return of $-8.5$, as independent sampling assigns a positive weight to exactly one of the two relevant banks with probability $48.8\%.$ Meanwhile, the deterministic evaluation leads to no positive allocation weights paired with a positive budget decision. All seeds with a worse performance than the zero-heuristic behave similarly.

The Mirrored Cycles network distinguishes this representational limitation from a learning limitation. The effective intervention divides the budget equally between the two central banks and hence, our proposed architecture is able to represent this strategy. However, this is where one of the main drawbacks of the REINFORCE algorithm becomes visible. Although this algorithm is theoretically well understood and analyzed, it can suffer from high-variance gradient estimates, making useful allocations difficult to learn reliably. Only our stabilized REINFORCE variant learns the optimal strategy successfully in all ten training seeds. Adding the LOO baseline was enough to learn this strategy in all ten training seeds.

\subsubsection{Scaling to Larger Networks}
\label{subsubsec:scaling-larger-networks}
The Random200 experiments show that the MDP formulation and our proposed architecture is flexibel enough, to both allow for a focus on capital usage, as well as contagion losses. 
Table~\ref{tab:random-200-combined} reports the results of various model configurations trained under two differend reward formulations. Considering the one-step reward as in ~\eqref{eq:reward-per-time-step}, we distinguish between the setting with no penalty, $\lambda=0$, as in the small networks and a penalty setting with $\lambda=0.5$. Figures~\ref{fig:performance-r200-p0} and~\ref{fig:performance-r200-p05} additionally show the individual training seeds.

\providecommand{\tenmodelrow}[2]{\makecell[l]{#1\\[-0.25em]{\tiny #2 parameters}}}
\providecommand{\tenheurrow}[1]{\makecell[l]{#1\\[-0.25em]{\tiny Heuristic}}}
\providecommand{\tenreward}[1]{$#1$}
\providecommand{\tenrewardiqr}[3]{\makecell[c]{$#1$\\[-0.25em]{\tiny $[#2,#3]$}}}

\begin{table}[!htb]\centering\scriptsize
\setlength{\tabcolsep}{4pt}\renewcommand{\arraystretch}{1.08}
\caption{Evaluation on Random200.}
\label{tab:random-200-combined}
\begin{adjustbox}{max width=\textwidth,center}
\begin{tabular}{@{}l*{2}{cc}@{}}\toprule
& \multicolumn{2}{c}{$\lambda=0$}
& \multicolumn{2}{c}{$\lambda=0.5$} \\
\cmidrule(lr){2-3}\cmidrule(l){4-5}
Model & Reward & Cap./Def. & Reward & Cap./Def. \\\midrule

\tenmodelrow{Default policy}{1036}
& \tenrewardiqr{-45.87}{-79.97}{-34.25} & $1050.0/10.4$
& \tenrewardiqr{-60.99}{-83.17}{-41.53} & $750.0/16.4$ \\

\tenmodelrow{Four layers}{1436}
& \tenrewardiqr{-45.19}{-62.50}{-41.81} & $1200.0/8.2$
& \tenrewardiqr{-67.65}{-81.40}{-50.14} & $925.0/15.7$ \\

\tenmodelrow{Mean readout}{1036}
& \tenrewardiqr{-47.04}{-50.34}{-42.42} & $1325.0/6.2$
& \tenrewardiqr{-51.91}{-84.51}{-41.80} & $775.0/14.7$ \\

\tenmodelrow{Mean aggregation}{1036}
& \tenrewardiqr{-52.86}{-81.71}{-34.25} & $1525.0/4.9$
& \tenrewardiqr{-54.90}{-79.69}{-47.82} & $1100.0/12.0$ \\

\tenmodelrow{Dual Encoder 2\arraybackslash2}{1460}
& \tenrewardiqr{-64.50}{-74.34}{-45.76} & $1300.0/6.9$
& \tenrewardiqr{-68.11}{-102.13}{-61.95} & $1300.0/6.9$ \\

\tenmodelrow{Dual Encoder 4\arraybackslash1}{1660}
& \tenrewardiqr{-37.50}{-58.86}{-32.25} & $900.0/12.4$
& \tenrewardiqr{-47.01}{-64.12}{-43.28} & $1025.0/10.6$ \\

\tenmodelrow{GATv2, $d=4$, $h=2$}{1012}
& \tenrewardiqr{-73.13}{-89.28}{-48.01} & $1400.0/5.8$
& \tenrewardiqr{-63.59}{-82.74}{-60.84} & $1100.0/9.9$ \\

\tenmodelrow{Default + LOO + std. + clip + Adam}{1036}
& \tenrewardiqr{-32.69}{-43.40}{-31.37} & $325.0/46.5$
& \tenrewardiqr{-42.55}{-72.10}{-39.54} & $391.8/16.7$ \\

\tenmodelrow{FFNN, 2 layers, width 50}{2157161}
& \tenrewardiqr{-57.53}{-66.22}{-51.12} & $925.0/59.2$
& \tenrewardiqr{-86.99}{-117.52}{-79.57} & $1300.0/49.1$ \\

\midrule

\tenheurrow{Zero}
& $-56.67$ & $0.0/184.5$
& $-97.29$ & $0.0/184.5$ \\

\tenheurrow{Uniform}
& $-101.65$ & $3000.0/0.0$
& $-101.65$ & $3000.0/0.0$ \\

\tenheurrow{Greedy}
& $-250.48$ & $472.2/145.4$
& $-286.37$ & $472.2/145.4$ \\

\bottomrule
\end{tabular}
\end{adjustbox}

\par\vspace{0.35em}
\parbox{\textwidth}{\scriptsize
For learned policies, returns are reported as median $[Q_1,Q_3]$ across the ten training-seed mean returns. Cap./Def. denotes the median seed-wise mean total realized capital injection / median seed-wise mean final defaults. For the heuristics the reported values are means across the 100 evaluation episodes.
}
\end{table}

\begin{figure}[!htb]
 \centering
 \includegraphics[width=\textwidth,height=0.76\textheight,keepaspectratio]{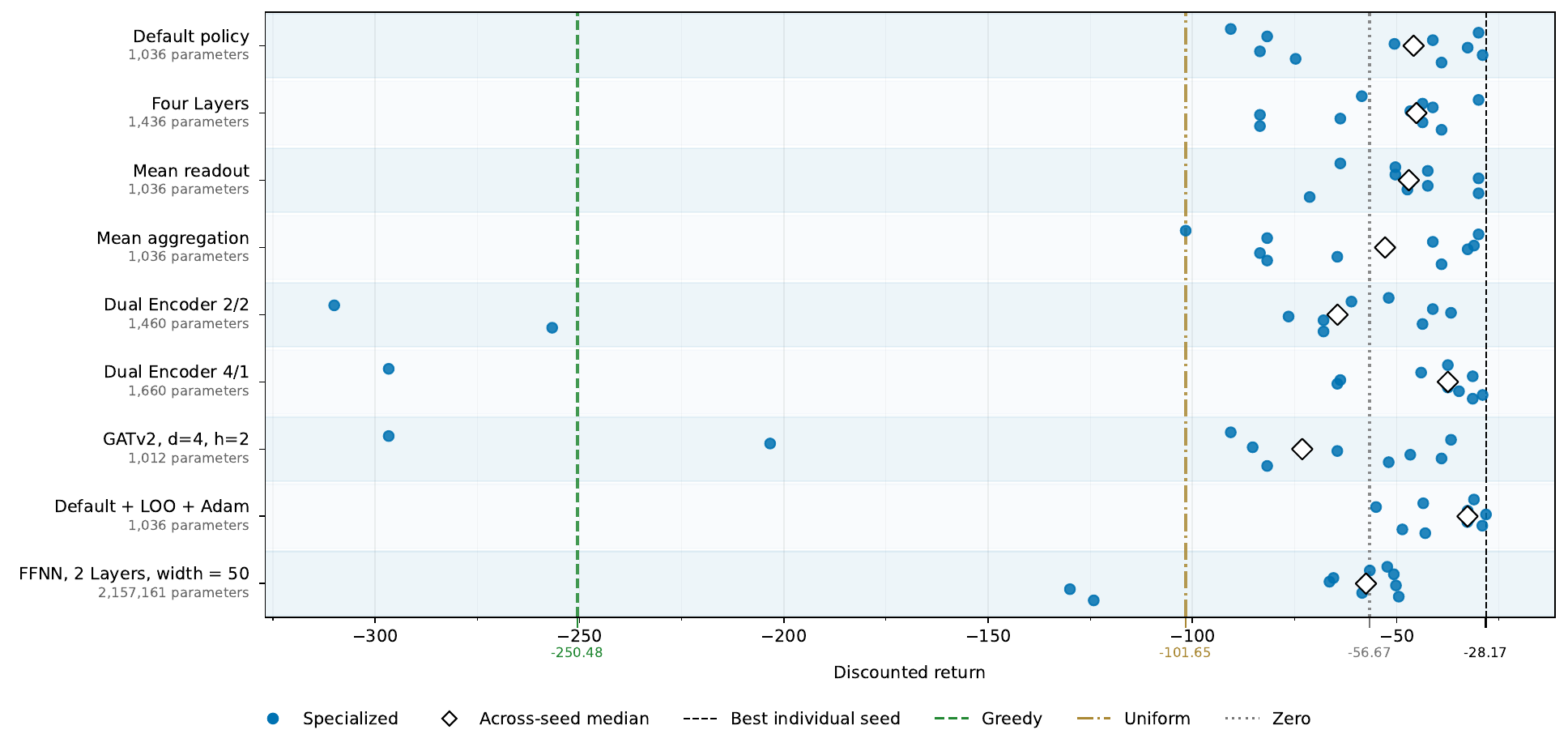}
 \caption{Random200 with $\lambda=0$.}
 \label{fig:performance-r200-p0}
\end{figure}

\begin{figure}[!htb]
 \centering
 \includegraphics[width=\textwidth,height=0.76\textheight,keepaspectratio]{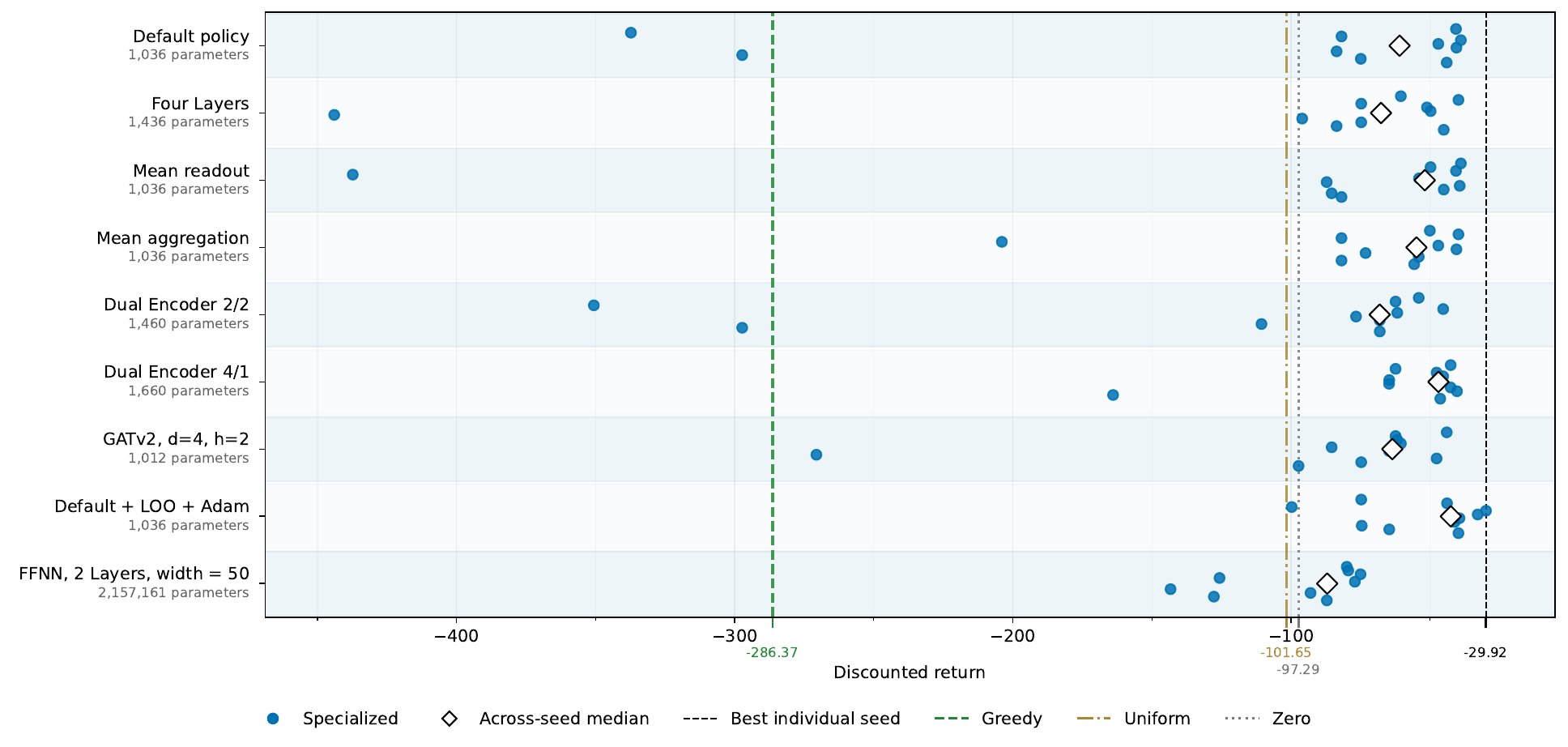}
 \caption{Random200 with $\lambda=0.5$.}
 \label{fig:performance-r200-p05}
\end{figure}      

\paragraph{Architecture and Training Reliability} 
Compared with Random10, the learned policies on Random200 perform better relative to the heuristics. This may be because, for one, Random200 instances are more similar in their structural properties across realizations than Random10 instances and, secondly, because Random200 networks are more interdependent, leading to a worse performance of the heuristics. As shown in Table~\ref{tab:random-200-combined}, blindly spending all available capital, as the uniform heuristic does, almost eliminates all defaults, while no intervention leads to an average of $184.5$ insolvent banks per episode.

Nonetheless, the zero-intervention heuristic performs best among the three heuristics in both settings. For $\lambda=0$, its return is $-56.67$, while with the added penalty of $\lambda=0.5$ its return decreases to $-97.29$. As the uniform heuristic prevents almost all defaults, its return is $-101.65$ in both settings.

Regardless of whether a penalty term is added, all policy configurations contain training seeds that outperform the best heuristic. However, only for $\lambda=0.5$ does the median return of every policy configuration exceed the return of the zero-intervention heuristic.

Simultaneously, several configurations contain individual training seeds that perform worse than all three heuristics. In these cases, the policy fails to effectively coordinate its two policy heads. 
While the budget head selects a positive intervention budget, the node head assigns zero allocation weight to all banks. Thus, no capital is injected, while the selected budget is penalized. Further training alone does not resolve this behavior, as the policy-gradient estimates toward the end of training are close to zero. Hence the corresponding policies appear to have entered a saturated regime.

The substantially larger parameter set of the FFNN and its ability to represent allocations that are excluded by our fairness restriction do not translate into better performance. In contrast, its best individual training seed achieves returns of $-49.51$ and $-75.01$ for $\lambda=0$ and $\lambda=0.5$, respectively, compared with the best MPNN-based returns of $-28.17$ and $-29.92$. At the same time, the FFNN is less affected by extreme training failures than several of the MPNN configurations. Its interquartile ranges are $15.10$ and $37.95$, and for $\lambda=0.5$ its full range across training seeds is only $68.29$, the smallest among all reported configurations.

There is also no substantial top-performance benefit from using a deeper encoder. The four-layer model reaches best-seed returns of $-29.98$ and $-39.82$, compared with $-28.99$ and $-38.94$ for the default policy. Its interquartile range is smaller in both settings, decreasing from $45.72$ to $20.69$ for $\lambda=0$ and from $41.64$ to $31.26$ for $\lambda=0.5$. Nonetheless, the four-layer configuration contains the worst performing training seed of all configurations in the $\lambda=0.5$ setting.

Although using a mean readout function or mean aggregation function reduces the Lipschitz bounds of our architecture, as seen in Remark~\ref{rem:theorem-score} and Proposition~\ref{prop:operatornorm-aggregate}, and therefore improves the constants in the finite-iteration guarantee of Theorem~\ref{theo:finite-time-constrained-stationarity}, this does not consistently lead to better empirical performance. Without a penalty, the median returns of both configurations are worse than that of the default model, while for $\lambda=0.5$, both achieve a higher median return.

The configurations using separate encoders emphasize the importance of the global budget decision. The routed configuration with four layers for the budget head and one layer for the node head achieves median returns of $-37.50$ and $-47.01$, compared with $-64.50$ and $-68.11$ for the configuration with two layers each.
As the corresponding median returns of the default policy are $-45.87$ and $-60.99$ respectively, restricting the node-wise decision to a single message-passing layer is compatible with strong performance in both settings.

The added flexibility of attention-based models does not translate into better performance. While the interquartile ranges of $41.27$ and $21.89$ are slightly smaller than the corresponding values of $45.72$ and $41.64$ for the default policy, the performance of the best individual seeds remain well below those of the default policy, with returns of $-36.73$ versus $-28.99$ for $\lambda=0$ and $-44.03$ versus $-38.94$ for $\lambda=0.5$.

The stabilized REINFORCE version achieves the highest median return among all reported model configurations in both reward settings. Its median returns are $-32.69$ and $-42.55$ for $\lambda=0$ and $\lambda=0.5$, respectively. It also avoids the extreme training failures observed for several other MPNN configurations: the ranges between its worst and best training seeds are $26.90$ and $69.83$, the smallest among all reported MPNN configurations in the respective settings. 

Moreover, the learned intervention behavior displays more efficiency in the budget distribution. Without a penalty, the median realized capital injection decreases from $1050.0$ to $325.0$, while the median number of final defaults increases from $10.4$ to $46.5$. For $\lambda=0.5$, the realized capital injection decreases from $750.0$ to $391.8$, while the median number of final defaults remains almost unchanged at $16.7$ compared with $16.4$ for the default policy.

A more in-depth analysis of the learned policies shows that these differences are largely driven by the node head. Policies trained with standard REINFORCE tend to learn approximately uniform allocation weights across solvent banks. Their performance is therefore driven primarily by the global budget decision rather than by institution-specific targeting, while the stabilized REINFORCE variant is able to learn more heterogeneous allocations. Preliminary experiments with PPO \cite{PPO} further support this observation, as PPO consistently learns heterogeneous node-level allocations. Hence, the ability to exploit the node-wise component of the architecture appears to depend strongly on the training algorithm.

\subsubsection{Scope of the Numerical Evidence}
\label{subsubsec:numerical-scope}
The experiments show that the proposed policy architecture is able to successfully learn useful financial interventions, while being restricted to local information for the node-wise allocation weights and to equal treatment of nodes with identical neighborhoods. In particular, the Random200 experiments show that strong performance is possible even when the node-wise weights depend only on the direct neighborhood of a bank. At the same time, the learned policies indicate that much of the performance in this setting can already be achieved through the global budget decision, while learning targeted node-wise allocations appears to be considerably more challenging.

Overall, these experiments remain a proof of concept. The small-network experiments show that the architecture is able to identify essential nodes, but that learning the corresponding intervention strategy can require substantial training. The Random200 experiments similarly indicate that standard REINFORCE does not fully exploit the node-wise allocation mechanism, whereas the stabilized variant and preliminary PPO experiments demonstrate that more targeted allocations can be learned and are beneficial.
The variation across seeds and performance difference compared to the stabilized REINFORCE algorithm indicate that there is room for improvement on the training side. A more systematic comparison of modern training algorithms, including actor-critic approaches such as PPO, as well as an analysis of the stochastic--deterministic evaluation gap and its relation to targeted intervention strategies, is left for future work.

\section*{Acknowledgments}
Computational support from the Zentrum für Informatoins- und Medientechnologie (ZIM) at Heinrich Heine University is gratefully acknowledged. We thank seminar participants at the Financial Mathematics Seminar
at UCLA, the CFMAR Seminar at UCSB, the Control and Optimization
Seminar at the University of Connecticut, and the Financial Math
Seminar at Worcester Polytechnic Institute, as well as participants
in the 2026 SIAM Conference on Optimization, the Swissquote Workshop on Graph Learning in Financial Networks and the 2026 Bachelier
Finance Society World Congress, for their helpful comments and
feedback.
\FloatBarrier

\bibliographystyle{plain}
\bibliography{references}

\section{Appendix}
\label{app:appendix}

This appendix collects the technical material required for the robustness results in Section~\ref{sec:robustness-convergence}. We first rewrite the message-passing architecture in matrix form, then give a formal representation of the policy heads, and finally record the derivative identities used in the proofs. Afterwards, we prove the aggregation bounds, the layer-wise Lipschitz estimates, the inductive MPNN bounds, the score function bounds, and the convergence results.

\subsection{Formal architecture and matrix representation}
\label{app:architecture-matrix}

We now formalize the definition of the policy and introduce additional notation. Since our Lipschitz-constant results require an explicit linear representation of the message-passing updates, we express each MPNN layer in matrix form. 

\subsubsection{Matrix representation of the MPNN}
\label{app:mpnn-matrix}

Throughout this subsection, we distinguish between an MDP state $s\in\mathcal S$ and the graph input used by the MPNN. For a state $s=(X,E)\in\mathcal S$, the graph input is given by the weighted adjacency matrix $E$ and the initial node features
$$x^{(0)}:=
\begin{pmatrix}
X_1\\
\vdots\\
X_n
\end{pmatrix}
\in \R^{nd_0}.$$
When a single graph input is fixed, we write $E$ for the weighted adjacency matrix and denote its entries by $e_{ij}=E_{ij}$.

Let $k\geq 1$. Then, for a fixed weighted adjacency matrix $E$, the message in the $k$-th layer is given by
\begin{align*}
v_i^{(k)} &:= \begin{pmatrix} x_i^{(k-1)}\\
\bigoplus_{j \in \mathcal{N}^+(i)} e_{j,i}x_j^{(k-1)}\\
\bigoplus_{j \in \mathcal{N}^-(i)} e_{i,j}x_j^{(k-1)}
\end{pmatrix} \in \R^{3d_{k-1}}.
\end{align*}
Next, we apply an affine map followed by an activation function to obtain the new node features.
\begin{align*}
z_i^{(k)} &:= W^{(k)} v_i^{(k)} + b^{(k)}, \qquad
x_i^{(k)}:= \sigma \odot z_i^{(k)},
\end{align*}
where $W^{(k)} \in \mathbb{R}^{d_{k}\times 3d_{k-1}}$ and $b^{(k)} \in \mathbb{R}^{d_{k}}$.

Since the message functions and aggregations are linear in the node features, the update of a node can be written in linear form. For $i,j \in \{1,\ldots,n\}$, we therefore define matrices ${A^{(k)}_{\operatorname{mean}}}_{i,j}(E)$ such that
\begin{align*}
v_i^{(k)} = \sum_{j=1}^n {A^{(k)}_{\operatorname{mean}}}_{i,j}(E) x_j^{(k-1)}.
\end{align*}
They are given by
\begin{align*}
{A^{(k)}_{\operatorname{mean}}}_{i,j}(E) = \begin{cases}
\begin{pmatrix}
I_{d_{k-1}}\\
0_{2d_{k-1} \times d_{k-1}}
\end{pmatrix}, & \text{if } j = i, \\[1.5em]
\begin{pmatrix}
0_{d_{k-1}\times d_{k-1}}\\
\1_{|\mathcal{N}^+(i)|>0}\frac{e_{j,i}}{|\mathcal{N}^+(i)|}\, I_{d_{k-1}}\\
\1_{|\mathcal{N}^-(i)|>0}\frac{e_{i,j}}{|\mathcal{N}^-(i)|}\, I_{d_{k-1}}
\end{pmatrix}, & \text{if } j\neq i.
\end{cases}
\end{align*}
These matrices correspond to choosing $\operatorname{mean}$ as aggregation function. When we choose $\operatorname{sum}$, the matrix changes to
\begin{align*}
{A^{(k)}_{\sum}}_{i,j}(E) = \begin{cases}
\begin{pmatrix}
I_{d_{k-1}}\\
0_{2d_{k-1} \times d_{k-1}}
\end{pmatrix}, & \text{if } j = i, \\[1.5em]
\begin{pmatrix}
0_{d_{k-1}\times d_{k-1}}\\
e_{j,i} I_{d_{k-1}}\\
e_{i,j} I_{d_{k-1}}
\end{pmatrix}, & \text{otherwise.}
\end{cases}
\end{align*}
When we do not specify which aggregation function is used, we write $A^{(k)}_{i,j}(E)$. In both cases, the matrix describes the influence of node $j$ on the update of node $i$ and depends only on the weighted adjacency matrix $E$. In every layer, the same matrix is applied to each feature coordinate.
This approach could be extended to more general linear message functions. More precisely, instead of messages of the form $e_{ij}x_j^{(k-1)}$, one may consider messages of the form
$$M_{i,j}(e_{ij})\, x_j^{(k-1)},$$
where $M_{i,j}(e_{ij})$ is a linear map that may depend on $i$, $j$, and $e_{ij}$. In that case, the matrices $A^{(k)}_{i,j}(E)$ are obtained by replacing the two lower blocks by $M_{i,j}(e_{ij})$.

We denote by $x^{(k)}$ the concatenation of every node's features into one vector, i.e.
\begin{align*}
x^{(k)} := \begin{pmatrix} x_1^{(k)} \\ \vdots \\ x_n^{(k)}\end{pmatrix} \in \R^{nd_{k}},
\end{align*}
and similarly define $v^{(k)}$ and $z^{(k)}$.
Denoting the weights by $\varphi^{(k)} \in \R^{d_{k}(3d_{k-1}+1)}$, layer $k$ can be represented by the update function
\begin{align*}
U^{(k)}: \R^{d_{k}(3d_{k-1}+1)} \times \R^{nd_{k-1}} \times \R^{n\times n} \to \R^{nd_{k}},
\end{align*}
given by
\begin{align*}
U^{(k)}(\varphi^{(k)}, x^{(k-1)}, E)
:= \left(I_n \otimes W^{(k)} \right) A^{(k)}(E) x^{(k-1)} + \mathds{1}_n \otimes b^{(k)},
\end{align*}
where $A^{(k)}(E)$ is the block matrix
\begin{align*}
A^{(k)}(E) :=
\begin{pmatrix}
A^{(k)}_{1,1}(E) & \cdots & A^{(k)}_{1,n}(E)\\
\vdots & \ddots & \vdots\\
A^{(k)}_{n,1}(E) & \cdots & A^{(k)}_{n,n}(E)
\end{pmatrix}.
\end{align*}
The update in the first layer differs only in that $x^{(0)}(s)\in\R^{nd_0}$ is part of the given state. Therefore, the dimensions of the matrices have to be adjusted accordingly.

In the following, we denote all parameters up to the $k$-th layer by $\theta^k=(\varphi^{(k)},\ldots,\varphi^{(1)})$ and denote the activated output of a layer by $T^{(k)} := \sigma \odot U^{(k)}$. Then the whole MPNN can be understood iteratively as the function
\begin{align*}
f^K: \theta^K \times \R^{n d_0} \times \R^{n \times n} \to \R^{nd_{K}},\quad
(\varphi^{(K)},\ldots, \varphi^{(1)}, x^{(0)}, E)
\mapsto T^{(K)}(\varphi^{(K)}, f^{K-1}(\theta^{K-1},x^{(0)},E), E),
\end{align*}
where $f^{K-1}$ denotes the MPNN up to the $(K-1)$-th layer and $f^0(\theta^0,x^{(0)},E):=x^{(0)}$.

\subsubsection{Policy heads}
\label{app:policy-heads}
Following Definition~\ref{def:heads}, we now give a formal representation of the global and local policy heads.

\noindent\textbf{Global head}: For the global head in Definition~\ref{def:heads}, we first consider $\operatorname{sum}$ as the readout function. Thus, we first sum the node embeddings:
$$\mathbf{x} := \read x := \sum_{i=1}^n x_i^{(K)} \in \R^{d_{K}},$$
where $\read \in \R^{d_{K} \times n d_{K}}$ is the summation matrix
\begin{align*}
\read = \left[ I_{{d_{K}}}  I_{d_{K}} \ldots I_{d_{K}}  \right] = \mathds{1}_n^\top \otimes I_{d_{K}} \in \R^{{d_{K}} \times n{d_{K}}}.
\end{align*}
The global representation is then mapped to $K_G:=|\mathcal A_{\mathrm{glob}}|$ logits by an affine map with parameters $W^G \in \R^{K_G \times {d_{K}}}$ and $b^G \in \R^{K_G}$.
$$z^G = W^G \mathbf{x} + b^G.$$
After fixing an ordering of the finite action set $\mathcal A_{\mathrm{glob}}$, the softmax function yields a probability distribution on $\mathcal A_{\mathrm{glob}}$.
$$\operatorname{softmax}: \R^{K_G} \to \mathcal{D}^{K_G}, \qquad
z \mapsto \left(\frac{\exp(z_i)}{\sum_{k=1}^{K_G} \exp(z_k)}\right)_{i=1,\ldots,K_G},$$
where $\mathcal{D}^{K_G} := \{x \in \R^{K_G} \mid x \geq 0, \sum_{i=1}^{K_G} x_i = 1\}$.
As before, we understand the weights $\tau^G = (W^G,b^G)$ as a vector and denote the affine transformation by
\begin{align*}
U^{G}: \R^{K_G({d_{K}}+1)} \times \R^{n {d_{K}}} \to \R^{K_G}, \qquad
(\tau^G, x^{(K)}) \mapsto W^G \read x^{(K)} + b^{G}.
\end{align*}
Hence, writing $\pi_\theta^G=\pi_\theta^{\mathrm{glob}}$, the global head is the function
$$\pi^G: \R^{K_G({d_{K}}+1)} \times \R^{n {d_{K}}} \times \mathcal S \times \mathcal A_{\mathrm{glob}} \to [0,1],\qquad
\pi_\theta^G(a_{\mathrm{glob}} \mid s)
=(\operatorname{softmax}\circ U^G(\tau^G,x^{(K)}))_{a_{\mathrm{glob}}}.$$

\noindent\textbf{Local heads}: The architecture of the local head differs slightly. For each node $i$, a vector of logits is computed from the local embedding $x_i$.
$$z^{N_i} = W^N x_i^{(K)} + b^N = W^N P_i x^{(K)} + b^N,$$
where $P_i \in \R^{{d_{K}} \times n {d_{K}}}$ is the projection matrix for the $i$-th node,
$$P_i =  \left[0_{d_{K}} \ldots 0_{d_{K}} I_{d_{K}} 0_{d_{K}} \ldots\right] \in \R^{{d_{K}} \times n{d_{K}}},$$
with the identity matrix in the $i$-th block.

Before applying softmax in the local head, we apply masking to ensure that the resulting distribution is supported on the admissible node-level action set $\mathcal A_{\mathrm{node},i}(s)$. The local head first produces logits for the full finite set $\mathcal A_{\mathrm{node}}$ and then masks inadmissible actions. For the theoretical analysis, this is done by setting the corresponding logits to $-\infty$ and using the convention $\exp(-\infty)=0$. Since $a_{\mathrm{node}}^0\in\mathcal A_{\mathrm{node},i}(s)$ for all states and nodes, the denominator of the softmax remains strictly positive.

Formally, after fixing an ordering of $\mathcal A_{\mathrm{node}}$, masking for node $i$ is the function
\begin{align*}
\operatorname{mask}^i_s&:\R^{K_{\mathrm{node}}}\to\left(\R\cup\{-\infty\}\right)^{K_{\mathrm{node}}},\\
\left(\operatorname{mask}^i_s(z)\right)_r &=
\begin{cases}
z_r, & \text{if the $r$-th node action belongs to } \mathcal A_{\mathrm{node},i}(s),\\
-\infty, & \text{otherwise.}
\end{cases}
\end{align*}
In numerical implementations, this masking is approximated by replacing $-\infty$ with a large negative value $-M$. We stress that for fixed $s$, the masked softmax is differentiable as a function of the logits and refer to \cite{analysis_masking} for an analysis of the influence of masks on the policy and its gradient.

Defining $U^N$ and $\tau^N$ analogously to the global head, the application of the shared local head to node $i$ can be represented by
\begin{align*}
\pi_\theta^i:
\R^{K_\mathrm{node}({d_{K}}+1)} \times \R^{n{d_{K}}} \times \mathcal S \times \mathcal A_{\mathrm{node}} \to [0,1],
\quad \pi_\theta^i(a_i \mid s)
= (\operatorname{softmax}\circ \operatorname{mask}^i_s\circ U^N(\tau^N,x^{(K)}))_{a_i}.
\end{align*}
For a given state $s$, its output is a probability distribution on $\mathcal A_{\mathrm{node},i}(s)$.

With a slight abuse of terminology, we refer to $\pi_\theta^i$ as the $i$-th local head. 
As all local heads share the weights $\tau^N$, our architecture is permutation equivariant. The complete policy is given by~\eqref{eq:policy-def}.

\subsubsection{Derivative identities}
\label{app:derivative-identities}
To prove our main results from Section~\ref{sec:robustness-convergence}, we first compute derivatives of a single MPNN layer as well as derivatives of the heads.
For a given weight matrix $W \in \R^{m \times n}$ and a bias vector $b \in \R^m$, we identify the parameter space with a vectorized subset of $\R^{m(n+1)}$, where parameters are ordered row-wise, i.e.,
$$W_{1,1}, W_{1,2},\ldots,b_1,W_{2,1},\ldots, b_m.$$

Furthermore, we work with Jacobian matrices and denote the partial derivative with respect to the vectorized weight matrix $W$ by $D_W$. The same convention applies to $D_b$.
In the following, we fix a state-action pair $(s,a) \in \mathcal S\times \mathcal A(s)$. Hence, we neglect the masking of actions, as these effectively only changes the size of the action space, which has no influence on our bounds. Similarly, all constants with respect to changes in the inputs consider sets with fixed masks and thus we also neglect the  masks in that case.

\noindent\textbf{Derivative of the global head}: By construction, for a given graph-level action $a_{\mathrm{glob}}$ and state $s$, we have
$$\log \pi_\theta^G(a_{\mathrm{glob}} \mid s)
= z^G_{a_{\mathrm{glob}}}-\log\!\Biggl(\sum_{k=1}^{K_G} \exp\big(z^G_k\big)\Biggr).$$
The derivative with respect to the input $x$ is given by
\begin{align}
D_x \log\pi_\theta^G (a_{\mathrm{glob}} \mid s)
=&\left( \frac{\partial \log \pi_\theta^G(a_{\mathrm{glob}}\mid s)}{\partial z^G} \right)
\left( \frac{\partial z^G}{\partial x} \right) \nonumber\\
=&
\left(e_{a_{\mathrm{glob}}} - \pi_\theta^G(\cdot \mid s)\right)^\top W^G\read,
\label{eq:graph-head-input}
\end{align}
where $e_{a_{\mathrm{glob}}}$ is the unit vector corresponding to the chosen class and $\pi_\theta^G(\cdot \mid s)$ is the vector of probabilities over all graph-level actions.
As
\begin{align*}
\frac{\partial U^G(\tau^G,x)}{\partial W^G} =  I_{K_G} \otimes (\read{x})^\top \quad \text{and} \quad \frac{\partial U^G(\tau^G,x)}{\partial b^G} = I_{K_G},
\end{align*}
the derivative of the global head with respect to its parameters is given by
\begin{align}
D_{\tau^G} \log\pi_\theta^G (a_{\mathrm{glob}} \mid s)
=&\left(\frac{\partial\log \pi_\theta^G(a_{\mathrm{glob}} \mid s)}{\partial z^G}\right)
\left(\frac{\partial z^G}{\partial \tau^G}\right) \nonumber\\
=&\left(e_{a_{\mathrm{glob}}} - \pi_\theta^G(\cdot \mid s)\right)^\top
\left(I_{K_G} \otimes [(\read{x})^\top\,1]\right),
\label{eq:graph-head-parameter}
\end{align}
where the parameters are ordered as
$$W^G_{1,1},\ldots, W^G_{1,d_K}, b^G_1, W^G_{2,1},\ldots,b^G_{K_G}.$$

\noindent\textbf{Derivative of the local heads}: Without loss of generality, we consider $i=1$.

Analogously to the global head,
$$\log\pi_\theta^1(a_1\mid s) = z^{N_1}_{a_1} - \log\left( \sum_{k=1}^{K_\text{node}}\exp\left(z^{N_1}_k\right)\right),$$
and the derivative with respect to its input is given by
\begin{align}
D_x\log \pi_\theta^1({a_1} \mid s)
=&\left(\frac{\partial\log \pi_\theta^1(a_1 \mid s)}{\partial z^{N_1}} \right)\left(\frac{\partial z^{N_1}}{\partial x}\right) \nonumber\\
=&\left(e_{a_1} - \pi_\theta^1(\cdot \mid s)\right)^\top W^N P_1.
\label{eq:node-input}
\end{align}
The derivative with respect to its parameters is given by
\begin{align}
D_{\tau^N} \log\pi_\theta^1 (a_1 \mid s)
=&\left(\frac{\partial\log \pi_\theta^1(a_1 \mid s)}{\partial z^{N_1}}\right)
\left(\frac{\partial z^{N_1}}{\partial \tau^N}\right) \nonumber\\
=&\left(e_{a_1} - \pi_\theta^1(\cdot \mid s)\right)^\top
\left(I_{K_\mathrm{node}} \otimes \left[(P_1{x})^\top\,1\right]\right).
\label{eq:node-parameter}
\end{align}

Combining the graph-level and node-level identities with the gated policy factorization in \eqref{eq:policy-def} gives
$$\nabla_\theta\log\pi_\theta(a\mid s)
=\nabla_\theta\log\pi_\theta^G(a_{\mathrm{glob}}\mid s)+\eta(a_{\mathrm{glob}})\sum_{i=1}^n\nabla_\theta\log\pi_\theta^i(a_i\mid s).$$
Thus, whenever $\eta(a_{\mathrm{glob}})=0$, the node-level derivative contributions vanish.

\noindent\textbf{Derivative of one MPNN layer}: We now derive, for a given layer $k\in \{1,\ldots,K\}$, the derivatives with respect to the input and the parameters. In these derivatives, the weighted adjacency matrix $E$ is fixed.
\begin{align}
D_x T^{(k)}(\varphi^{(k)},x,E)
=&D_x \left(\sigma \odot U^{(k)}(\varphi^{(k)},x,E)\right) \nonumber\\
=&S^{(k)} D_x U^{(k)}(\varphi^{(k)},x,E) \nonumber\\
=&S^{(k)} \left(I_n \otimes W^{(k)}\right)A^{(k)}(E),
\label{eq:layer-input}
\end{align}
where
$$S^{(k)}:= \operatorname{Diag}\bigl(\sigma' \odot U^{(k)}(\varphi^{(k)},x,E)\bigr)\in \mathbb{R}^{n {d_{k}} \times n {d_{k}}}$$
denotes the diagonal matrix of the derivatives of the activation function evaluated at $U^{(k)}(\varphi^{(k)},x,E)$.

Using $v^{(k)} = A^{(k)}(E)x$ gives
\begin{align*}
D_{W^{(k)}} U^{(k)}(\varphi^{(k)},x,E)
=&D_{W^{(k)}} \left((I_{n} \otimes  W^{(k)}) A^{(k)}(E) x + \mathds{1}_n \otimes b^{(k)}\right)\\
=&D_{W^{(k)}} \left((I_{n} \otimes  W^{(k)}) v^{(k)}\right)\\
=&D_{W^{(k)}} \begin{pmatrix} W^{(k)} v_1^{(k)} \\ \vdots \\ W^{(k)} v_n^{(k)} \end{pmatrix}
=\begin{pmatrix} I_{{d_{k}}} \otimes {v_1^{(k)}}^\top \\ \vdots \\ I_{{d_{k}}} \otimes {v_n^{(k)}}^\top \end{pmatrix}.
\end{align*}
Furthermore,
\begin{align*}
D_{b^{(k)}} U^{(k)}(\varphi^{(k)},x,E)
=&D_{b^{(k)}}\left((I_{n} \otimes  W^{(k)}) A^{(k)}(E) x + \mathds{1}_n \otimes b^{(k)}\right)\\
=&D_{b^{(k)}} \left(\mathds{1}_n \otimes b^{(k)}\right)\\
=&\mathds{1}_n \otimes I_{{d_{k}}}.
\end{align*}
As before, by stacking the derivatives with respect to $W^{(k)}$ and $b^{(k)}$ in the same order as $\varphi^{(k)}$, we can write this compactly as
\begin{align*}
D_{\varphi^{(k)}} U^{(k)}(\varphi^{(k)},x,E)
=\begin{pmatrix}
I_{{d_{k}}} \otimes \left[{v_1^{(k)}}^\top \,1\right]\\
\vdots\\
I_{{d_{k}}} \otimes \left[{v_n^{(k)}}^\top \,1\right]
\end{pmatrix}.
\end{align*}
Hence, the Jacobian with respect to the parameters is given by
\begin{align}
D_{\varphi^{(k)}} T^{(k)}(\varphi^{(k)},x,E)
&=S^{(k)} D_{\varphi^{(k)}} U^{(k)}(\varphi^{(k)},x,E) \nonumber\\
&=S^{(k)}
\begin{pmatrix}
I_{{d_{k}}} \otimes \left[{v_1^{(k)}}^\top \,1\right]\\
\vdots\\
I_{{d_{k}}} \otimes \left[{v_n^{(k)}}^\top \,1\right]
\end{pmatrix}.
\label{eq:layer-parameter}
\end{align}
\subsection{Proofs for Subsection~\ref{subsec:Lipschitz}}
\label{app:proof-lipschitz}
Throughout this appendix, we present the proofs for the shared-encoder architecture. These results can be easily adapted to the separate-encoder case. Since the Lipschitz estimates are stated at the level of asymptotic order, applying the same bounds to identical, independent encoders only changes constant prefactors.

\subsubsection{Preliminaries and aggregation bounds}
\label{app:aggregation-bounds}

For the proofs of our results from Section~\ref{sec:robustness-convergence}, we repeatedly rely on the following norm relations and inequalities for matrices $A,B$ and vectors $c,d$.
\begin{align*}
\norm*{A\otimes B} =& \norm*{A}\norm*{B}, \\    
\norm*{\begin{pmatrix}  A \\B   \end{pmatrix}} \leq& \sqrt{\norm*{A}^2 + \norm*{B}^2}, \\
\norm*{[A\, B]} \leq& \sqrt{\norm*{A}^2 + \norm*{B}^2}, \\
\norm*{c^\top}_{\operatorname{op}} =& \norm*{c}, \\
\norm*{\begin{pmatrix} c\\d \end{pmatrix}} =& \sqrt{\norm*{c}^2+\norm*{d}^2},\\
\norm*{A} \leq& \norm*{A}_F,
\end{align*}
where $\norm*{\cdot}_F$ denotes the Frobenius norm.

For the main results in Section~\ref{sec:robustness-convergence}, we frequently use of the Cauchy--Schwarz inequality in the form
\begin{align*}
\left(\sum_{i=1}^n a_i\right)^2 \leq n \sum_{i=1}^n a_i^2 
\end{align*} 
for $a_1,\ldots,a_n \in \R$. We also use the following well-known lemma.
\begin{lemma}\label{lemm:Schrankensatz}
Let $U\subset\R^l$ be open and $\mathcal{X}\subset U$ convex. Further, let $f:U\to \R^m$ be a differentiable function with
$\sup_{x\in \mathcal{X}}\norm*{Df(x)} < \infty$. Then the restriction $f|_\mathcal{X}$ is Lipschitz continuous with Lipschitz constant $\sup_{x\in \mathcal{X}}\norm*{Df(x)}$.
\end{lemma}
\begin{proof}
This lemma is Proposition C.29 in \cite{Lee2002IntroductionTS}. 
\end{proof}

\noindent\textbf{Aggregation bounds}: We are now in a position to prove Proposition~\ref{prop:operatornorm-aggregate}.

\begin{proof}[Proof of Proposition~\ref{prop:operatornorm-aggregate}]
Let $E\in\R^{n\times n}_{\geq 0}$ be a weighted adjacency matrix with entries $e_{ij}=E_{ij}$ and no self-loops. If $E=0$, both neighbor aggregates vanish and hence
\begin{align*}
\norm*{A^{(k)}_{\operatorname{mean}}(E)} =\norm*{A^{(k)}_{\sum}(E)}=1.
\end{align*}
Now assume in the following that the graph has at least one edge. To obtain bounds for the operator norm of $A^{(k)}_{\operatorname{mean}}(E)$, we define matrices $M^+, M^- \in \R^{n\times n}$ by
\begin{align*}
M^+_{ij}:= \1_{|\mathcal{N}^+(i)|>0} \frac{e_{j,i}}{|\mathcal{N}^+(i)|}, \quad M_{ij}^-:= \1_{|\mathcal{N}^-(i)|>0} \frac{e_{i,j}}{|\mathcal{N}^-(i)|}.
\end{align*}
Then
\begin{align*}
A^{(k)}_{\operatorname{mean}}(E) = P \left( \begin{pmatrix}I_n \\ M^+ \\ M^-\end{pmatrix} \otimes I_{d_{k-1}}\right)   
\end{align*} 
for a suitable permutation matrix $P \in \{0,1\}^{3n{d_{k-1}} \times 3n{d_{k-1}}}$.
Hence, the operator norm is bounded above by
\begin{align*}
\norm*{A^{(k)}_{\operatorname{mean}}(E)}^2 &=  \norm*{\begin{pmatrix}I_n \\ M^+ \\ M^-\end{pmatrix}}^2 \leq 1 + \norm*{M^+}^2 + \norm*{M^-}^2
\end{align*}
and bounded below by
\begin{align*}
\norm*{A^{(k)}_{\operatorname{mean}}(E)}^2 &\geq \max \{1+\norm*{M^+}^2, 1+\norm*{M^-}^2\}.
\end{align*}
Since the product of the row-sum norm and the column-sum norm is an upper bound for the spectral norm, we further obtain
\begin{align*}
\norm*{M^+}^2 &\leq \max_{i \in [n]} \left(\sum_{j \in [n]} |M_{ij}^+|\right) \max_{j\in[n]} \left( \sum_{i \in [n]} |M_{ij}^+|\right) \leq e_{\max}^2 \frac{\Delta^-}{\delta^+}.
\end{align*} 
A lower bound is obtained from
\begin{align*}
\norm*{M^+}^2 &\geq \max\left\{\max_{j \in [n]} \norm*{\begin{pmatrix} M_{1j}^+ \\ \vdots \\ M_{nj}^+ \end{pmatrix}}^2, \max_{i \in [n]} \norm*{\begin{pmatrix} M_{i1}^+ \\ \vdots \\ M_{in}^+ \end{pmatrix}}^2 \right\} \\
&=  \max\left\{ \max_{j \in [n]} \sum_{i=1}^n \1_{|\mathcal{N}^+(i)|>0} \frac{e_{j,i}^2}{|\mathcal{N}^+(i)|^2}, \max_{i \in [n]} \sum_{j=1}^n \1_{|\mathcal{N}^+(i)|>0} \frac{e_{j,i}^2}{|\mathcal{N}^+(i)|^2} \right\} \\
&\geq \max\left\{\max_{j\in[n]}e_{\min}^2\frac{|\mathcal N^-(j)|}{(\Delta^+)^2},\max_{\substack{i\in[n]\\|\mathcal N^+(i)|>0}}\frac{e_{\min}^2}{|\mathcal N^+(i)|}\right\}
= \max\left\{e^2_{\min} \frac{\Delta^-}{{\Delta^+}^2}, \frac{e^2_{\min}}{\delta^+}\right\}.
\end{align*} 
Similarly, we have
\begin{align*}
\max\left\{ e^2_{\min} \frac{\Delta^+}{{\Delta^-}^2}, \frac{e^2_{\min}}{\delta^-}\right\} \leq\norm*{M^-}^2 \leq e_{\max}^2 \frac{\Delta^+}{\delta^-}.    
\end{align*}
Therefore, the norm of $A^{(k)}_{\operatorname{mean}}(E)$ satisfies
\begin{align*}
1+e^2_{\min}\max\left\{\frac{\Delta^-}{{\Delta^+}^2}, \frac{1}{\delta^+}, \frac{\Delta^+}{{\Delta^-}^2}, \frac{1}{\delta^-}\right\}\leq\norm*{A^{(k)}_{\operatorname{mean}}(E)}^2 \leq 1 + e^2_{\max}\left(\frac{\Delta^+}{\delta^-} + \frac{\Delta^-}{\delta^+}\right),
\end{align*}
which is independent of the feature dimension, while its dependence on the graph is expressed through the degree quantities and edge weights.

Analogously,
\begin{align*}
{A}^{(k)}_{\sum}(E) =  P\left(\begin{pmatrix}I_n \\ \widehat{M}^+ \\ \widehat{M}^-\end{pmatrix} \otimes I_{d_{k-1}}\right)
\end{align*} 
with $\widehat{M}^+_{ij} = e_{j,i}$ and $\widehat{M}^-_{ij} = e_{i,j}.$ Hence, $\widehat{M}^{-} = (\widehat{M}^{+})^\top$ and both are given by the weighted adjacency matrix.
The claim follows by analogous calculations. 
\end{proof}

Further, the different aggregation functions differ in their Lipschitz constants with respect to the weighted adjacency matrix.

\begin{proposition}[Lipschitz constants of the aggregation map]\label{prop:lipschitz-aggregation-map}
For $\operatorname{sum}$ aggregation, the map $E\mapsto A^{(k)}_{\sum}(E)$ is globally Lipschitz continuous and satisfies
\begin{align*}
\norm*{A^{(k)}_{\sum}(E)-A^{(k)}_{\sum}(\overline E)}^2 \leq 2\norm*{E-\overline E}^2.
\end{align*}
For $\operatorname{mean}$ aggregation, assume that $E$ and $\overline E$ have the same nonempty edge support, i.e.
\begin{align*}
\1_{\{e_{ij}>0\}}=\1_{\{\overline e_{ij}>0\}} \quad \text{for all } i,j\in[n].
\end{align*}
Then the map $E\mapsto A^{(k)}_{\operatorname{mean}}(E)$ is Lipschitz continuous on this fixed-topology region and satisfies
\begin{align*}
\norm*{A^{(k)}_{\operatorname{mean}}(E)-A^{(k)}_{\operatorname{mean}}(\overline E)}^2 \leq \left(\frac{1}{(\delta^+)^2}+\frac{1}{(\delta^-)^2} \right)\norm*{E-\overline E}^2.
\end{align*}
\end{proposition}

\begin{proof}
As in the proof of Proposition~\ref{prop:operatornorm-aggregate}, writing for sum aggregation
\begin{align*}
A^{(k)}_{\sum}(E) = P \left( \begin{pmatrix} I_n\\ \widehat{M}^+\\ \widehat{M}^- \end{pmatrix} \otimes I_{d_{k-1}}\right)
\end{align*}
immediately proves the claim.

As for mean aggregation the neighbor sets are fixed, we have 
\begin{align*}
\norm*{A^{(k)}_{\operatorname{mean}}(E)-A^{(k)}_{\operatorname{mean}}(\overline E)}^2 &\leq \norm*{M^+ - \overline M^+}^2 + \norm*{M^- - \overline M^-}^2 \\ 
&\leq \left(\frac{1}{(\delta^+)^2}+\frac{1}{(\delta^-)^2} \right)\norm*{E-\overline E}^2.
\end{align*} 
\end{proof}

\subsubsection{Parameter Lipschitz bounds for the MPNN}
\label{app:mpnn-parameter-bounds}

\noindent\textbf{Layer-wise Lipschitz estimates}: For a fixed weighted adjacency matrix $E$, we first establish layer-wise Lipschitz bounds for the MPNN.

\begin{lemma}[Lipschitz bound for a single layer]
Under Assumption~\ref{assum:bounds}, the Jacobians of a single layer of the MPNN are bounded and Lipschitz continuous for inputs satisfying $\norm*{x}^2\leq s_0nd_{k-1}$. Defining
\begin{align*} 
{{}_xL_\varphi^{(k)}}^2 :=& 2 \norm*{A^{(k)}(E)}^2\left(s_1 + 3s_0s_2\norm*{A^{(k)}(E)}^2 \beta_k^2 nd_{k-1} \right), \\   
{{}_xL_x^{(k)}}^2 :=& 6 s_2\norm*{A^{(k)}(E)}^4 \beta_k^4, \\
{{}_\varphi L_x^{(k)}}^2 :=& 2 \norm*{A^{(k)}(E)}^2 \left(3s_2 \beta_k^2n\left(s_0\norm*{A^{(k)}(E)}^2 d_{k-1}+1\right)+s_1 \right),\\
{{}_\varphi L_\varphi^{(k)}}^2 :=& 6s_0s_2 \norm*{A^{(k)}(E)}^2 n^2d_{k-1}\left(s_0\norm*{A^{(k)}(E)}^2 d_{k-1}+1\right),\\
{B^{(k)}_\varphi}^2 :=& s_1 n\left(s_0\norm*{A^{(k)}(E)}^2 d_{k-1}+1\right),\\
{B^{(k)}_x}^2 :=& s_1 \norm*{A^{(k)}(E)}^2 \beta_k^2
\end{align*}
we have
\begin{align*}
\norm*{D_xT^{(k)}(\varphi,x,E) - D_xT^{(k)}(\bar\varphi,\bar x,E)}^2 
\leq& \,{{}_xL^{(k)}_\varphi}^2 \norm*{\varphi-\bar\varphi}^2 + {{}_xL^{(k)}_x}^2\norm*{x-\bar x}^2, \\
\norm*{D_\varphi T^{(k)}(\varphi,x,E) - D_\varphi T^{(k)}(\bar\varphi,\bar x,E)}^2 \leq& \,{{}_\varphi L^{(k)}_\varphi}^2 \norm*{\varphi-\bar\varphi}^2 + {{}_\varphi L^{(k)}_x}^2\norm*{x-\bar x}^2,\\
\norm*{D_xT^{(k)}(\varphi,x,E)}^2 \leq {B^{(k)}_x}^2, \quad& \norm*{D_\varphi T^{(k)}(\varphi,x,E)}^2 \leq {B^{(k)}_\varphi}^2.
\end{align*}
\label{lemm:lipschitz-single}
\end{lemma}

\begin{proof}
Let $\varphi,\bar\varphi$ be two different parameter vectors and $x,\bar x$ two different inputs for layer $k$. We write $\bar W^{(k)}$ and $\bar b^{(k)}$ for the weights and bias corresponding to $\bar\varphi$. Further, set
\begin{align*}
v^{(k)}:=A^{(k)}(E)x,
\qquad
\bar v^{(k)}:=A^{(k)}(E)\bar x.
\end{align*}

By \eqref{eq:layer-input},
\begin{align*}
\norm*{D_x T^{(k)} (\varphi, x,E) -  D_x T^{(k)} (\bar\varphi, \bar x,E)}^2 
=& \norm*{S^{(k)} D_x U^{(k)} (\varphi,x,E) - \bar S^{(k)} D_xU^{(k)} (\bar\varphi, \bar x,E)}^2  \\
\leq&2 \norm*{D_x U^{(k)} (\bar\varphi, \bar x,E)}^2\norm*{S^{(k)}-\bar S^{(k)}}^2 \\ 
&+ 2\norm*{S^{(k)}}^2 \norm*{D_x U^{(k)} (\varphi,x,E) - D_xU^{(k)} (\bar\varphi, \bar x,E)}^2 \\
=:&2\left((i) + (ii)\right).
\end{align*}
Since $\norm*{\cdot} \leq \norm*{\cdot}_F$ holds for any matrix, we further estimate
\begin{align*}
(i) \leq& \norm*{(I_n \otimes \bar W^{(k)}) A^{(k)}(E)}^2 
s_2 \norm*{U^{(k)}(\varphi,x,E) - U^{(k)}(\bar\varphi, \bar x,E)}^2 \\
\leq& \norm*{\bar W^{(k)}}^2\norm*{A^{(k)}(E)}^2 s_2 
\left(\norm*{\mathds{1}_n \otimes (b^{(k)} - \bar b^{(k)})}+\norm*{\left(I_n \otimes  W^{(k)} \right) v^{(k)} - \left(I_n \otimes  \bar W^{(k)} \right) \bar v^{(k)}}\right)^2 \\
\leq& \norm*{\bar W^{(k)}}^2\norm*{A^{(k)}(E)}^2s_2 3\left(n \norm*{b^{(k)} - \bar b^{(k)}}^2 + \norm*{W^{(k)}-\bar W^{(k)}}^2\norm*{\bar v^{(k)}}^2 + \norm*{W^{(k)}}^2\norm*{v^{(k)}-\bar v^{(k)}}^2\right).
\end{align*}
Since $A^{(k)}(E)$ is fixed in this lemma, we have
\begin{align*}
\norm*{v^{(k)}-\bar v^{(k)}}^2 \leq& \norm*{A^{(k)}(E)}^2\norm*{x-\bar x}^2.
\end{align*}

For (ii), we use
\begin{align*}
(ii) \leq& s_1 \norm*{(I_n \otimes W^{(k)})A^{(k)}(E) - (I_n \otimes \bar W^{(k)}) A^{(k)}(E)}^2 \\
\leq& s_1 \norm*{W^{(k)}-\bar W^{(k)}}^2 \norm*{A^{(k)}(E)}^2 .
\end{align*}
Collecting these estimates yields
\begin{align*}
\biggl\|D&_x T^{(k)} (\varphi, x,E) -  D_x T^{(k)} (\bar\varphi, \bar x,E)\biggr\|^2 \\
\leq& 6\norm*{\bar W^{(k)}}^2\norm*{A^{(k)}(E)}^2s_2 n \norm*{b^{(k)} - \bar b^{(k)}}^2 \\
&+2\norm*{A^{(k)}(E)}^2\left(s_1 +3\norm*{\bar W^{(k)}}^2 s_2\norm*{\bar v^{(k)}}^2\right)\norm*{W^{(k)} -\bar W^{(k)}}^2 \\
&+ 6\norm*{\bar W^{(k)}}^2\norm*{A^{(k)}(E)}^4 s_2\norm*{W^{(k)}}^2 \norm*{x-\bar x}^2 \\
\leq& \max\left\{6\norm*{\bar W^{(k)}}^2\norm*{A^{(k)}(E)}^2s_2 n,
2\norm*{A^{(k)}(E)}^2\left(s_1 +3\norm*{\bar W^{(k)}}^2 s_2\norm*{\bar v^{(k)}}^2\right)\right\} \norm*{\varphi - \bar\varphi}^2\\
&+ 6\norm*{\bar W^{(k)}}^2\norm*{A^{(k)}(E)}^4 s_2\norm*{W^{(k)}}^2 \norm*{x-\bar x}^2 \\
\leq&  {{}_xL_\varphi^{(k)}}^2 \norm*{\varphi - \bar\varphi}^2 + {{}_xL_x^{(k)}}^2 \norm*{x-\bar x}^2.
\end{align*} 

Since $D_\varphi U^{(k)}(\varphi,x,E)$ is independent of $\varphi$, and using the identity
\begin{align*}
\begin{pmatrix}I_{{d_{k}}}\otimes ([v_1^{(k)}\,1])^\top\\\vdots\\I_{{d_{k}}}\otimes ([v_n^{(k)}\,1])^\top\end{pmatrix} -
\begin{pmatrix}I_{{d_{k}}}\otimes ([\bar v_1^{(k)}\,1])^\top\\\vdots\\I_{{d_{k}}}\otimes ([\bar v_n^{(k)}\,1])^\top\end{pmatrix}
= \begin{pmatrix}I_{{d_{k}}}\otimes ([v_1^{(k)}-\bar v_1^{(k)}\,0])^\top\\\vdots\\I_{{d_{k}}}\otimes ([v_n^{(k)}-\bar v_n^{(k)}\,0])^\top\end{pmatrix},
\end{align*}
together with the bounds derived above, we obtain
\begin{align*}
\biggl\|D_\varphi T^{(k)}(\varphi,x,E) -& D_\varphi T^{(k)}(\bar\varphi,\bar x,E)\biggr\|^2 \\
=& \norm*{S^{(k)}D_\varphi U^{(k)}(x,E) - \bar S^{(k)}D_\varphi U^{(k)}(\bar x,E)}^2 \\
\leq& 2\norm*{S^{(k)} - \bar S^{(k)}}^2 \norm*{D_\varphi U^{(k)}(\bar x,E)}^2 + 2\norm*{S^{(k)}}^2  \norm*{v^{(k)}-\bar v^{(k)}}^2 \\
\leq& 6\norm*{\begin{pmatrix}\bar v^{(k)}\\ \mathds{1}_n\end{pmatrix}}^2 s_2 \biggl(n \norm*{b^{(k)} - \bar b^{(k)}}^2 
+ \norm*{W^{(k)}-\bar W^{(k)}}^2\norm*{\bar v^{(k)}}^2 \\
&+ \norm*{W^{(k)}}^2\norm*{A^{(k)}(E)}^2\norm*{x-\bar x}^2\biggr) + 2s_1 \norm*{A^{(k)}(E)}^2\norm*{x-\bar x}^2 \\
\leq& \left(6\norm*{\begin{pmatrix}\bar v^{(k)}\\ \mathds{1}_n\end{pmatrix}}^2 s_2 \norm*{W^{(k)}}^2\norm*{A^{(k)}(E)}^2 + 2s_1 \norm*{A^{(k)}(E)}^2 \right) \norm*{x-\bar x}^2 \\
&+ 6\norm*{\begin{pmatrix}\bar v^{(k)}\\ \mathds{1}_n\end{pmatrix}}^2 s_2 \max\left\{n, \norm*{\bar v^{(k)}}^2\right\} \norm*{\varphi - \bar\varphi}^2\\
\leq& {{}_\varphi L_x^{(k)}}^2\norm*{x-\bar x}^2 + {{}_\varphi L_\varphi^{(k)}}^2 \norm*{\varphi - \bar\varphi}^2.
\end{align*}

Finally, the uniform bounds for the individual Jacobians follow from
\begin{align*}
\norm*{D_x T^{(k)}(\varphi,x,E)}^2 \leq& \norm*{S^{(k)}}^2\norm*{W^{(k)}}^2\norm*{A^{(k)}(E)}^2 
\leq {B^{(k)}_x}^2
\end{align*}
and
\begin{align*}
\norm*{D_\varphi T^{(k)}(\varphi,x,E)}^2 \leq& s_1 \sum_{i=1}^n \norm*{\begin{pmatrix}v_i^{(k)}\\1\end{pmatrix}}^2 
\leq {B^{(k)}_\varphi}^2.
\end{align*}
\end{proof}

\noindent\textbf{Inductive MPNN bounds}: We now extend the single-layer estimates to the entire MPNN by propagating the Lipschitz and boundedness properties.

\begin{proposition}[Inductive Lipschitz bound for the MPNN up to layer $k$]\label{prop:lipschitz-GNN}
Let Assumption~\ref{assum:bounds} hold and fix a graph input $(X,E)$.  Then, for all admissible parameter vectors $\theta^k,\overline{\theta}^k$, the $k$-layer MPNN satisfies
\begin{align*}
\norm*{D_{\theta^k} f^{k}(\theta^k,x^{(0)},E)}\leq B_{\leq k}.
\end{align*}
Moreover, its derivative with respect to the parameters is Lipschitz continuous, i.e.,
\begin{align*}
\norm*{D_{\theta^k}f^{k}(\theta^k,x^{(0)},E)-D_{\theta^k}f^{k}(\overline{\theta}^k,x^{(0)},E)}
\leq L_{\leq k}\norm*{\theta^k-\overline{\theta}^k}.
\end{align*}
The constants are given recursively by
\begin{align*}
B_{\leq k}^2 :=& {B_\varphi^{(k)}}^2 + {B_x^{(k)}}^2 B_{\leq k-1}^2, \\
L_{\leq k}^2 :=& 3\max\left\{\left(B_{\leq k-1}^2 {{}_xL^{(k)}_x}^2+{{}_\varphi L^{(k)}_x}^2\right) B_{\leq k-1}^2 + {B_x^{(k)}}^2 L_{\leq k-1}^2, B_{\leq k-1}^2 {{}_xL^{(k)}_\varphi}^2+ {{}_\varphi L^{(k)}_\varphi}^2\right\},
\end{align*}
where $B_{\leq 1}^2 = {B_\varphi^{(1)}}^2, L_{\leq 1}^2 = {{}_\varphi L_\varphi^{(1)}}^2$ and the constants for $1\leq i\leq k$ are given in Lemma~\ref{lemm:lipschitz-single}.
\end{proposition}

\begin{proof}
The bounds for $k=1$ follow directly from Lemma~\ref{lemm:lipschitz-single}. For $k\ge 2$, we use an induction argument. Hence, we assume that the Jacobian of the MPNN with $k-1$ layers is bounded by $B_{\leq k-1}$ and Lipschitz continuous with constant $L_{\leq k-1}$.

Let
\begin{align*}
x^{(k-1)}:=f^{k-1}(\theta^{k-1},x^{(0)},E),
\qquad
\bar x^{(k-1)}:=f^{k-1}(\bar\theta^{k-1},x^{(0)},E).
\end{align*}
Since $f^{k}(\theta^{k},x^{(0)},E) = T^{(k)}(\varphi^{(k)},x^{(k-1)},E)$,
the chain rule yields
\begin{align*}
D_{\theta^{k}} f^{k}({\theta^{k}},x^{(0)},E)= \left[D_{\varphi}T^{(k)}(\varphi^{(k)},x^{(k-1)},E) \,\,\,D_{x} T^{(k)}(\varphi^{(k)},x^{(k-1)},E) \cdot D_{\theta^{k-1}}f^{k-1}(\theta^{k-1},x^{(0)},E) \right].
\end{align*}

For two parameter vectors $\theta^{k}, \overline{\theta}^{k}$ and consequently two inputs $x^{(k-1)},\overline{x}^{(k-1)}$, we estimate using Cauchy--Schwarz
\begin{align*}
\biggl\|D&_{\theta^{k}} f^{k}({\theta^{k}},x^{(0)},E) - D_{\theta^{k}} f^{k}({\overline{\theta}^{k}},x^{(0)},E)\biggr\|^2 \\
\leq&
\Bigl(\norm*{D_{\varphi} T^{(k)}(\varphi^{(k)},x^{(k-1)},E) - D_{\varphi} T^{(k)}(\overline{\varphi}^{(k)},\overline{x}^{(k-1)},E)} \\
&+ \Bigl\|D_x T^{(k)}(\varphi^{(k)},x^{(k-1)},E) D_{\theta^{k-1}} f^{k-1}({\theta^{k-1}},x^{(0)},E) \\
&\qquad - D_x T^{(k)}(\overline{\varphi}^{(k)},\overline{x}^{(k-1)},E) D_{\theta^{k-1}} f^{k-1}({\overline{\theta}^{k-1}},x^{(0)},E)\Bigr\|\Bigr)^2 \\
\leq& 3 \Big(\norm*{D_x T^{(k)}(\varphi^{(k)},x^{(k-1)},E)}^2 \norm*{D_{\theta^{k-1}} f^{k-1}({\theta^{k-1}},x^{(0)},E) - D_{\theta^{k-1}} f^{k-1}({\overline{\theta}^{k-1}},x^{(0)},E)}^2 \\
&+ \norm*{D_{\theta^{k-1}} f^{k-1}({\overline{\theta}^{k-1}},x^{(0)},E)}^2 \norm*{D_x T^{(k)}(\varphi^{(k)},x^{(k-1)},E) - D_x T^{(k)}(\overline{\varphi}^{(k)},\overline{x}^{(k-1)},E)}^2 \\
&+ \norm*{D_{\varphi} T^{(k)}(\varphi^{(k)},x^{(k-1)},E) - D_{\varphi} T^{(k)}(\overline{\varphi}^{(k)},\overline{x}^{(k-1)},E)}^2
\Big) \\
&=: 3\big((i) + (ii) + (iii)\big).    
\end{align*}

As $D_{\theta^{k-1}} f^{k-1}$ is Lipschitz continuous with constant $L_{\leq k-1}$ and $D_x T^{(k)}$ is bounded, we have
\begin{align*}
(i) \leq {B^{(k)}_x}^2 L_{\leq k-1}^2\norm*{\theta^{k-1}-\bar\theta^{k-1}}^2.
\end{align*}

By Lemma~\ref{lemm:lipschitz-single}, the map $D_xT^{(k)}(\varphi^{(k)},x^{(k-1)},E)$ is Lipschitz with constants ${}_xL^{(k)}_\varphi$ and ${}_xL^{(k)}_x$. Since $D_{\theta^{k-1}} f^{k-1}$ is bounded by $B_{\leq k-1}$, we obtain 
\begin{align*}
(ii) \leq& B_{\leq k-1}^2\left({{}_xL^{(k)}_\varphi}^2\norm*{\varphi^{(k)}-\bar\varphi^{(k)}}^2 + {{}_xL^{(k)}_x}^2 \norm*{x^{(k-1)}-\bar x^{(k-1)}}^2\right).
\end{align*}

Again by Lemma~\ref{lemm:lipschitz-single}, $D_{\varphi^{(k)}} T^{(k)}(\varphi^{(k)},x^{(k-1)},E)$ is Lipschitz in the first two arguments with constants ${}_\varphi L^{(k)}_\varphi$ and ${}_\varphi L^{(k)}_x$, which gives
\begin{align*}
(iii) \leq& {{}_\varphi L^{(k)}_\varphi}^2\norm*{\varphi^{(k)}-\bar\varphi^{(k)}}^2 + {{}_\varphi L^{(k)}_x}^2\norm*{x^{(k-1)}-\bar x^{(k-1)}}^2. 
\end{align*}

Since the derivative of $f^{k-1}$ is bounded by $B_{\leq k-1}$, Lemma~\ref{lemm:Schrankensatz} implies that $f^{k-1}$ is Lipschitz with constant at most $B_{\leq k-1}$. Hence we conclude
\begin{align*}
\norm*{D_{\theta^{k}} f^{k}({\theta^{k}},x^{(0)},E) - D_{\theta^{k}} f^{k}({\bar\theta^{k}},x^{(0)},E)}^2
\leq& 3  {B_x^{(k)}}^2 L_{\leq k-1}^2\norm*{\theta^{k-1}-\bar\theta^{k-1}}^2 \\
&+ 3 \left(B_{\leq k-1}^2 {{}_xL^{(k)}_\varphi}^2+{{}_\varphi L^{(k)}_\varphi}^2\right) \norm*{\varphi^{(k)}-\bar\varphi^{(k)}}^2 \\
&+ 3 \left(B_{\leq k-1}^2 {{}_xL^{(k)}_x}^2+{{}_\varphi L^{(k)}_x}^2\right) B_{\leq k-1}^2 \norm*{\theta^{k-1}-\bar\theta^{k-1}}^2\\
\leq& \max\biggl\{3 \left(B_{\leq k-1}^2 {{}_xL^{(k)}_x}^2+{{}_\varphi L^{(k)}_x}^2\right) B_{\leq k-1}^2 + 3  {B_x^{(k)}}^2 L_{\leq k-1}^2,\\
&3 \left(B_{\leq k-1}^2 {{}_xL^{(k)}_\varphi}^2+ {{}_\varphi L^{(k)}_\varphi}^2\right)\biggr\} 
\norm*{\theta^{k}-\bar\theta^{k}}^2 \\
=& L_{\leq k}^2 \norm*{\theta^k-\bar\theta^{k}}^2.
\end{align*}
Similarly, using the triangle inequality and the boundedness of $T^{(k)}$,
\begin{align*}
\norm*{D_{\theta^k} f^{k}(\theta^k,x^{(0)},E)}^2 
\leq& \norm*{D_\varphi T^{(k)}(\varphi^{(k)},x^{(k-1)},E)}^2 \\
&+ \norm*{D_xT^{(k)}(\varphi^{(k)},x^{(k-1)},E)}^2\norm*{D_{\theta^{k-1}} f^{k-1}(\theta^{k-1},x^{(0)},E)}^2 \\
\leq& {B_\varphi^{(k)}}^2 + {B_x^{(k)}}^2 B_{\leq k-1}^2 = B_{\leq k}^2.
\end{align*}
\end{proof}

With the above proposition, we are able to prove Theorem~\ref{theo:explicit-bounds}.

\begin{proof}[Proof of Theorem~\ref{theo:explicit-bounds}]
Solving the recursive formula given by Proposition~\ref{prop:lipschitz-GNN} yields
\begin{align*}
B_{\leq k}^2 &= {B_\varphi^{(k)}}^2 + {B_x^{(k)}}^2 B^2_{\leq k-1} \\
&={B_\varphi^{(k)}}^2 + {B_x^{(k)}}^2 \left({B_\varphi^{(k-1)}}^2 + {B_x^{(k-1)}}^2B_{\leq k-2}^2\right) \\
&= \sum_{j=1}^{k} \left(\prod_{l=j+1}^k {B_x^{(l)}}^2\right){B_\varphi^{(j)}}^2.
\end{align*}
Hence, for $q \neq 1$,
\begin{align*}
B_{\leq k}^2 &\leq \sum_{j=1}^k \left(\prod_{l=j+1}^k q\right) \alpha = \alpha \sum_{j=1}^k q^{k-j}= \alpha \frac{q^k-1}{q-1},
\end{align*}
and if $q=1$, then $B_{\leq k}^2 \leq \alpha k $.

With the notation
$$w_k:= 3B_{\leq k-1}^2\left(B_{\leq k-1}^2 {{}_xL^{(k)}_x}^2+{{}_\varphi L^{(k)}_x}^2+{{}_xL^{(k)}_\varphi}^2\right) + 3{{}_\varphi L^{(k)}_\varphi}^2$$
and 
\begin{align*}
{}_xL_x &:= \max \left\{{}_xL_x^{(1)}, \ldots, {}_xL_x^{(k)}\right\}, &{}_\varphi L_x := \max \left\{{}_\varphi L_x^{(1)}, \ldots, {}_\varphi L_x^{(k)}\right\} , \\
{}_xL_\varphi &:= \max \left\{{}_xL_\varphi^{(1)}, \ldots, {}_xL_\varphi^{(k)}\right\}, &{}_\varphi L_\varphi := \max \left\{{}_\varphi L_\varphi^{(1)}, \ldots, {}_\varphi L_\varphi^{(k)}\right\},
\end{align*}
we estimate
\begin{align*}
L_{\leq k}^2 \leq& 3 \left(\left(B_{\leq k-1}^2 {{}_xL^{(k)}_x}^2+{{}_\varphi L^{(k)}_x}^2\right) B_{\leq k-1}^2 + {B_x^{(k)}}^2 L_{\leq k-1}^2 + B_{\leq k-1}^2 {{}_xL^{(k)}_\varphi}^2+ {{}_\varphi L^{(k)}_\varphi}^2 \right) \\
=& w_k + 3{B_x^{(k)}}^2 L_{\leq k-1}^2 \\
\leq& w_k + 3{B_x^{(k)}}^2\left(w_{k-1} + 3{B_x^{(k-1)}}^2 L_{\leq k-2}^2\right) \\
\leq& \sum_{j=1}^k \left(\prod_{l=j+1}^k 3 {B_x^{(l)}}^2\right) w_j.
\end{align*}
For $q\neq 1$, using the bound on $B_{\le m}$ and setting $\widehat{\alpha}:= \tfrac{\alpha}{q-1}$, we obtain
\begin{align*}
L_{\leq k}^2 \leq& \sum_{j=1}^k \left(3 q\right)^{k-j} 3\left(\left(B_{\leq j-1}^2 {}_xL_x^2 + {}_\varphi L_x^2 + {}_xL_\varphi^2\right)B_{\leq j-1}^2 + {}_\varphi L_\varphi^2\right) \\
=& \sum_{m=0}^{k-1} \left(3 q\right)^{k-1-m} 3\left(\left(B_{\leq m}^2 {}_xL_x^2 + {}_\varphi L_x^2 + {}_xL_\varphi^2\right)B_{\leq m}^2 + {}_\varphi L_\varphi^2\right) \\
\leq& \sum_{m=0}^{k-1} \left(3 q\right)^{k-1-m} 3\left(\left(\widehat{\alpha}(q^m-1) {}_xL_x^2 + {}_\varphi L_x^2 + {}_xL_\varphi^2\right)\widehat{\alpha}(q^m-1) + {}_\varphi L_\varphi^2\right) \\
=& \widehat{\alpha}^2 {}_xL_x^2 \sum_{m=0}^{k-1} 3^{k-m} q^{k-1+m} + \widehat{\alpha} \left({}_\varphi L_x^2 + {}_xL_\varphi^2 - 2\widehat{\alpha} {}_xL_x^2\right) \sum_{m=0}^{k-1} 3^{k-m} q^{k-1} \\&+ \left(\widehat{\alpha}^2 {}_xL_x^2 - \widehat{\alpha}\left({}_\varphi L_x^2 + {}_xL_\varphi^2\right)+{}_\varphi L_\varphi^2\right) \sum_{m=0}^{k-1} 3^{k-m} q^{k-1-m}.
\end{align*}

Assume now that $q\neq 1$, $q\neq 3$ and $q\neq \tfrac{1}{3}$. Then
\begin{align*}
L_{\leq k}^2 \leq& \widehat{\alpha}^2 {}_xL_x^2 3^k q^{k-1} \frac{1-\left(\frac{q}{3}\right)^k}{1-\left(\frac{q}{3}\right)} + \widehat{\alpha} \left({}_\varphi L_x^2 + {}_xL_\varphi^2 - 2\widehat{\alpha} {}_xL_x^2\right) q^{k-1} \frac{3}{2}\left(3^k-1\right) \\
&+ \left(\widehat{\alpha}^2 {}_xL_x^2 - \widehat{\alpha}\left({}_\varphi L_x^2 + {}_xL_\varphi^2\right)+{}_\varphi L_\varphi^2\right) 3^k q^{k-1} \frac{1-\left(\frac{1}{3q}\right)^k}{1-\left(\frac{1}{3q}\right)} \\
=& 3^k q^{k-1} \left(C_1 \frac{1-\left(\frac{q}{3}\right)^k}{1-\left(\frac{q}{3}\right)} +  C_2 (1-3^{-k}) + C_3 \frac{1-\left(\frac{1}{3q}\right)^k}{1-\left(\frac{1}{3q}\right)}\right),
\end{align*}
where the constants $C_1, C_2,C_3$ can be chosen independently of $k$ and satisfy $C_i=\mathcal{O}(n^2d^2)$.

The calculations for the cases $q=3$ and $q=\tfrac{1}{3}$ work similarly. The case $q=1$ differs slightly, as we have to use the bound
\begin{align*}
B_{\leq m}^2 \leq \alpha m.
\end{align*}
However, all calculations are straightforward and yield
\begin{align*}
L_{\leq k}^2  = \begin{cases}
\mathcal{O}(n^2d^2q^{2k}) &\text{ if }  q>3, \\
\mathcal{O}(n^2d^2k3^{2k}) &\text{ if }  q=3, \\
\mathcal{O}(n^2d^23^kq^{k}) &\text{ if }  \tfrac{1}{3}<q<3,\\
\mathcal{O}(n^2d^2k) &\text{ if }  q=\tfrac{1}{3},\\
\mathcal{O}(n^2d^2) &\text{ if }  q<\tfrac{1}{3},
\end{cases}    
\end{align*}
where the $\mathcal{O}(\cdot)$ notation is with respect to the depth $k$, the number of nodes $n$ and the maximal feature dimension $d$.
\end{proof}

\noindent\textbf{Sharpness example}: We now consider Example~\ref{ex:lower-bounds} in detail. In particular, we derive lower bounds for the norm of the Jacobian and for its Lipschitz constant.

In the setup of Example~\ref{ex:lower-bounds} with $d_0=d_1=\ldots=d_K=1$, the aggregation matrix is independent of $k$ and given by
\begin{align*}
A(E)=A^{(k)}(E) =\begin{pmatrix}1\\0\\0\end{pmatrix}, \qquad \norm*{A(E)}=1, 
\end{align*}
where empty neighbor aggregates are taken to be zero. Further, by the choice of the weights and biases, the pre-activations
\begin{align*}
z^{(k)}=W^{(k)}A(E)x^{(k-1)}+b^{(k)}
\end{align*}
satisfy
\begin{align*}
z^{(k)}=x^{(k)}=0 \quad\text{for }k=1,\ldots,K-1, \qquad z^{(K)}=x_0, \qquad x^{(K)}=\tanh(x_0).
\end{align*}
It further holds that
\begin{align*}
\sigma'(0)=1,\quad \sigma''(0)=0, \quad\sigma'(x_0)=\frac{2}{3},\quad \sigma''(x_0)=-\frac{4}{3\sqrt{3}}.
\end{align*}
In particular, $s_0=s_1=1$, $s_2=16/27$, and thus
\begin{align*}
q=\norm*{\sigma'}_\infty^2\beta^2\norm*{A(E)}^2 =\beta^2>3.
\end{align*}
Moreover, the derivatives with respect to the biases are given by
\begin{align*}
D_{b^{(j)}}x^{(K)} =\sigma'(x_0)\beta^{K-j} =\frac{2}{3}\beta^{K-j}, \qquad j=1,\ldots,K.
\end{align*}
All derivatives with respect to the weight entries vanish, since the input to every layer is zero at the chosen parameter vector. Hence, the norm of the Jacobian satisfies
\begin{align*}
\norm*{D_{\theta^K}f^K(\theta^K,x^{(0)},E)}^2
&=\sum_{j=1}^K\left(D_{b^{(j)}}x^{(K)}\right)^2\\
&=\sigma'(x_0)^2\sum_{j=1}^K\beta^{2(K-j)}\\
&=\frac{4}{9}\sum_{l=0}^{K-1}q^l\\
&=\frac{4}{9}\frac{q^K-1}{q-1}.
\end{align*}
Since $\alpha=2$, the upper bound from Theorem~\ref{theo:explicit-bounds} yields
\begin{align*}
\frac{4}{9}\frac{q^K-1}{q-1} \leq B_{\leq K}^2 \leq 2\frac{q^K-1}{q-1},
\end{align*}
which shows geometric growth in the worst case.

The second derivatives satisfy
\begin{align*}
H_1:=D^2_{b^{(1)}}x^{(1)}=\sigma''(0)=0.
\end{align*}
For $2\leq k\leq K-1$, the chain rule gives
\begin{align*}
H_k:=D^2_{b^{(1)}}x^{(k)}
&=\sigma''(0)\beta^2 \left(D_{b^{(1)}}x^{(k-1)}\right)^2 +\sigma'(0)\beta H_{k-1}=0.
\end{align*}
As the first derivatives satisfy $D_{b^{(1)}}x^{(k)}=\beta^{k-1}$ for $k=1,\ldots,K-1$, we obtain in the last layer
\begin{align*}
H_K:=D^2_{b^{(1)}}x^{(K)}
&=D_{b^{(1)}}\left(\sigma'(z^{(K)})\beta D_{b^{(1)}}x^{(K-1)}\right)\\
&=\sigma''(x_0)\beta^2\left(D_{b^{(1)}}x^{(K-1)}\right)^2+\sigma'(x_0)\beta H_{K-1}\\
&=\sigma''(x_0)\beta^2\left(\beta^{K-2}\right)^2\\
&=\sigma''(x_0)\beta^{2(K-1)}.
\end{align*}

Thus, a lower bound on the Lipschitz constant of the derivative is given by
\begin{align*}
s_2q^{2(K-1)}=|\sigma''(x_0)|^2\beta^{4(K-1)}=|H_K|^2\leq L^2.
\end{align*}
Together with the upper bound from Theorem~\ref{theo:explicit-bounds}, this shows the claim.

\subsubsection{Input Lipschitz bounds}
\label{app:input-lipschitz-bounds}
We now derive Lipschitz bounds with respect to perturbations of the node features and the weighted adjacency matrix. 
For this, let $\mathcal{D}\subset \mathcal{S}$ be a subset of the state space such that there exist constants $a,c\geq 0$ satisfying
\begin{align*}
\norm*{A^{(i)}(E)} \leq a, \qquad 
\norm*{A^{(i)}(E)-A^{(i)}(\overline E)} \leq c \norm*{E-\overline E}
\end{align*}
for all $i\in\{1,\ldots,k\}$ and all $s=(X,E),\overline{s}=(\overline X,\overline E)\in\mathcal{D}$.

For sum aggregation, such constants exist on domains with uniformly bounded weighted adjacency matrices by Propositions~\ref{prop:operatornorm-aggregate} and~\ref{prop:lipschitz-aggregation-map}. For mean aggregation, the same holds on such domains with fixed edge support.
Further, let $f_{\max}$ be a finite uniform entrywise bound on the initial node features over $\mathcal D$.

\begin{proposition}[Input Lipschitz bound for the MPNN]\label{prop:input-lipschitz-GNN}
Fix parameters satisfying Assumption~\ref{assum:bounds}. Define constants $B^{\operatorname{in}}_{X,0}:=1$ and $B^{\operatorname{in}}_{E,0}:=0$, and, for $j\geq 1$,
\begin{align*}
B^{\operatorname{in}}_{X,j} &:=\sqrt{s_1}\beta_j a B^{\operatorname{in}}_{X,j-1},\\ 
B^{\operatorname{in}}_{E,j} &:=\sqrt{s_1}\beta_j\left(a B^{\operatorname{in}}_{E,j-1}+c \sqrt{s_0nd_{j-1}}\right).
\end{align*}
Then, for all $s=(X,E),\overline{s}=(\overline X,\overline E)\in\mathcal{D}$,
\begin{align*}
\norm*{f^{k}(\theta^k,X,E)-f^{k}(\theta^k,\overline X,\overline E)}
\leq B^{\operatorname{in}}_{X,k}\norm*{X-\overline X}_F
+ B^{\operatorname{in}}_{E,k}\norm*{E-\overline E}.
\end{align*}
In particular,
\begin{align*}
\norm*{f^{k}(\theta^k,X,E)-f^{k}(\theta^k,\overline X,\overline E)}
\leq\sqrt{\left(B^{\operatorname{in}}_{X,k}\right)^2+\left(B^{\operatorname{in}}_{E,k}\right)^2}
\sqrt{\norm*{X-\overline X}_F^2+\norm*{E-\overline E}^2}.
\end{align*}
\end{proposition}

\begin{proof}
We prove the claim by induction over $k$. For $k=0$, we have $f^0(\theta^0,X,E)=X$. Hence,
\begin{align*}
\norm*{f^0(\theta^0,X,E)-f^0(\theta^0,\overline X,\overline E)}
= \norm*{X-\overline X}_F,
\end{align*}
which proves the claim for $B^{\operatorname{in}}_{X,0}$ and $B^{\operatorname{in}}_{E,0}$. We now assume that the claim holds up to layer $j-1$ and obtain with $A^{(j)}:=A^{(j)}(E), \overline A^{(j)}:=A^{(j)}(\overline E)$ that
\begin{align*}
\bigl\|f^j(\theta^j,X,E)&-f^j(\theta^j,\overline X,\overline E)\bigr\| \\
&\leq \sqrt{s_1}\norm*{(I_n\otimes W^{(j)})\left(A^{(j)}x^{(j-1)}-\overline A^{(j)}\overline{x}^{(j-1)}\right)}\\
&\leq \sqrt{s_1}\beta_j \left( \norm*{A^{(j)}}\norm*{x^{(j-1)}-\overline{x}^{(j-1)}} + \norm*{A^{(j)}-\overline A^{(j)}}\norm*{\overline{x}^{(j-1)}}\right) \\
&\leq \sqrt{s_1}\beta_j \left(a \left(B^{\operatorname{in}}_{X,j-1} \norm*{X-\overline X}_F + B^{\operatorname{in}}_{E,j-1} \norm*{E-\overline E}\right) + c \norm*{E-\overline E}\norm*{\overline{x}^{(j-1)}}  \right) \\
&\leq B^{\operatorname{in}}_{X,j} \norm*{X-\overline X}_F + B^{\operatorname{in}}_{E,j} \norm*{E-\overline E}.
\end{align*}
This proves the claim.
\end{proof}

\begin{proof}[Proof of Theorem~\ref{theo:input-lipschitz-GNN}]
Solving the recursive formulas given by Proposition~\ref{prop:input-lipschitz-GNN}, we have
\begin{align*}
B^{\operatorname{in}}_{X,k}=\prod_{i=1}^k \sqrt{s_1}\beta_i a
\leq \left(q_{\leq k}^{\operatorname{in}}\right)^k,
\end{align*}
as $B^{\operatorname{in}}_{X,0}=1$. Further, since
\begin{align*}
B^{\operatorname{in}}_{E,j}\leq q_{\leq k}^{\operatorname{in}}B^{\operatorname{in}}_{E,j-1}+\alpha_{\leq k}^{\operatorname{in}},
\end{align*}
we obtain with the same methodology as in the proof of Theorem~\ref{theo:explicit-bounds} that
\begin{align*}
B^{\operatorname{in}}_{E,k}\leq\alpha_{\leq k}^{\operatorname{in}}S_k(q_{\leq k}^{\operatorname{in}}).
\end{align*}
This proves the claim.
\end{proof}

Before proving Theorem~\ref{theo:input-lipschitz-policy}, we note that the only remaining step is to estimate the Lipschitz constants of the policy heads with respect to the final embedding $x^{(K)}$.

\begin{proof}[Proof of Theorem~\ref{theo:input-lipschitz-policy}]

For the global head, using that the softmax is globally $1/2$-Lipschitz with respect to the Euclidean norm \cite{Lipschitz_bound_softmax} and that $\norm*{\read}^2=n$, we obtain
\begin{align*}
\left|\pi^G_\theta(a_{\mathrm{glob}}\mid s)-\pi^G_\theta(a_{\mathrm{glob}}\mid \overline{s})\right|
&\leq \frac12\norm*{W^G\read(x^{(K)}-\overline{x}^{(K)})}\\
&\leq \frac12\beta_G\sqrt n\norm*{x^{(K)}-\overline{x}^{(K)}}.
\end{align*}
As the masks are fixed on $\mathcal D$, we obtain for the local heads,
\begin{align*}
\sum_{i=1}^n \left|\pi^i_\theta(a_i\mid s)-\pi^i_\theta(a_i\mid \overline{s})\right|
&\leq \frac12\beta_N \sum_{i=1}^n \norm*{P_i(x^{(K)}-\overline{x}^{(K)})}\\
&\leq \frac12\beta_N\sqrt n\norm*{x^{(K)}-\overline{x}^{(K)}}.
\end{align*}
If $\eta(a_{\mathrm{glob}})=0$, the policy only depends on the global head. If $\eta(a_{\mathrm{glob}})=1$, as the probabilities are bounded by one, we obtain
\begin{align*}
\left|\pi_\theta(a\mid s)-\pi_\theta(a\mid \overline{s})\right|
&\leq \left|\pi^G_\theta(a_{\mathrm{glob}}\mid s)-\pi^G_\theta(a_{\mathrm{glob}}\mid \overline{s})\right|\\
&\quad+\eta(a_{\mathrm{glob}}) \sum_{i=1}^n \left|\pi^i_\theta(a_i\mid s)-\pi^i_\theta(a_i\mid \overline{s})\right|\\
&\leq \frac{\sqrt n}{2} \left( \beta_G+\eta(a_{\mathrm{glob}})\beta_N \right) \norm*{x^{(K)}-\overline{x}^{(K)}}.
\end{align*}
Theorem~\ref{theo:input-lipschitz-GNN} now gives the claim.
\end{proof}

\subsubsection{Score-function bounds}
\label{app:score-function-bounds}

\noindent\textbf{Policy-head estimates}: To prove Theorem~\ref{theo:score-explicit}, we first derive layer-wise bounds for the policy heads in the next two lemmas.

\begin{lemma}[Lipschitz bound for the global head]
\label{lemm:lipschitz-graph-head}
The gradient of the logarithmic global head is Lipschitz continuous in both the head parameters and the input, and it is uniformly bounded.
Defining
\begin{align*}
{{}_xL_{\tau}^G}^2 :=& \frac{3}{2}s_0 \beta_G^2 n^3{d_{K}}+ 4n,\\
{{}_{x}L_{x}^G}^2 :=& \frac{3}{2} \beta_G^4 n^2,\\
{{}_{\tau}L_{\tau}^G}^2 :=& \frac{3}{2}s_0 n^2{d_{K}}(s_0n^2{d_{K}}+1),  \\
{{}_{\tau}L_{x}^G}^2 :=& \frac{3}{2} \beta_G^2n (s_0n^2{d_{K}}+1)+4n,\\   
{B_{G}^2}:=& 2 \beta_G^2n, \\
{B_{\tau^G}^2} :=& 2 \left( s_0 n^2d_K+1\right)
\end{align*}
the following inequalities hold for all parameter pairs $\tau^G,\overline{\tau}^G$ satisfying the weight bounds in Assumption~\ref{assum:bounds} and all inputs $x,\overline{x}$ with $\max\{\norm*{x}^2,\norm*{\overline{x}}^2\}\leq s_0nd_K$.
\begin{align*}
\norm*{D_{\tau^G} \log\pi_{\tau}^G (a_{\mathrm{glob}} \mid s) - D_{\tau^G} \log\pi_{\overline{\tau}}^G (a_{\mathrm{glob}} \mid s)}^2
&\leq {{}_{\tau}L_{\tau}^G}^2 \norm*{\tau^G-\overline{\tau}^{G}}^2  + {{}_{\tau}L_{x}^G}^2 \norm*{x-\overline{x}}^2, \\
\norm*{D_{x} \log\pi_{\tau}^G (a_{\mathrm{glob}} \mid s) - D_{x} \log\pi_{\overline{\tau}}^G (a_{\mathrm{glob}} \mid s)}^2
&\leq {{}_{x}L_{\tau}^G}^2 \norm*{\tau^G-\overline{\tau}^{G}}^2  + {{}_{x}L_{x}^G}^2 \norm*{x-\overline{x}}^2, \\
\norm*{D_{x} \log\pi_{\tau}^G (a_{\mathrm{glob}} \mid s)}^2 &\leq {B_{G}^2}, \\
\norm*{D_{\tau^G} \log\pi_{\tau}^G (a_{\mathrm{glob}} \mid s)}^2 &\leq {B_{\tau^G}^2}.
\end{align*}
\end{lemma}

\begin{proof}
Let $\tau^G,\overline{\tau}^{G}$ be two different sets of parameters for the global head, as well as $x,\overline{x}$ two different inputs. By \eqref{eq:graph-head-parameter} we have
\begin{align*}
\bigl\|D_{\tau^G}& \log\pi_{\tau}^G (a_{\mathrm{glob}} \mid s) - D_{\tau^G} \log\pi_{\overline{\tau}}^G (a_{\mathrm{glob}} \mid s)\bigr\|^2 \\
=& \biggl\|\left(e_{a_{\mathrm{glob}}} - \pi_{\tau}^G(\cdot \mid s)\right)^\top (I_{K_G} \otimes [(\read{x})^\top \, 1]) - \left(e_{a_{\mathrm{glob}}} - \pi_{\overline{\tau}}^G(\cdot \mid s)\right)^\top (I_{K_G} \otimes [(\read{\overline{x}})^\top\, 1])\biggr\|^2 \\
\leq& 2\norm*{\pi_{\tau}^G(\cdot \mid s) - \pi_{\overline{\tau}}^G(\cdot \mid s)}^2 \norm*{[(\read\overline{x})^\top\,1]}^2 + 2\norm*{e_{a_{\mathrm{glob}}} - \pi_{\tau}^G(\cdot \mid s)}^2 \norm*{\read(x-\overline{x})}^2.
\end{align*}
For the first term, using the global $1/2$-Lipschitz bound for the softmax, we bound the change in the softmax probabilities as follows.
\begin{align*}
\|\pi^G_\tau(\cdot \mid s)-\pi^G_{\overline{\tau}}(\cdot \mid s)\|^2 
\leq& \frac{1}{4} \norm*{U^G(\tau,x)-U^G(\overline{\tau},\overline{x})}^2\\
\leq& \frac{1}{4}\left(\norm*{W^G-\overline W^G}\norm*{\read}\norm*{\overline{x}} + \norm*{b^G-\overline b^G} + \norm*{W^G}\norm*{\read}\norm*{x-\overline{x}}\right)^2. 
\end{align*}
Combining these estimates with $\norm*{e_{a_{\mathrm{glob}}} - \pi_{\tau}^G(\cdot \mid s)}^2 \leq 2$ yields
\begin{align*}
\bigl\|D_{\tau^G}& \log\pi_{\tau}^G (a_{\mathrm{glob}} \mid s) - D_{\tau^G} \log\pi_{\overline{\tau}}^G (a_{\mathrm{glob}} \mid s)\bigr\|^2 \\
\leq& \frac{3}{2} \norm*{[{(\read \overline{x})}^\top\,1]}^2\biggl(\norm*{W^G-\overline W^G}^2\norm*{\read}^2\norm*{\overline{x}}^2 \\
&+ \norm*{b^G-\overline b^G}^2 + \norm*{W^G}^2\norm*{\read}^2\norm*{x-\overline{x}}^2\biggr) +4 \norm*{\read}^2\norm*{x-\overline{x}}^2 \\
\leq&  \frac{3}{2} \norm*{[{(\read \overline{x})}^\top\,1]}^2 \max\left\{\norm*{\read}^2\norm*{\overline{x}}^2 ,1\right\} \norm*{\tau^G-\overline{\tau}^{G}}^2 \\
&+\left(\frac{3}{2} \norm*{[{(\read \overline{x})}^\top\,1]}^2\norm*{W^G}^2\norm*{\read}^2 + 4 \norm*{\read}^2 \right)\norm*{x-\overline{x}}^2\\
\leq& {{}_{\tau}L_\tau^G}^2 \norm*{\tau^G-\overline{\tau}^{G}}^2 + {{}_{\tau}L_{x}^G}^2 \norm*{x-\overline{x}}^2,
\end{align*}
where we used that
\begin{align*}
\norm*{\read}^2 =&  \norm*{\mathds{1}_n^\top \otimes I_{d_{K}}}^2 =n.
\end{align*}

Similarly, by \eqref{eq:graph-head-input},
\begin{align*}
\bigl\|D_{x}& \log\pi_{\tau}^G (a_{\mathrm{glob}} \mid s) - D_ {x} \log\pi_{\overline{\tau}}^G (a_{\mathrm{glob}} \mid s)\bigr\|^2  \\
=& \norm*{\left(e_{a_{\mathrm{glob}}} - \pi_{\tau}^G(\cdot \mid s)\right)^\top W^G\read - \left(e_{a_{\mathrm{glob}}} - \pi_{\overline{\tau}}^G(\cdot \mid s)\right)^\top \overline W^G\read}^2 \\
\leq& 2 \norm*{\pi_{\tau}^G(\cdot \mid s) - \pi_{\overline{\tau}}^G(\cdot \mid s)}^2\norm*{\overline W^G}^2\norm*{\read}^2 + 4\norm*{W^G-\overline W^G}^2\norm*{\read}^2 \\
\leq& 6\norm*{\overline W^G}^2\norm*{\read}^2 \frac{1}{4}\bigl(\norm*{W^G-\overline W^G}^2\norm*{\read}^2\norm*{\overline{x}}^2 \\
&+ \norm*{b^G-\overline b^G}^2 + \norm*{W^G}^2\norm*{\read}^2\norm*{x-\overline{x}}^2\bigr) + 4\norm*{W^G-\overline W^G}^2\norm*{\read}^2 \\
\leq& \max\bigl\{\frac{3}{2}\norm*{\overline W^G}^2\norm*{\read}^4 \norm*{\overline{x}}^2 + 4 \norm*{\read}^2,  \\
&\quad \frac{3}{2} \norm*{\overline W^G}^2 \norm*{\read}^2\bigr\}\norm*{\tau^G-\overline{\tau}^{G}}^2\\
&+ \frac{3}{2}\norm*{\overline W^G}^2\norm*{\read}^2 \norm*{W^G}^2\norm*{\read}^2\norm*{x-\overline{x}}^2\\
\leq& {{}_xL_{\tau}^G}^2 \norm*{\tau^G-\overline{\tau}^{G}}^2 + {{}_{x}L_{x}^G}^2\norm*{x-\overline{x}}^2.
\end{align*}
Finally, the uniform bound on the input derivative follows from
\begin{align*}
\norm*{D_{x} \log\pi_\tau^G(a_{\mathrm{glob}} \mid s)}^2 \leq 2 \norm*{W^G}^2 \norm*{\read}^2 \leq B_{G}^2
\end{align*}
and
\begin{align*}
\norm*{D_{\tau^G} \log\pi_\tau^G(a_{\mathrm{glob}} \mid s)}^2
\leq 2 \left(\norm*{\read}^2\norm*{x}^2 +1\right)
\leq B_{\tau^G}^2.
\end{align*}
\end{proof}

\begin{lemma}[Lipschitz bound for the local heads]
\label{lemm:lipschitz-node}
The gradient of each logarithmic local head is Lipschitz continuous in both the head parameters and the input, and it is uniformly bounded.
Defining
\begin{align*}
{{}_{x}L_{x}^N}^2 :=& \frac{3}{2} \beta_N^4,\\
{{}_{x}L_{\tau}^N}^2 :=& \frac{3}{2} s_0\beta_N^2{d_{K}} +4, \\
{{}_{\tau}L_{\tau}^N}^2 :=& \frac{3}{2}s_0 {d_{K}} \left( s_0 {d_{K}} +1 \right), \\
{{}_{\tau}L_{x}^N}^2 :=& \frac{3}{2}\beta_N^2 \left( s_0 {d_{K}} +1 \right) + 4, \\
{B_{N}^2} :=& 2 \beta_N^2,\\
{B_{\tau^N}^2} :=& 2 (s_0d_K+1)
\end{align*}
the following inequalities hold for each node $i$, all parameter pairs $\tau^N,\overline{\tau}^N$ satisfying the weight bounds in Assumption~\ref{assum:bounds} and all inputs $x,\overline{x}$ with $\max\{\norm*{P_i x}^2,\norm*{P_i\overline{x}}^2\}\leq s_0d_K$.
\begin{align*}
\norm*{D_{\tau^N} \log\pi_{\tau}^i (a_i \mid s) - D_{\tau^N} \log\pi_{\overline{\tau}}^i (a_i \mid s)}^2
&\leq {{}_{\tau}L_{\tau}^N}^2 \norm*{\tau^N-\overline{\tau}^{N}}^2 + {{}_{\tau}L_{x}^N}^2 \norm*{x-\overline{x}}^2, \\
\norm*{D_x \log\pi_{\tau}^i (a_i \mid s) - D_x \log\pi_{\overline{\tau}}^i (a_i \mid s)}^2
&\leq {{}_{x}L_{\tau}^N}^2 \norm*{\tau^N-\overline{\tau}^{N}}^2 + {{}_{x}L_{x}^N}^2 \norm*{x-\overline{x}}^2, \\
\norm*{D_x \log\pi_{\tau}^i (a_i \mid s)}^2 &\leq B_N^2, \\
\norm*{D_{\tau^N} \log\pi_{\tau}^i (a_i \mid s)}^2 &\leq B_{\tau^N}^2.
\end{align*}
\end{lemma}

\begin{proof}
Without loss of generality, we consider $i=1$. Let $\tau^N,\overline{\tau}^{N}$ be two different sets of parameters for the local head, as well as $x,\overline{x}$ two different inputs. By~\eqref{eq:node-parameter} we have
\begin{align*}
\bigl\|D_{\tau^N} \log\pi_{\tau}^1 (a_1 \mid s) -& D_{\tau^N} \log\pi_{\overline{\tau}}^1 (a_1 \mid s)\bigr\|^2 \\
=& \biggl\|\big(e_{a_1} - \pi_{\tau}^1(\cdot \mid s)\big)^\top \big(I_{K_{\mathrm{node}}} \otimes [x_1^\top \, 1]\big) \\
&- \big(e_{a_1} - \pi_{\overline{\tau}}^1(\cdot \mid s)\big)^\top \big(I_{K_{\mathrm{node}}} \otimes [\overline{x}_1^\top \,1]\big)\biggr\|^2 \\
\leq& 2\norm*{\pi_{\overline{\tau}}^1(\cdot \mid s) - \pi_{\tau}^1(\cdot \mid s)}^2 \norm*{I_{K_{\mathrm{node}}} \otimes [\overline{x}_1^\top\,1]}^2 \\
&+ 2\norm*{e_{a_1} - \pi_{\tau}^1(\cdot \mid s)}^2 \norm*{(I_{K_{\mathrm{node}}} \otimes [x_1^\top\,1]) - (I_{K_{\mathrm{node}}} \otimes [\overline{x}_1^\top\,1])}^2 \\
\leq& 2\norm*{\pi_{\tau}^1(\cdot \mid s) - \pi_{\overline{\tau}}^1(\cdot \mid s)}^2 \norm*{[\overline{x}_1^\top\,1]}^2 \\
&+ 4 \norm*{P_1}^2 \norm*{x-\overline{x}}^2,
\end{align*}
where we used $\norm*{e_{a_1} - \pi_{\tau}^1(\cdot \mid s)}^2 \leq 2$ and the norm identity for Kronecker products.

Using $\norm*{P_1}^2=1$ and the global $1/2$-Lipschitz bound for the softmax, we proceed analogously to the global head case. 
\begin{align*}
\norm*{\pi^1_\tau(\cdot \mid s)-\pi^1_{\overline{\tau}}(\cdot \mid s)}^2
\leq& \frac{1}{4} \norm*{U^{N}(\tau^N,x)-U^{N}(\overline{\tau}^{N},\overline{x})}^2\\
=& \frac{1}{4} \norm*{W^N P_1 x + b^N - \overline W^N P_1 \overline{x} - \overline b^N}^2\\
\leq& \frac{1}{4} \left(\norm*{W^N-\overline W^N}\norm*{\overline{x}_1} + \norm*{b^N-\overline b^N} + \norm*{W^N}\norm*{x-\overline{x}}\right)^2.
\end{align*}
Collecting the above estimates yields
\begin{align*}
\bigl\|D_{\tau^N} \log\pi_{\tau}^1 (a_1 \mid s) -& D_{\tau^N} \log\pi_{\overline{\tau}}^1 (a_1 \mid s)\bigr\|^2 \\
\leq& \frac{3}{2}\norm*{[\overline{x}_1^\top\,1]}^2 \Bigl(\norm*{W^N-\overline W^N}^2\norm*{\overline{x}_1}^2 \\
&+ \norm*{b^N-\overline b^N}^2 + \norm*{W^N}^2\norm*{x-\overline{x}}^2\Bigr) + 4\norm*{x-\overline{x}}^2 \\
\leq& \frac{3}{2}\norm*{[\overline{x}_1^\top\,1]}^2 \max\left\{\norm*{\overline{x}_1}^2,1\right\} \norm*{\tau^N-\overline{\tau}^{N}}^2 \\
&+\left(\frac{3}{2}\norm*{[\overline{x}_1^\top\,1]}^2 \norm*{W^N}^2+4\right)\norm*{x-\overline{x}}^2\\ 
\leq& {{}_{\tau}L_\tau^N}^2 \norm*{\tau^N-\overline{\tau}^{N}}^2 + {{}_{\tau}L_x^N}^2 \norm*{x-\overline{x}}^2,
\end{align*}
where we used that $\norm*{[\overline{x}_1^\top\,1]}^2\leq s_0d_K+1$ and $\max\{\norm*{\overline{x}_1}^2,1\}\leq s_0d_K$, since $s_0d_K\geq 1$.

Similarly, using \eqref{eq:node-input}, we obtain
\begin{align*}
\bigl\|D_x \log\pi_{\tau}^1 (a_1 \mid s) -& D_x \log\pi_{\overline{\tau}}^1 (a_1 \mid s)\bigr\|^2 \\
=& \norm*{\big(e_{a_1} - \pi_{\tau}^1(\cdot \mid s)\big)^\top W^N P_1 - \big(e_{a_1} - \pi_{\overline{\tau}}^1(\cdot \mid s)\big)^\top \overline W^N P_1}^2 \\
\leq& 2\norm*{\pi_{\tau}^1(\cdot \mid s) - \pi_{\overline{\tau}}^1(\cdot \mid s)}^2 \norm*{\overline W^N}^2 + 4\norm*{W^N-\overline W^N}^2 \\
\leq& \frac{3}{2}\norm*{\overline W^N}^2 \Bigl(\norm*{W^N-\overline W^N}^2\norm*{\overline{x}_1}^2 \\
&+ \norm*{b^N-\overline b^N}^2 + \norm*{W^N}^2\norm*{x-\overline{x}}^2\Bigr) + 4\norm*{W^N-\overline W^N}^2 \\
\leq& \max\Bigl\{\frac{3}{2}\norm*{\overline W^N}^2\norm*{\overline{x}_1}^2+4,\,\frac{3}{2}\norm*{\overline W^N}^2\Bigr\}\norm*{\tau^N-\overline{\tau}^{N}}^2\\
&+ \frac{3}{2}\norm*{\overline W^N}^2 \norm*{W^N}^2\norm*{x-\overline{x}}^2 \\
\leq& {{}_xL_{\tau}^N}^2 \norm*{\tau^N-\overline{\tau}^{N}}^2+ {{}_xL_x^N}^2 \norm*{x-\overline{x}}^2.
\end{align*}
Here, we again used $\norm*{e_{a_1}-\pi_{\tau}^1(\cdot\mid s)}^2\leq 2$ and $\norm*{P_1}^2=1$ in the first inequality. The last inequality follows from the weight and input bounds and $s_0d_K\geq 1$.
The boundedness follows from
\begin{align*}
\norm*{D_x \log\pi_{\tau}^1 (a_1 \mid s)}^2 \leq 2\norm*{W^N}^2\norm*{P_1}^2 \leq B_N^2
\end{align*}
and
\begin{align*}
\norm*{D_{\tau^N} \log\pi_{\tau}^1 (a_1 \mid s)}^2 \leq 2\norm*{[x_1^\top\,1]}^2 \leq B_{\tau^N}^2.
\end{align*}
\end{proof}

Having established Lipschitz and boundedness properties for the global and local heads, we now combine these results with the Lipschitz continuity of the MPNN. This yields a Lipschitz bound for the full score function.

\begin{proposition}[Lipschitz bound for the score function]\label{prop:lipschitz-score}
Fix a state-action pair $(s,a)\in\mathcal S\times\mathcal A(s)$ and write
\begin{align*}
\theta = (\theta^K,\tau^G,\tau^N),\qquad \overline{\theta} = (\overline{\theta}^K,\overline{\tau}^G,\overline{\tau}^N).
\end{align*}
Then the logarithmic policy with summation readout satisfies
\begin{align*}
|\log\pi_{\theta}(a\mid s)-\log\pi_{\overline{\theta}}(a\mid s)|
\leq B_g\norm*{\theta-\overline{\theta}},
\end{align*}
where
\begin{align*}
B_g^2:= B_{\tau^G}^2 + \eta(a_{\mathrm{glob}})n^2B_{\tau^N}^2
+2\left(B_G^2+\eta(a_{\mathrm{glob}})n^2B_N^2\right)B_{\leq K}^2.
\end{align*}
Moreover, the corresponding score function is Lipschitz continuous, i.e.,
\begin{align*}
\norm*{g_{(s,a)}(\theta)-g_{(s,a)}(\overline{\theta})}
\leq L_g\norm*{\theta-\overline{\theta}},
\end{align*}
with
\begin{align*}
L_g^2 :=&\max \biggl\{{{}_{\tau}L_{\tau}^G}^2 + 4 B_{\leq K}^2 {{}_{x}L_{\tau}^G}^2, \eta(a_{\mathrm{glob}})n^2\left({{}_{\tau}L_{\tau}^N}^2 +  4 B_{\leq K} ^2{{}_{x}L_{\tau}^N}^2\right), \\
&\quad \quad 4 \left(B_{G}^2 + \eta(a_{\mathrm{glob}})n^2 B_{N}^2\right) L_{\leq K}^2  \\
&\quad \quad + \left({{}_{\tau}L_{x}^G}^2+  \eta(a_{\mathrm{glob}})n^2 {{}_{\tau}L_{x}^N}^2 + 4 B_{\leq K}^2 \left({{}_{x}L_x^G}^2 + \eta(a_{\mathrm{glob}})n^2{{}_{x}L_{x}^N}^2\right) \right)B_{\leq K}^2\biggr\}.
\end{align*}
The occurring constants are given by
\begin{align*}
 {{}_xL_{\tau}^G}^2 :=& \frac{3}{2}s_0 \beta_G^2 n^3{d_{K}}+ 4n,\\
{{}_{x}L_{x}^G}^2 :=& \frac{3}{2} \beta_G^4 n^2,\\
{{}_{\tau}L_{\tau}^G}^2 :=& \frac{3}{2}s_0 n^2{d_{K}}(s_0n^2{d_{K}}+1), \\
{{}_{\tau}L_{x}^G}^2 :=&  \frac{3}{2} \beta_G^2n (s_0n^2{d_{K}}+1)+4n,\\   
{B_{G}^2}:=& 2 \beta_G^2n, \\
{B_{\tau^G}^2} :=& 2 \left( s_0 n^2d_K+1\right), \\
{{}_{x}L_{\tau}^N}^2 :=& \frac{3}{2} s_0\beta_N^2{d_{K}} +4, \\
{{}_{x}L_{x}^N}^2 :=& \frac{3}{2} \beta_N^4,\\
{{}_{\tau}L_{\tau}^N}^2 :=& \frac{3}{2}s_0 {d_{K}} \left( s_0 {d_{K}} +1 \right), \\
{{}_{\tau}L_{x}^N}^2 :=& \frac{3}{2}\beta_N^2 \left( s_0 {d_{K}} +1 \right) + 4, \\ 
{B_{N}^2} :=& 2 \beta_N^2, \\
{B_{\tau^N}^2} :=& 2 (s_0d_K+1).
\end{align*}
\end{proposition}

\begin{proof}
Fix $s=(X,E)$ and set
\begin{align*}
x^{(K)}:=f^K(\theta^K,X,E),
\qquad 
\overline{x}^{(K)}:=f^K(\overline{\theta}^K,X,E).
\end{align*}
Since the score-function difference decomposes as
\begin{align*}
\norm*{g_{(s,a)}(\theta) - g_{(s,a)}(\overline{\theta})}^2 =& \norm*{\nabla_{\tau^G} \log\pi_\theta(a\mid s) - \nabla_{\tau^G} \log\pi_{\overline{\theta}}(a\mid s)}^2 \\
&+ \norm*{\nabla_{\tau^N} \log\pi_\theta(a\mid s) - \nabla_{\tau^N} \log\pi_{\overline{\theta}}(a\mid s)}^2 \\
&+ \norm*{\nabla_{\theta^{K}} \log\pi_\theta(a\mid s) - \nabla_{\theta^{K}} \log\pi_{\overline{\theta}}(a\mid s)}^2,
\end{align*}
we bound each contribution separately.
By Lemma~\ref{lemm:lipschitz-graph-head} we have
\begin{align*}
\norm*{\nabla_{\tau^G} \log\pi_\theta(a\mid s) - \nabla_{\tau^G} \log\pi_{\overline{\theta}}(a\mid s)}^2 =& \norm*{D_{\tau^G} \log\pi_\theta^G(a_{\mathrm{glob}}\mid s) - D_{\tau^G} \log\pi_{\overline{\theta}}^G(a_{\mathrm{glob}}\mid s)}^2 \\
&\leq {{}_{\tau}L_{\tau}^G}^2 \norm*{\tau^G-\overline{\tau}^G}^2 + {{}_{\tau}L_{x}^G}^2 \norm*{x^{(K)}-\overline{x}^{(K)}}^2.
\end{align*}
Similarly, by Lemma~\ref{lemm:lipschitz-node},
\begin{align*}
\norm*{\nabla_{\tau^N} \log\pi_\theta(a\mid s) - \nabla_{\tau^N} \log\pi_{\overline{\theta}}(a\mid s)}^2 =& \eta(a_{\mathrm{glob}}) \norm*{D_{\tau^N} \sum_{i=1}^n \log\pi_\theta^i(a_i\mid s) - D_{\tau^N} \sum_{i=1}^n\log\pi_{\overline{\theta}}^i(a_i\mid s)}^2 \\
\leq& \eta(a_{\mathrm{glob}}) n \sum_{i=1}^n \norm*{D_{\tau^N} \log\pi_\theta^i(a_i\mid s) - D_{\tau^N}\log\pi_{\overline{\theta}}^i(a_i\mid s)}^2 \\
\leq& \eta(a_{\mathrm{glob}}) n^2 \bigl({{}_{\tau}L_{\tau}^N}^2 \norm*{\tau^N-\overline{\tau}^N}^2 + {{}_{\tau}L_{x}^N}^2 \norm*{x^{(K)}-\overline{x}^{(K)}}^2\bigr).
\end{align*}
For the term involving the MPNN parameters, the chain rule gives
\begin{align*}
&\norm*{\nabla_{\theta^{K}} \log\pi_\theta(a\mid s) - \nabla_{\theta^{K}} \log\pi_{\overline{\theta}}(a\mid s)}^2 \\
&\qquad\qquad = \norm*{D_x \log\pi_{\theta} (a \mid s) D_{\theta^{K}} f^K (\theta^{K},X,E) }\\
&\qquad\qquad \,\, \left. - D_x \log\pi_{\overline{\theta}} (a \mid s) D_{\theta^{K}} f^K (\overline{\theta}^{K},X,E)\right\|^2 \\
&\qquad\qquad\leq 2\norm*{D_x \log\pi_{\theta} (a \mid s) - D_x \log\pi_{\overline{\theta}} (a \mid s)}^2\norm*{D_{\theta^{K}} f^K (\overline{\theta}^{K},X,E)}^2 \\
&\qquad\qquad \,\,+ 2\norm*{D_x \log\pi_{\theta} (a \mid s)}^2\norm*{D_{\theta^{K}} f^K (\theta^{K},X,E) - D_{\theta^{K}} f^K (\overline{\theta}^{K},X,E)}^2 \\
&\qquad\qquad = 2\left( (i) + (ii)\right).
\end{align*}

Using the decompositions in \eqref{eq:logpolicy}, \eqref{eq:graph-head-input}, and \eqref{eq:node-input}, we find that
\begin{align*}
\norm*{D_x \log\pi_{\theta} (a \mid s) - D_x \log\pi_{\overline{\theta}} (a \mid s)}^2 
=& \biggl\|D_x \log\pi_{\theta}^G(a_{\mathrm{glob}} \mid s) - D_x \log\pi_{\overline{\theta}}^G(a_{\mathrm{glob}} \mid s) \\
&+ \eta(a_{\mathrm{glob}}) \biggl(\sum_{i=1}^n D_x \log\pi_{\theta}^i (a_i \mid s) - \sum_{i=1}^n D_x \log\pi_{\overline{\theta}}^i (a_i \mid s)\biggr)\biggr\|^2 \\
\leq& 2\left( {{}_{x}L_{\tau}^G}^2 \norm*{\tau^G-\overline{\tau}^G}^2 + {{}_{x}L_x^G}^2 \norm*{x^{(K)}-\overline{x}^{(K)}}^2\right) \\
&+ \eta(a_{\mathrm{glob}}) 2n^2 \left({{}_{x}L_{\tau}^N}^2 \norm*{\tau^N-\overline{\tau}^N}^2 + {{}_{x}L_{x}^N}^2 \norm*{x^{(K)}-\overline{x}^{(K)}}^2 \right).
\end{align*}
This implies
\begin{align*}
(i) \leq& 2 B_{\leq K}^2 \left( {{}_{x}L_{\tau}^G}^2 \norm*{\tau^G-\overline{\tau}^G}^2 + {{}_{x}L_x^G}^2 \norm*{x^{(K)}-\overline{x}^{(K)}}^2 \right.\\
&\left. \qquad\qquad\qquad\quad+ \eta(a_{\mathrm{glob}}) n^2 \left({{}_{x}L_{\tau}^N}^2 \norm*{\tau^N-\overline{\tau}^N}^2 + {{}_{x}L_{x}^N}^2 \norm*{x^{(K)}-\overline{x}^{(K)}}^2 \right)\right).
\end{align*}

By Proposition~\ref{prop:lipschitz-GNN}, the Jacobian of the MPNN satisfies
\begin{align*}
\norm*{D_{\theta^{K}}f^K (\theta^{K},X,E) - D_{\theta^{K}}  f^K (\overline{\theta}^{K},X,E)}^2 
\leq&  L_{\leq K}^2 \norm*{\theta^{K}-\overline{\theta}^{K}}^2.
\end{align*}
Hence, 
\begin{align*}
(ii) \leq 2 \left(B_{G}^2 + \eta(a_{\mathrm{glob}})n^2 B_{N}^2\right) L_{\leq K}^2 \norm*{\theta^{K}-\overline{\theta}^{K}}^2.
\end{align*}
Since $f^K(\theta^K,X,E)$ is Lipschitz with constant $B_{\leq K}$, we have 
\begin{align*}
\norm*{x^{(K)}-\overline{x}^{(K)}}^2 \leq B_{\leq K}^2 \norm*{\theta^K-\overline{\theta}^{K}}^2.
\end{align*}
Collecting all the above yields
\begin{align*}
\bigl\|g_{(s,a)}(\theta) -& g_{(s,a)}(\overline{\theta})\bigr\|^2 \\
\leq& {{}_{\tau}L_{\tau}^G}^2 \norm*{\tau^G-\overline{\tau}^G}^2 + {{}_{\tau}L_{x}^G}^2 \norm*{x^{(K)}-\overline{x}^{(K)}}^2 \\
&+ \eta(a_{\mathrm{glob}}) n^2 \left({{}_{\tau}L_{\tau}^N}^2 \norm*{\tau^N-\overline{\tau}^N}^2 + {{}_{\tau}L_{x}^N}^2 \norm*{x^{(K)}-\overline{x}^{(K)}}^2\right)\\
&+ 4 B_{\leq K}^2 \left( {{}_{x}L_{\tau}^G}^2 \norm*{\tau^G-\overline{\tau}^G}^2 + {{}_{x}L_x^G}^2 \norm*{x^{(K)}-\overline{x}^{(K)}}^2\right) \\
&+ \eta(a_{\mathrm{glob}}) 4 B_{\leq K}^2 n^2 \left({{}_{x}L_{\tau}^N}^2 \norm*{\tau^N-\overline{\tau}^N}^2 + {{}_{x}L_{x}^N}^2 \norm*{x^{(K)}-\overline{x}^{(K)}}^2 \right) \\
&+ 4 \left(B_{G}^2 + \eta(a_{\mathrm{glob}})n^2 B_{N}^2\right) L_{\leq K}^2 \norm*{\theta^{K}-\overline{\theta}^{K}}^2 \\
\leq& \left({{}_{\tau}L_{\tau}^G}^2 + 4 B_{\leq K}^2 {{}_{x}L_{\tau}^G}^2 \right) \norm*{\tau^G-\overline{\tau}^G}^2\\
&+ \eta(a_{\mathrm{glob}})n^2\left({{}_{\tau}L_{\tau}^N}^2 +  4 B_{\leq K} ^2{{}_{x}L_{\tau}^N}^2\right) \norm*{\tau^N-\overline{\tau}^N}^2 \\
&+ 4 \left(B_{G}^2 + \eta(a_{\mathrm{glob}})n^2 B_{N}^2\right) L_{\leq K}^2 \norm*{\theta^{K}-\overline{\theta}^{K}}^2 \\
&+ \left({{}_{\tau}L_{x}^G}^2+ \eta(a_{\mathrm{glob}}) n^2 {{}_{\tau}L_{x}^N}^2 + 4 B_{\leq K}^2 \left({{}_{x}L_x^G}^2 + \eta(a_{\mathrm{glob}})n^2{{}_{x}L_{x}^N}^2\right) \right)B_{\leq K}^2 \norm*{\theta^{K}-\overline{\theta}^{K}}^2 \\
\leq& \max \biggl\{{{}_{\tau}L_{\tau}^G}^2 + 4 B_{\leq K}^2 {{}_{x}L_{\tau}^G}^2, \eta(a_{\mathrm{glob}})n^2\left({{}_{\tau}L_{\tau}^N}^2 +  4 B_{\leq K} ^2{{}_{x}L_{\tau}^N}^2\right), \\
&\quad \qquad 4 \left(B_{G}^2 + \eta(a_{\mathrm{glob}})n^2 B_{N}^2\right) L_{\leq K}^2  \\
&\quad \qquad + \left({{}_{\tau}L_{x}^G}^2+  \eta(a_{\mathrm{glob}})n^2 {{}_{\tau}L_{x}^N}^2 + 4 B_{\leq K}^2 \left({{}_{x}L_x^G}^2 + \eta(a_{\mathrm{glob}})n^2{{}_{x}L_{x}^N}^2\right) \right)B_{\leq K}^2\biggr\} \norm*{\theta-\overline{\theta}}^2 \\
\leq& L_g^2 \norm*{\theta-\overline{\theta}}^2.
\end{align*}

The score function is also bounded, since
\begin{align*}
\norm*{g_{(s,a)}(\theta)}^2 =& \norm*{\nabla_{\tau^G} \log\pi_\theta(a\mid s)}^2 + \norm*{\nabla_{\tau^N} \log\pi_\theta(a\mid s)}^2 + \norm*{\nabla_{\theta^{K}} \log\pi_\theta(a\mid s)}^2    \\
\leq& B_{\tau^G}^2 + \eta(a_{\mathrm{glob}}) n^2 B_{\tau^N}^2 \\
&+ \norm*{D_x \log\pi_{\theta}^G (a_{\mathrm{glob}} \mid s) + \eta(a_{\mathrm{glob}}) \sum_{i=1}^n D_x \log\pi_{\theta}^i (a_i \mid s)}^2 \norm*{D_{\theta^{K}} f^K (\theta^{K},X,E)}^2\\
\leq& B_{\tau^G}^2 + \eta(a_{\mathrm{glob}}) n^2 B_{\tau^N}^2 + 2\left(B_G^2 + \eta(a_{\mathrm{glob}}) n^2 B_N^2\right) B_{\leq K}^2.
\end{align*}
Now Lemma~\ref{lemm:Schrankensatz} implies that $\theta\mapsto \log\pi_\theta(a\mid s)$ is Lipschitz continuous with constant $B_g$.
\end{proof}

We are now in a position to prove Theorem~\ref{theo:score-explicit}.
\begin{proof}[Proof of Theorem~\ref{theo:score-explicit}]
Under the uniform bounds used for the asymptotic estimates, Theorem~\ref{theo:explicit-bounds} yields
\begin{align*}
B_{\leq K}^2
&=\mathcal O\left(ndS_K(q)\right),\\
L_{\leq K}^2
&=\mathcal O\left(n^2d^2G_K(q)\right).
\end{align*}
The constants in Proposition~\ref{prop:lipschitz-score} are of the order
\begin{IEEEeqnarray*}{rClrClrClrCl}
{}_{\tau}L_{\tau}^G{}^2 &=& \mathcal{O}(n^4d^2), &\qquad
{}_{x}L_{\tau}^G{}^2 &=& \mathcal{O}(n^3d), &\qquad
{}_{\tau}L_{x}^G{}^2 &=& \mathcal{O}(n^3d), &\qquad
{}_{x}L_{x}^G{}^2 &=& \mathcal{O}(n^2), \\
{}_{\tau}L_{\tau}^N{}^2 &=& \mathcal{O}(d^2), &
{}_{x}L_{\tau}^N{}^2 &=& \mathcal{O}(d), &
{}_{\tau}L_{x}^N{}^2 &=& \mathcal{O}(d), &
{}_{x}L_{x}^N{}^2 &=& \mathcal{O}(1), \\
B_{\tau^G}^2 &=& \mathcal{O}(n^2d), &
B_{\tau^N}^2 &=& \mathcal{O}(d), &
B_G^2 &=& \mathcal{O}(n), &
B_N^2 &=& \mathcal{O}(1).
\end{IEEEeqnarray*}
The claims follow immediately by substituting these into Proposition~\ref{prop:lipschitz-score} and using that $S_K(q)\geq 1$ as well as $S_K(q)^2=\mathcal{O}(G_K(q))$.
\end{proof}

\subsection{Proofs for Subsection~\ref{subsec:REINFORCE}}
\label{app:proofs-reinforce}

We first note that the dependence of the one-step reward on the next state does not affect the policy-gradient arguments used below. Indeed, for any policy $\pi$, we may define the expected one-step reward
$$\bar r(s,a):=\int_{\mathcal S} r(s,a,s')\,p(\mathrm ds'\mid s,a).$$
Setting $\bar R_t:=\bar r(S_t,A_t)$, we have $\E_{s_0}^\pi[R_t\mid S_t,A_t]=\bar R_t$ almost surely. Hence, boundedness of the rewards, the tower property, and interchange of expectation and summation yield
\begin{align*}
\E_{s_0}^\pi\left[\sum_{t=0}^\infty \gamma^t R_t\right]
&= \sum_{t=0}^\infty \gamma^t \E_{s_0}^\pi[R_t]
= \sum_{t=0}^\infty \gamma^t \E_{s_0}^\pi\left[\E_{s_0}^\pi[R_t\mid S_t,A_t]\right] \\
&= \sum_{t=0}^\infty \gamma^t \E_{s_0}^\pi[\bar R_t]
= \E_{s_0}^\pi\left[\sum_{t=0}^\infty \gamma^t \bar R_t\right].
\end{align*}
Thus, for the expected discounted-return objective, the next-state-dependent reward $r(s,a,s')$ can be replaced by $\bar r(s,a)$.

\subsubsection{Almost sure convergence}
\label{app:almost-sure-convergence}
\begin{proof}[Proof of Theorem~\ref{theo:projected-reinforce-convergence}]
By Definition~\ref{def:box-parameters}, the parameter set $\Theta$ is compact. Moreover, the score function is bounded and Lipschitz continuous on $\Theta$ by the uniform bounds from Subsection~\ref{subsec:Lipschitz}. Hence Lemma~4.2 and Theorem~4.3 in \cite{Zhang} imply that $\nabla J$ is bounded and Lipschitz continuous on $\Theta$.

We verify the assumptions of Theorem~4.1 in \cite{Borowski}. The update in Algorithm~\ref{alg:projected-reinforce} can be written as
$$\theta_{k+1}=\Pi_\Theta\bigl(\theta_k+\alpha_k(h(\theta_k)+e_k)\bigr),$$
where
$$h(\theta):=\nabla J(\theta),\qquad e_k:=\bar g_k-\nabla J(\theta_k).$$
Further, let
$$V(\theta):=C-J(\theta),\qquad C:=\max_{\theta\in\Theta}J(\theta).$$
By compactness of $\Theta$ and continuity of $J$, the constant $C$ is finite.

By Moreau's decomposition theorem, we obtain for every $\theta\in\Theta$ that
\begin{align*}
\bigl\langle \nabla V(\theta),\Pi_{T_\Theta(\theta)}(\nabla J(\theta))\bigr\rangle
&= -\bigl\langle \nabla J(\theta),\Pi_{T_\Theta(\theta)}(\nabla J(\theta))\bigr\rangle \\
&= -\bigl\langle \Pi_{T_\Theta(\theta)}(\nabla J(\theta))
+\Pi_{N_\Theta(\theta)}(\nabla J(\theta)),\Pi_{T_\Theta(\theta)}(\nabla J(\theta))\bigr\rangle \\
&= -\bigl\|\Pi_{T_\Theta(\theta)}(\nabla J(\theta))\bigr\|^2 \leq 0.
\end{align*}
Thus
$$S_\Theta=\left\{\theta\in\Theta:\bigl\langle \nabla V(\theta),\Pi_{T_\Theta(\theta)}(\nabla J(\theta))\bigr\rangle=0\right\}.$$

Assuming for the moment that
$$\sum_{k=0}^\infty \alpha_k e_k$$
converges almost surely, all assumptions of Theorem~4.1 in \cite{Borowski} are satisfied, and the claim follows. It remains to prove this convergence.

Let $(\mathcal F_k)_{k\in\N_0}$ be the natural filtration generated by the information available before iteration $k$. Then $\theta_k$ is $\mathcal F_k$-measurable, while $\bar g_k$ and $e_k$ are $\mathcal F_{k+1}$-measurable. By Theorem~4.3 in \cite{Zhang}, each estimator $\widehat g_{k,i}$ satisfies
\begin{align*}
\E\left[\widehat g_{k,i}\mid\mathcal F_k\right]=\nabla J(\theta_k) \quad \text{and} \quad \norm*{\widehat g_{k,i}}\leq C_g \quad\text{almost surely}
\end{align*}
for some constant $C_g>0$. Hence $\bar g_k$ is an unbiased estimator of $\nabla J(\theta_k)$ and bounded by $C_g$, implying that
$$M_n:=\sum_{k=0}^{n-1}\alpha_k e_k$$
is a martingale with respect to $(\mathcal F_n)_{n\in\N}$. Further, as $\bar g_k$ and $\nabla J$ are bounded, $e_k$ is almost surely bounded by some constant $C_e>0$. Hence
$$\E\left[\norm*{M_n}^2\right] \leq C_e^2 \sum_{k=0}^{n-1}\alpha_k^2 \leq C_e^2\sum_{k=0}^{\infty}\alpha_k^2 < \infty.$$
Thus, the martingale $(M_n)_{n\in\N}$ is bounded in $L^2$ and therefore converges almost surely. This completes the proof.
\end{proof}

\subsubsection{Finite-time bound}
\label{app:finite-time-bound}
\begin{proof}[Proof of Theorem~\ref{theo:finite-time-constrained-stationarity}]
Let $(\mathcal F_k)_{k\in\N_0}$ be the natural filtration generated by the information available before iteration $k$. As in the previous proof, the single-sample estimators are conditionally unbiased.
$$\E\left[\widehat g_{k,i}\mid\mathcal F_k\right]=\nabla J(\theta_k).$$
Moreover, by the definition of $\sigma$ we have
\begin{align*}
\norm*{\widehat g_{k,i}} &\leq \frac{B_g}{1-\gamma}\sum_{t=0}^{T'_{k,i}}\gamma^{t/2}\left|\widetilde R_t^{(i)}\right|\leq \frac{B_gR_{\max}}{(1-\gamma)(1-\sqrt{\gamma})} = \sigma,
\end{align*}
which yields
$$\E\left[\norm*{\widehat g_{k,i}-\nabla J(\theta_k)}^2\mid\mathcal F_k\right]\leq \sigma^2.$$
Since the batch samples are conditionally independent, this gives
\begin{align}
\label{eq:minibatch-variance}
\E\left[\norm*{\bar g_k-\nabla J(\theta_k)}^2\mid\mathcal F_k\right]\leq \frac{\sigma^2}{m_k}.
\end{align}

We apply Corollary~2 of \cite{Ghadimi2016} to the minimization problem
$$\min_{\theta\in\Theta}-J(\theta).$$
In their notation, this corresponds to $X=\Theta$, $h\equiv 0$, and the Euclidean prox-function
$$V(u,\theta)=\frac12\norm*{u-\theta}^2,$$
whose strong convexity modulus equals one. With
$$G(\theta_k,\xi_{k,i}):=-\widehat g_{k,i},\qquad G_k:=\frac1{m_k}\sum_{i=1}^{m_k}G(\theta_k,\xi_{k,i})=-\bar g_k,$$
their projected stochastic gradient step in their RSPG algorithm becomes
$$\theta_{k+1}=\Pi_\Theta(\theta_k-\alpha_kG_k)=\Pi_\Theta(\theta_k+\alpha_k\bar g_k).$$
Thus Algorithm~\ref{alg:projected-reinforce} is covered by their framework. Since $\nabla J$ is Lipschitz continuous with constant $L_J$, Corollary~2 of \cite{Ghadimi2016} yields
\begin{align}
\label{eq:bound-corollary}
\E\left[\alpha_\tau^{-2}\norm*{\theta_{\tau+1}-\theta_\tau}^2\right]
\leq
\frac{\Delta_0+\sigma^2\sum_{k=0}^{N-1}\frac{\alpha_k}{m_k}}{W_N}.
\end{align}

It remains to relate the projected step to the tangential component of the gradient. By the characterization of the Euclidean projection, we obtain for all $\theta \in \Theta$
\begin{align*}
\langle \theta_k + \alpha_k \bar g_{k}- \theta_{k+1},  \theta-\theta_{k+1} \rangle \leq 0
\quad \Leftrightarrow \quad
\theta_k + \alpha_k \bar g_{k}-\theta_{k+1} 
\in N_\Theta(\theta_{k+1}).
\end{align*}
Hence
$$\bar g_k-\frac{1}{\alpha_k}(\theta_{k+1}-\theta_k)\in N_\Theta(\theta_{k+1}).$$
Using Moreau's decomposition for the tangent and normal cones of the closed convex set $\Theta$, we obtain
\begin{align}
\label{eq:tangent-bound-single-k}
\norm*{\Pi_{T_\Theta(\theta_{k+1})}\bigl(\nabla J(\theta_{k+1})\bigr)}
&= \operatorname{dist}\left(\nabla J(\theta_{k+1}),N_\Theta(\theta_{k+1})\right) \notag \\
&\leq \norm*{\nabla J(\theta_{k+1})-\left(\bar g_k-\frac{1}{\alpha_k}(\theta_{k+1}-\theta_k)\right)} \notag \\
&\leq \norm*{\nabla J(\theta_{k+1}) - \nabla J(\theta_k)} + \norm*{\nabla J(\theta_k) - \bar g_k } + \frac{1}{\alpha_k} \norm*{\theta_{k+1} - \theta_k} \notag \\
&\leq \left(L_J+\frac{1}{\alpha_k}\right)\norm*{\theta_{k+1}-\theta_k} + \norm*{\nabla J(\theta_k)-\bar g_k}.
\end{align}
Finally, taking expectations in \eqref{eq:tangent-bound-single-k}, averaging over the independent random index $\tau$, and using \eqref{eq:bound-corollary} and \eqref{eq:minibatch-variance}, gives 
\begin{align*}
\E\left[\norm*{\Pi_{T_\Theta(\theta_{\tau+1})}\bigl(\nabla J(\theta_{\tau+1})\bigr)}^2\right]
&\leq 2(1+L_J\alpha_{\max})^2 \E\left[\alpha_\tau^{-2}\norm*{\theta_{\tau+1}-\theta_\tau}^2\right] \\
&\quad +2\sum_{k=0}^{N-1}\frac{w_k}{W_N} \E\left[\norm*{\bar g_k-\nabla J(\theta_k)}^2\right]\\
&\leq 2(1+L_J\alpha_{\max})^2 \frac{\Delta_0+\sigma^2\sum_{k=0}^{N-1}\frac{\alpha_k}{m_k}}{W_N} \\
&\quad +2\sigma^2\sum_{k=0}^{N-1}\frac{w_k}{W_Nm_k}.
\end{align*}
This proves the claim.
\end{proof}

\end{document}